\documentclass[12pt, reqno]{amsart}
\usepackage{amssymb, mathrsfs}
\usepackage{latexsym}
\usepackage{amsmath}
\usepackage{amsthm}
\usepackage{amsfonts}
\usepackage{dsfont, upgreek}
\usepackage{mathtools}
\usepackage{epsfig}
\usepackage{amscd}
\usepackage{graphicx}
\usepackage{tikz-cd}
\usepackage{appendix}
\usepackage{enumerate}
\usepackage{color}
\usepackage{todonotes}

\usepackage[left=1.4in,right=1.4in,top=1.2in,bottom=1.2in]{geometry}

\usepackage{tikz}
\usetikzlibrary{decorations.markings}
\usetikzlibrary{arrows.meta}

\usepackage{extarrows}

\usepackage[colorlinks=true,linkcolor=blue,citecolor=blue]{hyperref}

\usepackage[all,cmtip]{xy}
\theoremstyle{plain}
\newtheorem{theorem}{Theorem}[section]
\newtheorem{ntheorem}{Theorem}
\newtheorem{prop}[theorem]{Proposition}
\newtheorem{cond}[theorem]{condition}
\newtheorem{lemma}[theorem]{Lemma}
\newtheorem{corollary}[theorem]{Corollary}

\theoremstyle{definition} 
\newtheorem{definition}[theorem]{Definition}

\theoremstyle{remark} 
\newtheorem{remark}[theorem]{Remark}

\numberwithin{equation}{section}
\makeatletter
\let\save@mathaccent\mathaccent
\newcommand*\if@single[3]{%
	\setbox0\hbox{${\mathaccent"0362{#1}}^H$}%
	\setbox2\hbox{${\mathaccent"0362{\kern0pt#1}}^H$}%
	\ifdim\ht0=\ht2 #3\else #2\fi
}
\newcommand*\rel@kern[1]{\kern#1\dimexpr\macc@kerna}
\newcommand*\overbar[1]{\@ifnextchar^{{\wide@bar{#1}{0}}}{\wide@bar{#1}{1}}}
\newcommand*\wide@bar[2]{\if@single{#1}{\wide@bar@{#1}{#2}{1}}{\wide@bar@{#1}{#2}{2}}}
\newcommand*\wide@bar@[3]{%
	\begingroup
	\def\mathaccent##1##2{%
		\let\mathaccent\save@mathaccent
		\if#32 \let\macc@nucleus\first@char \fi
		\setbox\z@\hbox{$\macc@style{\macc@nucleus}_{}$}%
		\setbox\tw@\hbox{$\macc@style{\macc@nucleus}{}_{}$}%
		\dimen@\wd\tw@
		\advance\dimen@-\wd\z@
		\divide\dimen@ 3
		\@tempdima\wd\tw@
		\advance\@tempdima-\scriptspace
		\divide\@tempdima 10
		\advance\dimen@-\@tempdima
		\ifdim\dimen@>\z@ \dimen@0pt\fi
		\rel@kern{0.6}\kern-\dimen@
		\if#31
		\overline{\rel@kern{-0.6}\kern\dimen@\macc@nucleus\rel@kern{0.4}\kern\dimen@}%
		\advance\dimen@0.4\dimexpr\macc@kerna
		\let\final@kern#2%
		\ifdim\dimen@<\z@ \let\final@kern1\fi
		\if\final@kern1 \kern-\dimen@\fi
		\else
		\overline{\rel@kern{-0.6}\kern\dimen@#1}%
		\fi
	}%
	\macc@depth\@ne
	\let\math@bgroup\@empty \let\math@egroup\macc@set@skewchar
	\mathsurround\z@ \frozen@everymath{\mathgroup\macc@group\relax}%
	\macc@set@skewchar\relax
	\let\mathaccentV\macc@nested@a
	\if#31
	\macc@nested@a\relax111{#1}%
	\else
	\def\gobble@till@marker##1\endmarker{}%
	\futurelet\first@char\gobble@till@marker#1\endmarker
	\ifcat\noexpand\first@char A\else
	\def\first@char{}%
	\fi
	\macc@nested@a\relax111{\first@char}%
	\fi
	\endgroup
}
\makeatother

\title{$C^*$-Algebraic Higher Signatures on Non-Witt Space}
\author{Mingyu Liu}
\begin{document}
\begin{abstract}
 Signature plays an important role in geometry and topology. In the space with singularity, Goresky and MacPherson extend the signatures to oriented pseudomanifolds with only even codimensional stratums by using generalized Poincare duality of intersection homology. After that Siegel  extended the signature on Witt spaces. Higson and Xie study the $C^*$- higher signature on Witt space. Non-Witt spaces need a refined intersection homology to hold Poincare duality. Followed by the combinatorial framework developed by Higson and Roe, this paper construct the $C^*$-signature on non Witt space with noncommutative geometric methods. In conical singular case, we compare analytical signature  of smooth stratified non Witt space by Albin, Leichtnam, Mazzeo and Piazza. 
\end{abstract}
\maketitle

\section{Introduction}

Signature $\mathrm{sign}(M)$ of oriented $4k$ manifold $M$ is an invariant of nondegenerate symmetric quadratic forms in the middle cohomology $\mathrm{H}^{2k}(M)$.  Connection with signature and $L$ class by Hirzebruch inspired many research such as index theorem. If manifold is not simply connected and $BG$ is the classify space for the fundamental group $G$, for homology class $[x] \in \mathrm{H}^*(BG, \mathbb{Q})$, the higher signature is defined as  $\operatorname{sign}_{[x]}(M, f)=\left\langle\mathcal{L}(M) \cup f^{*}[x],[M]\right\rangle $ . Famous Novikov conjecture states that higher signature is homotopy invariant. 

However, in pseudomanifold there is no classical signature because of lacking Poincare duality. In the pseudomanifold, singularity prevent the well defined cup product. 
Goresky and MacPherson introducing intersection homology to generalize the Poincare Duality with perversity \cite{G}. Then the signature with cobordism invariance extend to the pseudomanifold with only even codimension stratum.

For a special case of the oriented pseudomanifold called Witt space, Siegel prove the Poincare Duality and define the signature with similar property \cite{S}. Albin, Leichtnam, Mazzeo and Piazza used an analytic approach to study the higher signature index class for Witt space in \cite{AAW}. Roe and Higson in \cite{HR1} , \cite{HR2} and \cite{HR3} develop a framework to connect algebraic surgery and k theory of $C^*-$ algebra. 
Higson and Xie use a combinatorial method to study the $C^*$-algebraic higher signatures of Witt spaces\cite{X}. 

If the pseudomanifold is not a Witt space, there is not any self dual chain in terms of the sheafification of intersection homology. 
However, some important space for example Zucker’s reductive Borel–Serre compactification of a locally symmetric space  in \cite{zucker19822} is generally not a Witt space. It is worth to develop the signature theory on the non Witt space. 

Cheeger in \cite{cheeger1979spectral} develop $\mathrm{L}^2$ cohomolgy theory to study riemannian space with conical singularity. 
Cheeger, Goresky and McPherson find $\mathrm{L}^2$ cohomology is closely connected lower and upper middle perversity intersection homology in \cite{cheeger1982cohomology}. Also in \cite{cheeger1979spectral} he finds the Lagrange condition of Poincare duality of boundary which called Cheeger boundary condition.  Albin, Leichtnam, Mazzeo and Piazza developed a theory to use iterated fibration structure to generalize smooth stratified space. In A pseudomanifold which has self dual mezzoperversities called Cheeger space.  They study the higher signature index class for Cheeger space in \cite{AA}.

In sheaf theoretical intersection homology, Banagl construct a self dual sheaf which is compatible with intersection homology in \cite{B}. The space with self dual sheaf is called $L$ space. Then he construct $\mathrm{L}$ class of stratified non Witt space in \cite{banagl2006class}.

In this paper, I will to construct the $C^*$-algebraic higher signature on the non-Witt space $X$. We actually build a geometrically controlled Poincare complex on $X$. This is not a direct application of other research of non Witt space. Then we follow the framework of Higson and Xie \cite{X} in Witt space. The advantage of this framework is combinatorial and easy to compute K homology class. 

Below is the organization of this paper.  

In section two, we introduce the geometric and algebra preliminaries and notation. Here the  property of Hilbert-Poincare complex and the filtered complex $W^{\bar p}_*(X)$ with perversity $\bar p$ in pseudomanifold are the fundamental knowledge in the later argument.  

In section 3, we study non Witt space $X$ where the Witt condition only fails in the link of conical singularity. If the signature of links is $0$, we construct two equivalent modified chain complex $\widetilde{W}^{\bar{m}}_j(X)$ and $\widetilde{W}^{\bar{n}}_j(X)$ which interpolate between filtered simplicial complex with lower middle perversity and upper middle perversity:
$$W^{\bar{m}}_j(X)\hookrightarrow\widetilde{W}^{\bar{m}}_j(X)\hookrightarrow \widetilde{W}^{\bar{n}}_j(X) \hookrightarrow W^{\bar{n}}_j(X).$$ 

Next step is prove the generalized Poincare duality map $\widetilde{\mathbb{P}}$ \ from  \ $\widetilde{W}_{\bar{m}}^{*}(X)$ to $\widetilde{W}^{ \bar{n}}_{n-*}(X)$ 
$$\widetilde{\mathbb{P}}:=-\widetilde{\cap} [X]:\widetilde{W}_{\bar{m}}^{j}(X)\to\widetilde{W}^{\bar{n}}_{n-j}(X),$$
is a geometrically controlled chain equivalence. Then $X$ is a geometrically controlled Poincare pseudomanifold.  After the $\ell^2$ completion, we can construct the signature on the Hilbert Poincare complex based on $\widetilde{W}^{\bar{m}}_j(X)$. We follow the Higson and Roe's study of signature on Hilbert Poincare complex. Let the geometrically controlled Poincare pseudomanifold without Witt condition be $L$ space. We have following theorem: (See section \ref{sign} and theorem 3.4.1 for the detail)

\begin{ntheorem}\begin{enumerate}
\item The $C^*-$ algebraic higher signatures $\operatorname{sign}_{\Gamma}(X, f) \in K_{n}\left(C_{r}^{*}(\Gamma)\right)$ of L spaces $X$ are invariant under L cobordism.
That is , if $X_1$ and $X_2$ are n dimensional two closed oriented L spaces with continuous maps $f_1 : X_1 \to B\Gamma$ and $f_2 : X_2 \to B\Gamma.$ Suppose $X_1$ and $X_2$ are $\Gamma$-equivariantly L-cobordant, then
$$\operatorname{sign}_{\Gamma}\left(X_{1}, f_{1}\right)=\operatorname{sign}_{\Gamma}\left(X_{2}, f_{2}\right).$$ 

\item $C^*-$ algebraic higher signatures of L spaces are invariant under stratified homotopy equivalences which keep the Lagrange structure. Suppose $X$ and $Y$ are two closed oriented L spaces, and $f:Y \to B\Gamma $ is a continuous map. If $\phi :X\to Y$ is a stratified homotopy equivalence and keep the Lagrange structure, then
$$\operatorname{sign}_{\Gamma}(X, f \circ \phi)=\operatorname{sign}_{\Gamma}(Y, f)$$

\end{enumerate}
\end{ntheorem}

For section 6, we build a self dual chain complex of Non Witt space $X$ when $X$ exists the compatible Lagrange structure on every odd codimensional stratum. We can construct an iterated modified $W^{m}_*(X)[i]$ on this space for every odd codimensional stratum $\chi_{n-2s_i-1}$. Let the final chain be $\widetilde W^{m}_*(X)$. Then we will show the non Witt space $X$ which admit $\widetilde W^{m}_*(X)$ is Poincare pseudomanifold. The construction of signature is basically same with the one of conical case.

For general non Witt space $X$, the condition of Poincare pseudomanifold for Non Witt space $X$ is existence of the compatible Lagrange structure on every odd codimensional stratum. It still can prove $X$ is a Poincare pseudomanifold. Let the final chain be $\widetilde W^{m}_*(X)$. (See section 4.2 for detail)

\begin{ntheorem} If $X$ is the n dimensional oriented pseudomanifold has enough compatible Lagrange structure to construct $\widetilde W^{m}_*(X)$, then Poincare dual map with fundamental class $[X]$ is geometrically controlled chain equivalence for the chain $\widetilde{W}^{ \bar{m}}_{*}(X)$. 
\end{ntheorem}
\subsection*{Acknowledgements}
The author wants to thank my advisor Zhizhang Xie for many stimulating direction over the years. 
\section{Preliminaries}
Some basic technique are given in this section. First is the quick introduction of signature . Section 2 is a description of k theory of $C^*-$ algebra and index. Then we introduce the framework of Hilbert-Poincare complex in section 3.  Poincare duality in the category of geometric module is introduced in section 4 .
Then we introduce the intersection homology in pseudomanifold in section 5.
For developing a self dual chain complex which is geometrically controlled automatically, section 6 discuss a filtered simplical complex $W^{p}_*(X)$.

\subsection{Poincare Duality and Signatures}
 Poincare in \cite{poincar2010papers} founded the framework of algebraic topology. For a compact oriented $n$-manifold $M$, the intersection $\cap$ of $i$-cycle $c_1$ and $j$-cycle $c_2$ is a well defined $(i+j-n)$-cycle whose homology class depends only the homology class of $c_1$ and $c_2$. 
This means intersection define a product of homology :$H_{i}(X) \times H_{j}(X) \to H_{i+j-n}(X).$ 
Poincare Duality $\mathbb{P}$ is isomorphism from $H^k(M)$ to $H_{n-k}(M)$. It guarantees the product by intersection is a nondegenerate bilinear form :
 $$H_{k} (M,\mathbb{Q}) \times H_{n-k}(M,\mathbb{Q}) \stackrel{\cap}\longrightarrow H_{0}(M,\mathbb{Q}) \to \mathbb{Q}.$$
For the dimension of manifold  is $4k$, the bilinear form is symmetric:
 $$H_{2k} (M,\mathbb{Q}) \times H_{2k}(M,\mathbb{Q}) \to \mathbb{Q}.$$
After diagonalization, assume the symmetric bilinear form is $(\eta^{+},\eta^{-})$. Signature of $M$, denoted as $\mathrm{sign}(M)$, is defined to be $\eta^{+}-\eta^{-}$. 

Signature is a homotopy invariant because it only relays on homology.  Given two manifold $M$ and $N$,if there exists a compact manifold $W$ whose boundary is the disjoint union of $M$ and $N$, $\partial W=M \sqcup N $, we say $M$ and $N$ are cobordism. 
Thom in \cite{thom1954quelques} prove that Signature is cobordism invariant. And he proved signature of manifold has these propositions: 
\begin{enumerate}
\item $\mathrm{sign}(M\sqcup M') = \mathrm{sign}(M)+\mathrm{sign}(M')$.
\item For the product of $M$ and $M'$,  $\mathrm{sign}(M\times M)=\mathrm{sign}(M)\cdot \mathrm{sign}(M)$.
\item if M is the oriented boundary of a manifold then we have $\mathrm{sign}(M) = 0.$

\end{enumerate}

We can define the signature as the index of signature operator. Consider the square-integrable de Rham complex on Riemannian manifold $M$ of $4k$, $d: \Omega^{p}(M) \rightarrow \Omega^{p+1}(M)$ is the exterior derivative, $d^{*}$ is the adjoint operator with inner product $ \Omega^{p+1}(M) \rightarrow \Omega^{p}(M)$. 
 We define $D=d+d^*$ to be the signature operator of $M$. It maps $\Omega^{+}(M) \rightarrow \Omega^{-}(M)$. Then the index of $D$ is
$$\mathrm{index}(D)=\operatorname{dim}\left(\operatorname{ker} D\right)-\operatorname{dim}\left(\operatorname{coker} D\right)$$
Because of Hodge theory $\Omega(M)=\operatorname{ker} \triangle \oplus \operatorname{ker} d \oplus \operatorname{ker} d^{*}$,  we know:
$$\mathrm{index}(D)=\mathrm{sign}(M).$$

If the $L$ polynomial is the multiplicative formal power series of $\frac{\sqrt{z}}{\tanh (\sqrt{z})}$ and 
$p_k$ is the Pontrjagin classes of vector bundle $E$ over $M$ $p_{k}(E, \mathbf{Q}) \in H^{4k}(M, \mathbf{Q})$.
The Hirzebruch's signature theorem connect signature of closed and oriented $4n$-manifold $M$ with L class $L(M)= L_{n}\left(p_{1}(M), \ldots, p_{n}(M)\right)$:
$$\operatorname{sign}(M)=\langle L(M),[M]\rangle \in \mathbb{Z}.$$

Then Hirzebruch's work and Riemann-Roch theorem motivated Atiyah and Singer the prove the famous index theorem in \cite{atiyah1968index}, \cite{atiyah1968index2} and \cite{atiyah1973heat}. 

For non simply connected manifold $M$. Suppose $B\Gamma$ is the classifying space for $\Gamma=\pi_1(M)$ and f : $M \to B\Gamma$ be a continuous map. For each cohomology class $[x] \in H^*(B\Gamma;Q)$, the so called higher signature class is defined as : 
$$\left\langle f^{*}(x) \cup L_{i}(M),[M]\right\rangle \in \mathbb{Q}$$
The higher signature class can be seen as a special case of K-theoretical higher index.
\subsection{$C^*-$ algebra and Index}\label{higherindex}
Begin with Connes in \cite{connes1994noncommutative}, Noncommutative geometry contain a large area in mathematics now. K theory plays an key role in studying the topology especially in classify $C^{*}-$ algebra and index theory. 
This section is based on the textbook of \cite{willettyu} , \cite{roe1993coarse} and a survey \cite{xie2019higher}.  

Let X be a proper metric space. A nondegenerate $X$-module is a separable Hilbert space $H$ equipped with a nondegenerate $*$-representation of $C_0(X)$. With help of $X$- module, we can introduce two important $C^{*}-$ algebra : Roe algebra and localization algebra.
\begin{definition}\cite{willettyu}
	Let $H_X$ be a $X$-module and $T$ a bounded linear operator acting on $H_X$. 
	\begin{enumerate}
		\item The propagation of $T$ is a nonnegative real number 
		$$ \textup{sup}\{ d(x, y)\mid (x, y)\in \textup{supp}(T)\},$$
		where $\textup{supp}(T)$ is  the complement (in $X\times X$) of the set of points $(x, y)\in X\times X$ for which there exist $f, g\in C_0(X)$ such that $gTf= 0$ and $f(x)\neq 0$, $g(y) \neq 0$;
		\item $T$ is locally compact if $fT$ and $Tf$ are compact for all $f\in C_0(X)$; 
		\item $T$ is pseudo-local if $[T, f]$ is compact for all $f\in C_0(X)$.  
	\end{enumerate}
\end{definition}
Motivated by local index, Yu in \cite{yu1997localization} introduced localization algebra. 
\begin{definition}\cite{willettyu}
	Let $H_X$ be a standard nondegenerate $X$-module and $\mathcal B(H_X)$ the set of all bounded linear operators on $H_X$.  
	\begin{enumerate}
		\item The Roe algebra $C^\ast(X)$ of $X$ is the $C^\ast$-algebra generated by all locally compact operators in $\mathcal B(H_X)$ with finite propagation.
		\item Localization algebra $C_L^\ast(X)$  is the $C^\ast$-algebra generated by all bounded and uniformly norm-continuous functions $f: [0, \infty) \to C^\ast(X)$   such that 
		$$ \textup{propagation of $f(t) \to 0 $, as $t\to \infty$. } $$ 
		\item $C_{L, 0}^\ast(X)$ is the kernel of the evaluation map 
		$$ \textup{ev} : C_L^\ast(X) \to C^\ast(X),  \quad   \textup{ev} (f) = f(0).$$
		In particular, $C_{L, 0}^\ast(X)$ is an  ideal of $C_L^\ast(X)$.
	\end{enumerate}
\end{definition}

If X is a locally compact metric space X with a proper and
  isometric action of $\Gamma$.  Let $H_X$ be a $X$-module equipped with a covariant unitary representation of $\Gamma$. We call $(H_x,\Gamma,\phi)$ is covariant system if:
  $$\pi(\gamma)(\varphi(f) v)=\varphi\left(f^{\gamma}\right)(\pi(\gamma) v).$$ 
 Here $\phi$ is the representation of $C_0(X)$, $\pi$ is the representation of $\Gamma$ , and $f \in C_{0}(X), \gamma \in \Gamma, v \in H_{X}$ $f^{\gamma}(x)=f\left(\gamma^{-1} x\right)$. 
 We assume the covariant system $(H_x,\Gamma,\phi)$ is admissible in sense of \cite{yu2007characterization}.
we use $\mathbb{C}[X]^{\Gamma}$ to denote the $*$-algebra of all $\Gamma$-invariant locally compact operators with finite
propagations in $B(H_X)$. We define the equivariant Roe algebra $C^*(X)^{\Gamma}$ to be the completion of $\mathbb{C}[X]^{\Gamma}$ in $B(H_x)$. In this situation, $C^*(X)^{\Gamma}$ is $*$-isomorphic to $C_r^*(\Gamma)\otimes K$. 
Kasparov in \cite{kasparov1975topological} introduce K-homology which is abstract elliptic operator to study index theory further.
\begin{definition}\cite{kasparov1975topological}
Let $X$ be a locally compact metric space with a proper and cocompact isometric action of $\Gamma$. $H_X$ is an admissible $(X,\Gamma)$-module, $F$ is $\Gamma$-equivariant and $F\in B(H_X)$. The K-homology groups $K_*^{\Gamma}(X)$ are generated by the following cycles modulo certain equivalence relations: 
\begin{enumerate}
    \item an even cycle for $K_0^{\Gamma}(X)$ is a pair $(H_X,F)$ such that, $F^*F-I$ and $FF^*-I$ are locally compact and $[F,f]=Ff-fF$ is compact for all $f\in C_0(X)$.
    \item an odd cycle for $K_1^{\Gamma}(X)$ is a pair $(H_X,F)$ such that $F^2 -I$ and $F-F^{*}$ are locally compact and $[F,f]$ is compact for all $f\in C_0(X)$.
\end{enumerate}
\end{definition}

Given a short exact sequence of $C^*-$ algebra:
$$0 \rightarrow \mathcal{J} \rightarrow \mathcal{A} \rightarrow \mathcal{A} / \mathcal{J} \rightarrow 0,$$ 
we know a six-term exact sequence in K-theory. The boundary is $\partial_{0}: K_1(\mathcal{A} / \mathcal{J}) \rightarrow k_0(\mathcal{J}).$ Suppose $u$ is an invertible element in $\mathcal{A} / \mathcal{J}$, and $v$ is the inverse of $u$. Let $U$, $V$ be the lifts of $u$ and $v$ in $A$. We define:
$$W=\left(\begin{array}{cc}1 & U \\ 0 & 1\end{array}\right)\left(\begin{array}{cc}1 & 0 \\ -V & 1\end{array}\right)\left(\begin{array}{cc}1 & U \\ 0 & 1\end{array}\right)\left(\begin{array}{cc}0 & -1 \\ 1 & 0\end{array}\right) \ \text{and} \ e_{11}=\left(\begin{array}{ll}1 & 0 \\ 0 & 0\end{array}\right)$$
Then $P=We_{11}W^{-1}-e_{11}$ is an idempotent of $ \mathcal{J}$.
Define the index of $u$ to be $$\partial([u])\vcentcolon=[P]-\left[e_{11}\right] \in K_{0}(\mathcal{J}).$$

Suppose $(H_X,F)$ is a even cycle of $K_0^{\Gamma}(X)$. Choose a $\Gamma-$ invariant locally finite open cover $\{U_i\}$ of $X$ with diameter $(U_i) < c$ for fixed $c$. If $\{\phi_i\}$ is a $\Gamma$-invariant continuous partition of unity subordinate to $\{U_i\}$, define $\mathcal{F}$:
$$\mathcal{F}=\sum_{i} \phi_{i}^{1 / 2} F \phi_{i}^{1 / 2}.$$
$\partial([\mathcal{F}])$ is the higher index of $(H_X,F)$

For the local index, we change $c$ to $1/n$ for the cover $\{U_i\}$. For $t \in[n, n+1]$, define $\mathcal{F}(t)$:
$$\mathcal{F}(t)=\sum_{j}(1-(t-n)) \phi_{n, j}^{1 / 2} \mathcal{F} \phi_{n, j}^{1 / 2}+(t-n) \phi_{n+1, j}^{1 / 2} \mathcal{F} \phi_{n+1, j}^{1 / 2}$$
The local index of $(H_X,F)$ is $\partial([\mathcal{F}(t)]) \in K_{0}\left(C_{L}^{*}(X)^{\Gamma}\right)$. 

\begin{theorem}\cite{yu1997localization}
If a discrete group $\Gamma$ acts properly on a locally compact space $X$, then the local index map is an isomorphism from the K-homology group $K_*\Gamma(X)$ to the $K$-group of the localization algebra $K_*(C_L^*(X)^{\Gamma}).$
\end{theorem}
Let $\Gamma=\pi_{1}(M)$ be a higher index class in $K_n(C_r^*(\Gamma))$ is called the higher signature of M.

\subsection{Signature on the Hilbert-Poincare Complex}\label{hilbertpc}
In 2005, Higson and Roe \cite{HR1} introduce analytic surgery exact sequence. This section is basically from a serious article \cite{HR1} \cite{HR2} \cite{HR3}. Let us introduce the Hilbert-Poincare complex over $C*$-algebra $A$ in order to build the higher signature on any geometry controlled Poincare complex.
\begin{definition} \cite{HR1} \label{hpc}
A n-dimensional Hilbert-Poincare complex $(E,b,T)$ over $C^*-$ algebra $C$ is a complex of finitely generated Hilbert $C$-module with adjointable operator $T:E_{p}\to E_{n-p}$:
$$ E_0\mathop\leftarrow\limits^{b_1}E_{1}\mathop\leftarrow\limits^{b_2}...\mathop\leftarrow\limits^{b_{n-1}} E_{n-1}\mathop\leftarrow\limits^{b_n} E_{n}$$
such that:
\begin{enumerate}
    \item  if $v\in E_p$ ,then $T^*v=(-1)^{(n-p)p}Tv$;
	\item if $v\in E_p$, then $Tb^*v+(-1)^pbTv=0$; 
	\item $T$ introduce the isomorphism of homology of dual complex(quasi-isomorphism here).
\end{enumerate}
$$ ...\mathop\leftarrow\limits^{b^*}E_{n-p}\mathop\leftarrow\limits^{b^*} E_{n-p-1}\mathop\leftarrow\limits^{b^*}...\mathop\leftarrow\limits^{b^*} E_{p+1}\mathop\leftarrow\limits^{b^*} E_{p}\mathop\leftarrow\limits^{b^*} ...$$
\end{definition}
 If $(E,b,T)$ is the $n$-dimensional Hilbert-Poincare complex.  We can define the self adjoint bounded operator by $Sv=i^{p(p-1)+l}Tv$. Let $B=b+b^*$. Then the self adjoint operator $B\pm S$ is invertible . 
\begin{definition} \cite{HR1} \label{signature}
The signature of Hilbert-Poincare complex $(E,b,T)$ is :
\begin{itemize}
\item For odd-dimensional Hilbert-Poincare complex $(E,b,T)$,  the signature in $K_1(C)$ is defined by the invertible operator $$(B+S)(B-S)^{-1} : E_{ev}\to E_{ev},$$ where $E_{e v}=\oplus_{p} E_{2 p}$
\item For even-dimensional Hilbert-Poincare complex $(E,b,T)$, the signature is defined by the positive projection $[P_{-}]-[P_{+}]$ of $B+S$ and $B-S$.
\end{itemize}
\end{definition}
Let $(E, b, T )$ and $(E', b', T')$ be a pair of n-dimensional Hilbert–Poincare complexes. A homotopy equivalence with other is a chain map $A: (E,b)\to(E',b')$ which induces an isomorphism on homology, and for the two chain maps:
$$ATA^{*}, T^{\prime}:\left(E_{n-*}^{\prime}, b^{\prime *}\right) \rightarrow\left(E_{*}^{\prime}, b^{\prime}\right)$$
induce the same map on homology.

Then signatures in $K_n(C)$ of two homotopy equivalent n-dimensional Hilbert–Poincare complexes are equal.
The signature of Hilbert Poincare complex still has the bordism property .

\begin{theorem}\cite[Theorem 7.6]{HR1}
If $(E,b,T,P)$ is an $(n+1)$-dimensional algebraic Hilbert– Poincare  pair then the K-theoretic signature of its boundary $(P_E,P_b,T_0)$ is zero.
\end{theorem}
In the third article \cite{HR3}, Higson and Roe use this framework to connect the algebraic surgery exact sequence and exact sequence of K-theory of $C^*$-algebra extension.

There are two important geometrical realization of Hilbert Poincare complex . 

First is Hodge–de Rham complex of a complete Riemannian manifold $X$. Let $b$ is operator adjoint of differential $d$. After $L^2$-completions :  $$\Omega_{L^{2}}^{0}(X) \stackrel{b}\leftarrow{\cdots} \cdots \stackrel{b}\leftarrow{\Omega}_{L^{2}}^{n-1}(X) \stackrel{b}\leftarrow{\Omega}_{L^{2}}^{n}(X)$$
For the complete Riemannian manifold, the two minimal domain and maximal domain are the same. Hodge operator $T: \Omega_{L^{2}}^{p}(X) \rightarrow \Omega_{L^{2}}^{n-p}(X)$ by $\langle T \alpha, \beta\rangle=\int_{M} \alpha \wedge \bar{\beta}$ provide the $(\Omega_{L^{2}}^{i}(X), b^*)$ is a Hilbert Poincare complex.

Another is $\left(C_{*}^{\ell^{2}}(X), b^{*}\right)$ on boundary geometry simplicial complex. Here $C_{*}^{\ell^{2}}(X)$ is $\ell^{2}$ cochain.  Moreover, the two Hilbert Poincare complex  $\left(\Omega_{L^{2}}^{*}(X), d\right)$ and $\left(C_{*}^{\ell^{2}}(X), b^{*}\right)$ are homotopy equivalent in the sense of Hilbert Poincare complex. Furthermore, the signature of $\left(\Omega_{L^{2}}^{*}(X), d\right)$ or $\left(C_{*}^{\ell^{2}}(X), b^{*}\right)$ defined here is equal  with classical signature of $X$. We will prove the similar result on self dual non Witt space with conical singularity. 
 
In this thesis, we consider the Hilbert space in Hilbert Poincare complex are the analytically controlled X-modules. A linear map between two analytically controlled X-modules $T:H_1\to H_2$ is said to be analytically controlled if $T$ is the norm limit of locally compact and finite propagation bounded operators . Hodge de Rham complex on a complete Riemannian manifold and $(\ell^{2},b)$ chain complex on boundary geometry simplicial complex Analytically Controlled Hilbert Poincare complex. For analytically controlled Poincare complex over $X$, the signature index is in  $K_*(C^*(X))$. 
\subsection{Geometric Control} \label{geocon}
There is a natural way from the geometrically controlled category to the  analytically controlled category. In this section, we consider the $X$ is a connected simplicial complex with a path metrics $d$.

\begin{definition}\cite{HR2}
$X$ is bounded geometry if exist a number N such that each of the vertices of $X$ lies in at most $N$ different simplices of $X$.
\end{definition}
Next we can introduce geometrically controlled $X$-modules and geometrically controlled linear map over simplicial complex.  
\begin{definition}\cite{HR2}
Let X be a proper metric space. A complex vector space V is geometrically controlled over X if it is provided with a basis $B \subset V$  and a function $c:B \to X$ with the following property: for every $R > 0$ there is an $N<\infty$ such that if $S \subset X$ has diameter less than $R$ then $c^{-1}[S]$ has cardinality less than $N$. The function $c$ is the control map for $V$. A geometrically control linear map $T:V\to W$ is :
\begin{enumerate}
    \item $V$ and $W$ geometrically controlled,
    \item the matrix coefficients of $T$ is uniformly bounded,
    \item the propagation of $T$ is finite.
\end{enumerate}
\end{definition}
The geometrically controlled Hilbert Poincare complex $(E,b,T)$ is a complex of n-dimensional geometrically controlled X-modules $E_i$ together with geometrically controlled linear maps $T:E_{i}\to E_{n-i}$ and $b$ such that:
\begin{enumerate}
    \item  if $v\in E_i$ ,then $T^*v=(-1)^{(n-i)i}Tv$;
	\item if $v\in E_i$, then $Tb^*v+(-1)^ibTv=0$; 
	\item $T$ introduce the isomorphism of homology of dual complex.
\end{enumerate}

One important example of geometrically controlled Poincare complex is the simplicial complex of a closed smooth manifold.  Let $C^{l_2}_*(X)$ to be the Hilbert space
of square integrable simplicial $*$-chains on $X$. Then the chain complex:
 $$C_{0}^{\ell^{2}}(X) \stackrel{b^{*}}{\longrightarrow} \cdots \stackrel{b^{*}}{\longrightarrow} C_{n}^{\ell^{2}}(X)$$
 is a Hilbert complex. Moreover, $(C_{*}^{\ell^{2}}(X),b^*)$ is still a Hilbert complex.
 
 Suppose $f\in C_0(X)$ and $c=\sum k_{\sigma}[\sigma]$ is an $p-$ chain. Let $c_{\sigma}$ be the barycenter of $\sigma$.  It is natural to define $C^{l_2}_*(X)$  as an X-module:
$$f\cdot c= \sum f(c_\sigma)k_{\sigma}[\sigma].$$
\begin{definition} {[ Definition 3.13 in \cite{HR2}]}
Let $X$ be a bounded geometry simplicial complex.  X is a geometrically controlled Poincare complex of dimension n if it is provided with an n-dimensional fundamental cycle $[X]$ for which the associated duality chain map $\mathrm{P}$ is a chain equivalence in the geometrically controlled category.
\end{definition}

In \cite{HR2} section 4 Higson and Roe prove the proposition below: 
\begin{prop} \label{geoana}
Let $X$ be an oriented bounded geometry combinatorial manifold with triangulation $T$.  Then the Poincare dual map with fundamental class $[X]$ gives $X$ the structure of a geometrically controlled Poincare complex.
\end{prop}
We will prove a similiar result for bounded geometry combinatorial Non Witt space with triangulation $T$. 
For any geometrically controlled Poincare complex with $\mathrm{P}$, define  $T=\frac{1}{2}\left(\mathrm{P}^{*}+(-1)^{p(n-p)} \mathrm{P}\right)$.   

\begin{prop}\cite[Proposition 4.1]{HR2} 
After $\ell_2$ completion, every geometrically controlled Poincare complex defines an analytically controlled Poincare complex .
\end{prop}\label{geohp}
Bordism invariance of higher signature. First we introduce the Poincare pair of Hilbert Poincare complex. A complemented subcomplex $(PE,Pb)$ of the geometrically controlled $(E,b)$ is a family of complemented geometrically controlled submodules complex which $Pb$ maps $PE_i$ to $PE_{i-1}$ for all $p$. Then we can define complement complex $(P^\perp E,P^\perp b)$. 
\begin{definition}{ \cite{HR1}}\label{ppair}
An $(n+1)$-dimensional algebraic Hilbert–Poincare pair is a complex of finitely generated Hilbert modules together with a family of bounded adjointable operators $T: E_p \to E_{n+1-p}$ and a family of orthogonal projections $P:E_p \to E_p$ such that
\begin{enumerate}
    \item the orthogonal projections P determines a subcomplex of $(E,b)$;
    \item the range of the operator $T_0=T b^{*}+(-1)^{p} b T$ is contained within the range of $P$;
    \item $P^\perp T$ induces an isomorphism from the homology of the complex $(E,b^*)$ to the homology of the complex $(P^\perp E,P^\perp b)$;
    \item $T^{*}=(-1)^{(n+1-p) p} T: E_{p} \rightarrow E_{n+1-p}$.
\end{enumerate}
\end{definition}

The geometrically controlled Poincare complex $\left(P E, P b, T_{0}\right)$ is defined as the boundary of the geometrically controlled Poincare pair $(E,b,T,P).$ 
Then Higon and Roe in \cite{HR1} theorem 7.6 prove the signature  of algebraic Hilbert Poincare pair in the meaning of \ref{signature} is bordism invariance. In other words
\begin{equation}\label{bordism1}
    \text{signature of } (PE,Pb,T_0)=0
\end{equation}

\subsection{Poincare Duality in Geometric Modules} \label{geom}
With help of geometric algebra, it is easy to prove Poincare dual map is chain equivalence in the category of geometrically control.

In the following sections, we assume the $X$ is a combinatorial manifold simplicial complex.  $X'$ is the first barycentric subdivision. For the simplex $\sigma$ . the Dual cell of $\sigma $ is notated as $D(\sigma)$. The star of $\sigma$ is $st(\sigma)$,the link of $\sigma$ is $lk(\sigma)$ .

In order to explore the further structure over geometrically controlled Poincare complex of $X$, Roe in \cite{R} introduce geometric $\mathbf{R}-X$ module. Geometric module play an important role in the controlled topology. Later we will see the filtered chain complex $W^{p}_*(X)$ as the chain complex of geometric module. 

\begin{definition} \cite{R}\label{geomodule}
The geometric $\mathbf{R}$- module M over simplicial complex $X$ is a list of $M_\sigma$ of $\mathbf{R}$- module parameterized by the faces $\sigma$ of $X$. The geometric morphism $\phi: M\to N$ is a list of $\phi_{\tau,\sigma}$ $\mathbf{R}$-linear map from $M_\sigma$ to $N_\tau$. If $\tau$ is not the face of $\sigma$, we have $\phi_{\tau,\sigma}=0$.
\end{definition}
The diagonal part $\hat{\phi}$ of geometrical morphism $\phi$ is defined as
\[
  \hat{\phi}_{\sigma,\tau}=\begin{cases}
               \phi_{\sigma,\tau} \quad  if  \ \sigma= \tau \\
               0 \ others \\               
            \end{cases}
\]
In fact we only need to consider the diagonal part of geometric morphism due to the global-local principle.
\begin{theorem}\cite{R}\label{landg}
A chain map between finite chain complex of geometric $(R,X)$-modules is a chain equivalence if and only if the induced map on the diagonal part is chain equivalence.  
\end{theorem} 

Next we consider the $\mathbf{R}$ module is the $\mathbb{C}$ vector space.
The first key example of geometric module is that we see $i$-cochain complex $C^i(X)$ with coefficient $\mathbb{C}$ as the geometric module. This is by giving each $i$-simplex $\sigma$ a free generator in dimension i and zero boundary maps. Another is the geometric $(\mathbb{C}, X)$- module over $X'$. Here $C_q(X',R)$ assigns to a simplex $\sigma \in K$ the free $\mathbb{C}$ vector space generated by those $q$-simplices of $X'$ whose root is $\sigma$.

Usually we only need to consider the diagonal part of geometrical module because of Global-Local principle. For the $C^{\bullet}(X',\mathbb{C})$, the diagonal part over $\sigma$ is spanned by all those simplices of $X'$ which have $\sigma$ as their tip.. For the $C_{\bullet}(X',\mathbb{C})$, the diagonal part is the $(D(\sigma),\partial D(\sigma))$.  

For oriented $n$ -dimensional homology manifold, the cap product introduce the chain equivalence in the category of $(\mathbb{C}, X)$ -modules:
$$C^{\bullet}(X',\mathbb{C})\to C_{n-\bullet}(X',\mathbb{C}) $$ 
The key point of geometric Poincare duality is k-fold suspension isomorphism:
$$H_{r}(X, X \ominus \hat{\sigma} ; \mathbb{C}) \rightarrow H_{r-k}(D(\sigma, X), \partial D(\sigma, X)).$$

\subsection{Coefficient System}
Sheaf is a appropriate tools to handle the global and local relation. Especially, Deligne observe that intersection homology is very suitable to approach with sheaf language, see more detail in \cite{maxim2019intersection} and\cite{C1}. Coefficient system is the dual category of constructible sheaf  by Schneider in \cite{SS} and \cite{sss}. It is helpful to define Verdier duality of Building. Coefficient system is also called cellular cosheaf by Curry in \cite{C1}. 

A sheaf is a presheaf with Gluing property.  For detail information of sheaf, see \cite{bredon2012sheaf}. Here we require the module should be complex vector space. 

In a simplicial complex $X$, the interior of $\sigma$ is $\mathring\sigma$, then for every $\sigma$, the open star $\mathring{St}(\sigma)$ of $\sigma$ is open covering, and we have $\sigma<\tau$ then $\mathring{St}(\tau) \subset \mathring{St}(\sigma)$,  then we can construct the constant sheaf $\mathbb{C}_X$ for every open set on simplicial complex by this cover.  For a given sheaf, we can define the sheaf cohomology $\check{\mathrm{H_i}}(X, \mathcal{F})$ via \v{C}ech cohomology. We have $\check{\mathrm{H_i}}(X, \mathbb{C}_X)=H_i(X,\mathbb{C})$. 

Local constant sheaf ${\mathcal{F}}$ is the sheaf for each x in X, there is an open neighborhood U of x such that ${\mathcal{F}}|_{U}$ is a constant sheaf on U. Local constant sheaf are not preserved by the functors $\mathcal{R}f^*$, $\mathcal{R}f!$, $f^*$, $f!$ in general. The smallest closure of constant sheaf is the set of constructible sheaf.  

A sheaf is $\mathbb{S}$-constructible about the stratification $\mathbb{S}$ if it is locally constant on the stratum . It is called cohomologically $\mathbb{S}$-constructible if it is cohomology locally constant for all $i$.  Constructible sheaf on the simplicial complex is also named cellular sheaf. 
\begin{definition} \cite{C1}
The cellular sheaf $\mathcal{V}=(V_F)_F$ is a list of complex vector spaces on $X$,
For each face $F\subset X$, there is a  $\mathbb{C}$-vector spaces $V_F$  , 
and there is the linear maps $r^{F_2}_{F_1} :V_{F_2}\to V_{F_1}$ for each pair of facets $F_2\subset F_1$(restrication map not extension), these maps hold for $r^F_F =id$ and $r^{F_1}_{F_3} =r^{F_2}_{F_3}\circ r^{F_2}_{F_1}$  for any $F_2\subset\overline{F_1}$ and $F_3\subset \overline{F_2}$.
\end{definition}
Let $I$ is the injective resolution. For cellular sheaf, the injective sheaf is very easy to construct. For $f: X\to Y$, the right derived functor $\mathcal{R}f_*$ of $f_*$ is $f_*\circ I$. The advantage of cellular sheaf is the injective sheaf of that is simple. Let $V$ be a vector space,  $${[\sigma]}^V(\tau)=\begin{cases} 
V \quad \tau\leq\sigma\\
0 \quad others\\
\end{cases}
$$
here ${[\sigma]}^V$ is named elementary injective. Every injective cellular sheaf is isomorphic to $\mathop\oplus\limits_{\sigma\in X}{[\sigma]}^{V_\sigma}$.

A building $X$ of $G$ carries a natural $G$-action which is isometric and respects the partition into facets.

With help of coefficient system, it is easy to connect the other research when intersection homology are defined on sheaf. Furthermore, Xie and Higson give a general method to construct $\mathrm{K}-$ homology class on self dual constructible sheaf of any simplicial complex . 

For convenience, we need the coefficient system on the simplicial setting for building K-homology class for equivariant case. When we consider higher signature, it requires to prove the Poincare duality map is chain equivalence under $C_*^r(\Gamma)$ module.  

In general, the map of cosheaf is extension map instead of restriction map. The Poincare-Verdier duality exchange the coefficient system and Constructible sheaves. 

Let $X$ be a simplicial complex and $Sh(\mathbb{C}_X)$ is the abelian category of sheaves of $\mathbb{C}$-vector space $D^b(X)$ is the bounded derived category. Let $D^b(Cons(X))$ be the bounded derived category of constructible sheaves on $X$. Actually, the coefficient system can use to express the dualizing complex $\omega_X$. For any facet $F$ in $X$, the star of facet $St(F)$ is unions of all facet .  The morphism $T: D^b(Cons(X))\to D^b(X)$is equivalence of category.

For any $x\in X$, let $F(x)$ be the unique facet containing $x$, define $st(x)=st(F(x))$. Given any sheaf $S$, then $S_F=S(st(F))$ is the weakly constructible sheaf. There is a equivalent functor from coefficient system to constructible sheaf. 
\begin{definition} \cite{SS} \label{coeff}
The coefficient system $\mathcal{V}=(V_F)_F$ is a list of complex vector spaces on $X$,
For each face $F\subset X$, there is a  $\mathbb{C}$-vector spaces $V_F$  , 
and there is the linear maps $r^{F_1}_{F_2} :V_{F_1}\to V_{F_2}$ for each pair of faces $F_2\subset F_1$, these maps hold for $r^F_F =id$ and $r^{F_1}_{F_3} =r^{F_2}_{F_3}\circ r^{F_2}_{F_1}$  for any $F_2\subset\overline{F_1}$ and $F_3\subset \overline{F_2}$.
\end{definition}

Given a cellular sheaf $\mathcal{F}$ on $X$, It is easy to define the coefficient system $\hat{\mathcal{F}}$ via dual cell, let $\sigma$ is a simplex,  $\bar \sigma$ is the dual simplex about $\sigma$, then define $\hat{\mathcal{F}}(\bar \sigma)=\mathcal{F}(\sigma)$. 
By definition, coefficient system is a geometrical module. 
Let $\mathrm{Coeff}(X)$ be the abelian category of coefficient system in $X$, $D^b(\mathrm{Coeff}(X))$ is the bounded derived category of coefficient system. \\
Given a coefficient system $\mathcal{V}$, the complex of oriented chains is defined as: 
$$C_{c}^{o r}\left(X_{(d)}, \mathcal{V}\right) \stackrel{\partial}{\longrightarrow} \ldots \stackrel{\partial}{\longrightarrow} C_{c}^{o r}\left(X_{(0)}, \mathcal{V}\right)$$
$X_{(q)}$ denotes the set of all $q$-dimensional oriented facet $(F,c)$.  Then we define the homology of coefficient system by the oriented chains. The functor from coefficient system to constructible sheaf is closed related with dualize functor named $\omega(v)$ so we use the same notation.

In \cite{G4}, a oriented cellular pseudomanifold is a convex linear cell complex $K$, purely of some dimension $d$, such that every $d$-1-dimensional cell is a face of exactly two $d$-dimensional cells; together with a choice of orientation of each $d$-dimensional cell such that the induced orientations cancel on every $d-1$-dimensional cell. Then intersection homology chain is coefficient system.

\subsection{Intersection Homology on Pseudomanifold}
 If we consider the space with singularity such as suspension of torus, the intersection of two cycle may be not a cycle. This means the usual intersection product of homology is not well defined in orientated pseudomanifold due to the singularity. Furthermore, the Poincare dual map induced by fundamental class $[X]$ is not isomorphism.
 $$\mathrm{H}^{n-i}(X) \longrightarrow \mathrm{H}_{i}(X)$$
 
 Instead, Goresky and MacPherson introduce the intersection homology $\mathrm{IH}_{i}^{\bar p}$ which is 
$$\mathrm{H}^{n-i}(X) {\longrightarrow} \mathrm{IH}_{i}^{\bar p}(X) {\longrightarrow}\mathrm{H}_{i}(X).$$
to assert the generalized Poincare duality for complementary  perversity.
 The intersection homology has many good properties such as: topological invariant, stratified homotopy equivalent, invariant under normalization , and independent in the stratification. 
 
Furthermore,  if the oriented pseudomanifold satisfies the Witt condition called Witt space then Poincare duality holds for lower middle perversity intersection homology. So we can define the signature on oriented Witt space.

In this article, we consider the question in the piecewise linear (\textbf{PL}) category. I will use the basic knowledge and notation in \cite {G} and \cite{friedman_2020}. In that category, the pseudomanifold is defined as below:

The object in this category are polyhedrons and the morphism are piecewise linear homomorphism. In the PL pseudomanifold, it admits the \textbf{PL}-stratification.

When $X$ is the n-dimensional pseudomanifold, that means locally compact space $X$ and closed singular space of $X_{sig}$ such that: $\mathrm{dim}(X_{sig})<n-1$ and $X-X_{sig}$ is oriented n-manifold which is dense in $X$.  For every pseudomanifold exists a  stratification :
\begin{equation}\label{stra}
X=X_n \supset  X_{n-2} \supset X_{n-3} ... \supset X_1\supset X_0.
\end{equation}
for each point $p\in X_i - X_{i-1}$, there is a filtered space 
$V=V_{n} \supset V_{n-1} \supset \cdots \supset V_{i}=p.$
The neighborhood of $p$ in $X_i-X_{i-1}$ is piecewise linear homomorphically to $V\times B_i.$ If $\chi_i =X_i-X_{i-1}\neq \varnothing$ ,  $\chi_i $ is an open manifold, we call $i$ stratum of the stratification.
satisfying:

Assume $Z$ is the pseudomanifold of dimension $(n-1)$ and a closed subspace of $X$. If $X-(\sum X \cup Z)$ is an n-dimensional oriented manifold which is dense in $X$ and $Z$ is collared in X, we say this n-dimensional pseudomanifold with boundary is a pair of pseudomanifolds $(X, Z)$. The stratification of the pair
$(X, Z)$ satisfies the filtration of $Z$ given by $Z_{j-1}=X_{j}\cap Z$ stratifies $Z$ ,
the filtration of $X-Z$ given by $X_{j}-Z_{j-1}$ stratifies $X-Z$ and the filtration respect the collaring of $Z$ in $X$.
 
 The object in this category are polyhedrons and the morphism are piecewise linear homomorphism. In the \textbf{PL} pseudomanifold, it admits the \textbf{PL}-stratification. In the simplicial viewpoint, pseudomanifold can be defined as below.
 
\begin{definition}\cite{herreman2005h}
A $n-$ dimensional and closed pseudomanifold is a finite simplicial complex with the following characteristic:
\begin{enumerate}
\item it is non-branching: Every (n-1) dimensional simplex is a face of precisely two n-dimensional simplices.

\item  Any two n-dimensional simplices $\sigma$ $\sigma'$ can be joined by a  "chain" of n-dimensional simplices $\sigma_i$, each pair $(\sigma_i,\sigma_{i-1})$ have a common $(n-1)$-dimensional face.

\item every simplex is a face of some n-dimensional simplex.
\end{enumerate}
\end{definition}
Here the characteristic 1 is equivalent to $\mathrm{dim}(X_{sig})<n-1$. 
 \begin{definition}
Let the triangulation $T$ is compatible with the piecewise linear structure and $C_i^T(X)$ is the chain complex of $X$ with $T$. The chain complex $C_i(X)$ is the direct limit under refinement of $C_i^T(X)$ for all triangulations $T$.
 \end{definition}
 
The key point of Goresky and MacPherson's work is using allowable condition and perversity to control chain with meaningful intersection with singularity. 
\begin{definition}\cite{G}
The perversity is a sequence of integer  $\bar{p}=(p_2,p_3,p_4,....p_n,...)$ with $p_2=0$ and $p_{i+1}=p_i$ or $p_{i+1}=p_i+1$.
\end{definition}

For intersection homology theory, there are four special but important perversities. The zero perversity is  $\bar 0=(0,0,...0)$, and maximum perversity is defined as $\bar t=(0,1,2,...n-2)$. The lower middle perversity is $ \bar{m}=[\frac{n-1}{2}]$ and upper middle perversity is $\bar{n}=[\frac{n-2}{2}]$. If two perverities satisfies $\bar p+ \bar q=\bar t$, we say they are complementary perversities. Clearly $\bar{m}$ and $\bar{n}$ are complementary perversities.
 
We can define the $(\bar p,i)-$ allowable condition.  
\begin{definition}\cite{G}
For a perversity $\bar p$ and an integer $i$, a subspace $Y\subset X$ is called $(\bar p,i)-$ allowable, if $\mathrm{dim}(Y) \leq i$ and $\mathrm{dim}(Y \cap X_{n-k}) \leq i-k+p_k$ for all $k\geq 2$.
\end{definition}
 \begin{definition}\cite{G}
 $\mathrm{IC}_i^{\bar p}(X)$ is the subgroup of $C_i(X)$ consisting of $\xi$ such that
 the support of $\xi$ and $\partial \xi$ are both $(\bar p,i)-$ allowable.
 \end{definition}
\begin{definition}\cite{G}
The intersection homology with perversity $\bar p$ is the homology group of chain complex $\mathrm{IC}^{\bar p}_i(X)$. We use the notation $\mathrm{IH}^{\bar p}_i(X)$ to denote. Although the intersection homology seems to consider the direct limit of all triangulation, we only need to consider the stratification with barycentric subdivision of skeleton in \cite{G4}.
\end{definition}
Usually the intersection homology are not equal for different perversity. However, Siegel find pseudomanifold in some condition have the homology isomorphism :
\begin{equation}
   \mathrm{IH}^{\bar{m}}_{j}(X)\cong\mathrm{IH}^{\bar n}_{j}(X). 
\end{equation} \label{wittsegel}
The pseudomanifold with the Witt condition is called Witt space. Many space such as manifold and any complex projective variety are the Witt space. This is because manifold has no singular stratum and complex variety has only even codimensional stratum. However, the suspension of a torus is not a Witt space because of $H_1(T^2)\neq 0$.   
\begin{definition}\cite{S}
Let $X$ be a stratified pseudomanifold, with stratification \ref{stra}, $(L(\chi_{i}, x)$ is the link of $\chi_i$ at $x$. Then X is a Witt space if and only if for any $i=n-(2l+1)$ with $l\geq 1$.

$$\mathrm{IH}_{\ell}^{\bar{m}}\left(L\left(\chi_{i}, x\right) ; \mathbb{Q}\right)=0$$

\end{definition}
Because we research the non witt space, we consider the Witt condition for the stratum $\chi_i$:
\begin{definition}\label{nwitt}
The Witt condition for the codimensional stratum $\chi_i$ is the vanishing of the rational intersection homology of intrinsic links $(L(\chi_{i}, x)$ :
\begin{equation}
\mathrm{IH}_{\ell}^{\bar{m}}\left(L\left(\chi_{i}, x\right) ; \mathbb{Q}\right)=0
\end{equation}

\end{definition}
The condition of Witt space is stratified homotopy invariant.  If $X$ is a Witt space for some stratification then it is a Witt space for any stratification.  

In \cite{G} Theorem 3.3, Goresky and McPherson prove the Generalized Poincare Duality for intersection homology with complementary perversities $\bar p$ and $\bar q$.
\begin{theorem} \cite{G}
Suppose $X$ is a oriented pseudomanifold. Let $\bar p$ and $\bar q$ are complementary perversity. Then the bilinear form from intersection pairing:
$$\mathrm{IH}_{i}^{\bar p}(X,\mathbb{Q}) \times \mathrm{IH}_{i}^{\bar q}(X,\mathbb{Q}) \rightarrow \mathrm{IH}_{0}^{\bar t}(X,\mathbb{Q}) $$
is nondegenerate.
\end{theorem}

With the combination pf Generalized Poincare Duality and \ref{wittsegel}, it is natural to find lower middle perversity intersection chain is a self dual chain. Siegel construct the generalized signature invariant on the Witt space.

There are some important conception such as link and star needed to study pseudomanifold. 
Let $\sigma$ be a simplex in $X$ with triangulation $T$.
The star of $\sigma$ is the set of all simplices having $\sigma$ as a face.  $$\mathrm{St}(\sigma,T) = \{\tau \in T | \sigma \leq \tau\}.$$ 

Usually the star is not closed set, we use $\overline{\mathrm{St}}$ to represent the closure. The link of $\sigma$ (denoted $lk(\sigma, T)$) is set of all simplices in the closed star that are disjoint from $\sigma$.
$$lk(\sigma, T)=\{v \in \overline{\mathrm{St}}(\sigma, T) \mid v \cap \sigma=\emptyset\}$$
The intrinsic link plays an important role in pseudomanifold. The intrinsic link of barycenter of $i$-dimensional simplex $\hat\sigma$ is defined as $L(\hat\sigma)$. Moreover, we have   $$L(\hat\sigma)*S^{i-1}=lk(\sigma, T').$$  Locally, the neighbourhood $U$ of $x\in \chi_{i}$ is PL-homomorphics to $B^{i}\times C(L(x))$. (local trival condition). It is because $\chi_k$ is the union of interior of $k$-simplex. For one simplex named $\sigma$, we have $$D(\sigma,T')=\hat\sigma*lk(\sigma,T')$$ 

\subsection{Filter complex and $W^{p}_{*}(x)$}
In order to study the pseudomanifold, we introduce a special simplical chain complex named $W^{p}_*(X)$. These chain are first introduced in Siegel's dissertation \cite{S}. Xie and Higson in Appendix A of \cite{X} find the map for every basis of $W^{p}_*(X)$ is geometrically controlled. 

Based on fixed stratification, there are 3 different types of chain related with intersection homology. Although $IC^{\bar p}_*$ is enough to define intersection homology, this chain is not exploit enough to study the stratified homotopy equivalent. Goresky and MacPherson in \cite{G} introduce the basic set $Q^i_p(X)$ which suits for intersection pairing to prove many proposition. However, the disadvantage of the basic sets $Q^i_p$ is not fine enough to show chain equivalence. I will use $W^{p}_*$ instead of others.

With the triangulation $T$,  $X$ can be seen as simplicial complex. Usually we can give the stratification by the skeleton of $T$. More precisely, $X_i$ is a subcomplex of the i-th skeleton of X. 
$$X=X_n \supset  X_{n-2} \supset X_{n-3} ... \supset X_1\supset X_0.$$

With $T'$ is the barycentric subdivision of $T$, and $R^{\bar{p}}_i(X)$ are the $(\bar{p},i)$ allowable simplex in $T'$ with respect to the stratification:
$$\mathrm{dim}(R^{\bar{p}}_i(X))\leq i, \quad \mathrm{dim}(R^{\bar{p}}_i(X)\cap X_{n-k})\leq i-k+p_k.$$
Let $C^{T'}_i(R^{\bar{p}}_i)$ be the free abelian group generated by simplices of $R^{\bar{p}}_i(X)$. 
\begin{definition}
 $W^{\bar{p}}_i(X)$ are the subgroup in $C^{T'}_i(R^{\bar{p}}_i)$ with boundary supported in $R^{\bar{p}}_i$. 
\end{definition}
The intersection homology $\mathrm{IH}^{\bar{p}}_*(X)$ with the perversity $p$ is defined as the homology group with chain complex $W^{\bar{p}}_*(X)$ with boundary map $\partial$ below:
$$W^{\bar{p}}_n(X)\mathop\to\limits^{\partial} W^{\bar{p}}_{n-1}(X)\mathop\to\limits^{\partial}W^{\bar{p}}_{n-2}(X)....W^{\bar{p}}_{2}(X)\mathop\to\limits^{\partial}W^{\bar{p}}_1(X)\mathop\to\limits^{\partial}W^{\bar{p}}_0(X).$$
\begin{remark}
For equivariant case we need to consider the chain complex with coefficient $\mathbb{F}=C_{r}^{*}(\Gamma)$.
\end{remark}
Because the intersection homology deal with cone frequently. The filtration of $C(Z)$ to define $W^{\bar{p}}_{i}(C(Z)$ is different because $v*z'$ is not barycentric subdivision.
Suppose there is a filtration with $Z$:
$$Z=Z_{n} \supset Z_{n-1} \supset \cdots \supset_{1} \supset Z_{0}.$$
Then it induces a filtration on $C(Z)=v*Z$ by 
$$C(Z)=v * Z_n \supset v * Z_{n-1} \supset v * Z_{n-2} \supset \cdots \supset v * Z_{1} \supset v * Z_{0} \supset\{v\}$$
This filtration is flag like and the homology respected by $W^{\bar{p}}_{i}(C(Z)$ is intersection homology . 

The advantage of using $W^{\bar{m}}_{*}(X)$ instead of $\mathrm{IC}^{\bar{m}}_{*}(X)$ or $Q^i_p(X)$ is that $W^{\bar{m}}_j(X)\subset W^{\bar{n}}_j(X)$ for every $j$. Moreover $W^{\bar{m}}_{*}(X)$ has many important properties to construct the Hilbert Poincare complex.

In this thesis, the continuous map $f : X \to Y$ between two stratified spaces require to be stratum- preserving . That is for each pure stratum $T$ of $Y$ , the inverse image $f^{-1}(T)$ is a union of pure strata of $X$. A stratum-preserving map $f : X\to Y$ is placid if for each pure stratum $T$ of $Y$: $$\mathrm{codim}(f^{-1}(T)) \geq\mathrm{codim}(T).$$

For all finite filtered simplicial complex $Z$ of dimension $K-1$, denote $M_k$ to be the set of all $C(Z)$. A element $\xi$ in $W^{\bar{m}}_j(X)$ is modeled by $M_k$ if there exists a cone $C(Z)\in M_k$ and a placid simplicial map $\phi: C(Z) \to X$ such that $\xi = \phi(\omega)$, where $ \omega\in W^{\bar p}_k(C(Z)$.
\begin{lemma}\cite[Proposition A.6]{X} \label{lmofw}
If X is a filtered simplicial complex. There is a natural basis $\xi_i$
of $(W^{\tilde m}_*(X))$ such that $\xi_i$ cannot written as a sum of two nonzero elements and modeled by an element of $M_k$.
\end{lemma}

The lemma of \ref{lmofw} means every basis is supported on the star of vertex, which means the basis element is geometrically controlled over $X$. $W^{\bar{m}}_j(X)$ is a geometrically controlled $X$- module. Next we see the filtered complex $W^{\bar{p}}_i(X)$ is a geometric module over simplicial complex.   

Specially in the case of Witt space, that means the lower middle perversity intersection homology in the odd codimensional stratum about the link vanish, the inclusion map is the chain equivalence in \cite{S}. 
$$\iota:W_{*}^{\bar{m}}(X)  \hookrightarrow W_{*}^{\bar{n}}(X) $$
In particular, we can conclude that for lower middle perversity $\bar m$:
$\mathrm{IH}^{\bar{m}}_{j}(X)\cong\mathrm{IH}^m_{n-j}(X).$  
\begin{definition}
\cite{X} 
Let $X$ be an oriented pseudomanifold of dimension $n$.  $X$
is a geometrically controlled Poincare pseudomanifold of dimension $n$ if the duality chain map $P$ associated to the fundamental class $[X]$ is a chain equivalence in the geometrically controlled category. 
\end{definition}
Higson and Xie in \cite{X} theorem C.9 prove oritended Witt space is a geometrically controlled Poincare pseudomanifold. In this thesis, we need the result below.

Recall lower middle perversity $ \bar{m}=[\frac{n-1}{2}]$ and upper middle perversity $\bar{n}=[\frac{n-2}{2}]$. 
\begin{theorem}\cite[Theorem 4.3]{X}  \label{ordiagonal}
There exists a diagonal approximation map 
$$\Delta: W_{*}^{\overline{0}}(X) \rightarrow W_{*}^{\bar{m}}(X) \otimes W_{*}^{\bar{n}}(X)$$
that is unique up to chain homotopy.
Then the cap product is defined as :
$$\cap: W_{\bar{m}}^{j}(X) \otimes W_{k}^{\overline{0}}(X) \stackrel{q \otimes \Delta}{\longrightarrow} W_{\bar{p}}^{j}(X) \otimes\left(W_{*}^{\bar{m}}(X) \otimes W_{*}^{\bar{n}}(X)\right)_{k} \stackrel{\varepsilon \otimes 1}{\longrightarrow} W_{k-j}^{\bar{n}}(X)$$
\end{theorem}
Here we use $W^j_{ \bar{m}}(X) =\mathrm{Hom}\left(W_{j}^{\bar{p}}(X), \mathbb{C}\right)$ to represent the dual cochain complex of $W_j^{ \bar{m}}(X)$. The general Poincare duality map for lower middle perversity chain $W^{\bar{m}}_*$ and upper middle perversity chain $W^{ \bar{n}}_*$ can be defined by 
$$\mathbb{P}:=- \cap [X] : W^j_{ \bar{m}}(X) \to W^{ \bar{n}}_{n-j}(X),$$
here the $[X] \in \mathrm{IH}_{n}^{\overline{0}}(X)$ are the fundamental class. 

This map satisfies $\partial\mathbb{P}\upsilon= (-1)^j\mathbb{P}\partial^{*}\upsilon$ for all
$\upsilon\in W^j_{\bar{n}}(X)$, and this map introduce a chain equivalence and furthermore an isomorphism from $\mathrm{IH}_{\bar{m}}^{j}(X)$ to $\mathrm{IH}^n_{n-j}(X)$. 

In the equivariant case, we use the filter complex $W^{\bar p}_i(X,F)$ with the local coefficient system $F = C^*_r(\Gamma)$, we obtain a chain complex of $C^*_r(\Gamma)$-Hilbert module 

In this dissertation, motivated by Cheeger \cite{C}, I use Higson and Xie's technique and results in \cite{X} to prove in some condition there are two equivalent chain $\widetilde{W}^{\bar{n}}_*(X)$ which interpolate with the lower middle perversity and upper middle perversity intersection homology chain complex:  
$$W^{\bar{m}}_j(X)\hookrightarrow\widetilde{W}^{\bar{m}}_j(X)\hookrightarrow \widetilde{W}^{\bar{n}}_j(X) \hookrightarrow W^{\bar{n}}_j(X).$$ 
Because we know the Poincare map is $\mathbb{P}:=- \cap [X] : W^j_{ \bar{m}}(X) \to W^{ \bar{n}}_{n-j}(X),$ the new Poincare dual on the new chains: 
$$\widetilde{\mathbb{P}}:=-\widetilde{\cap} [X]:\widetilde{W}_{\bar{m}}^{j}(X)\to\widetilde{W}^{\bar{n}}_{n-j}(X),$$
More specially, ${\widetilde{\mathrm{IH}}}^{\bar{m}}_j(X)\cong {\widetilde{\mathrm{IH}}}^{\bar{n}}_{n-j}(X)$. 
 Because every basis of is modelled by an element of cone, let us consider the cone first. 

\begin{remark}
In the pseudomanifold, the neighborhood of $x$ in $\chi_{k}$ is piecewise linear homomorphic to $x*S^{k-1}*L_x$, here $L_x$ is unique in P.L-category when $x$ in the same stratum $\chi_{k}$. So the key point of this note is dealing with the cone of link. Furthermore, we have $\mathrm{IH}^{p}_i(X,X-x)= \mathrm{IH}^{p}_{i-k-1}(L_x)$ when $i>n-p(n-k)-1$, and otherwise is 0. That means local homology is total decided by the link of stratum. \cite{G2}. However, the local homology of manifold is totally different with the pseudomanifold. For the manifold $M$, any point $x\in M$, the only non trivial homology is $\mathrm{H}_n(M, M-x)=Z$. So the key point of Poincare duality is the duality in the link of stratum. 
\end{remark}

\subsection{K-homology Class on the Geometrically Controlled Hilbert\\ Poincare Complex}\label{khl}

In the appendix of \cite{XX}, Weinberger, Xie and Yu construct the analytic K-homology class of signature class on the PL manifold $M$.  In fact, their methods can generalize to any combinatorial geometrically controlled Poincare complex on the PL pseudomanifold in \cite{X}. They proved that if there is a self dual coefficient system on simplicial complex, we can build the K-homology class on it. An advantage of this construction of k-homology is that we can define higher signature class if we consider the coefficient system $C_r^*(\Gamma)$ -Hilbert module. My research use the same framework to study higher signature on the Non-Witt space. My next target is building geometrically controlled Poincare complex on the Non-Witt space. 

A bounded geometry combinatorial manifold with triangulation is a bounded geometry simplicial complex.
 
The outline of \cite{XX} is that. Suppose there is a control map $\varphi: X\rightarrow X$. Here $X$ is a metric space with path metric. The key property is that the geometrically controlled $X$-module $E$ in PL space is homotopy invariant under the subdivision. Then we obtain a geometrically controlled Poincare complex $E^j$ of $X$ whose propagation approaches $0$ as $j \to \infty$. Then it is nature to form a continuous family of geometrically controlled Poincare complexes parametrized by t. This family  determines the localization algebra $C_{L}^{*}(X)$ then the K-homology class is an element of $K_{n}\left(C_{L}^{*}(X)\right)$.

Because we consider the bounded geometry space. This requires the number of simplices containing any given vertex should be uniformly bounded for all successive subdivision. This is call standard subdivision. Then we can use this subdivision to control the propagation. It is natural to construct the element of K-homology
\begin{definition}\label{sub1}\cite{XX}
Let $\sigma = [v_0,....v_n]$ be a standard simplex where the vertices $v_i$ with given order. Define the standard subdivision $\mathrm{Sub}(\sigma)$ as below :
\begin{equation}
    v_{ij}=(v_i+v_j)/2 \ \text{when}\ i\leq j
\end{equation}
When $j=i$,  $v_{ii}$ is just $v_i$. Hence the new division inherits the partial order of old vertices by setting :
\begin{equation}
v_{i j} \leq v_{k l} \quad \text { if } \quad k \leq i \text { and } j \leq l
\end{equation}
\end{definition}
We can follow the method of \cite{X} chapter 6 without any change. If $X$ is a bounded geometry piecewise linear Poincare pseudomanifold, then  $\operatorname{Sub}^n(X)=\operatorname{Sub}(\operatorname{Sub}^{n-1}(X))$ is uniform bounded geometry for $n\in\mathbb{N}$. Let $(W^{ \triangle }_{*}(X)\otimes \mathbb{C},b)$ be the Poincare complex based on the filtered simplicial complex $W^{\bar p}_*(X)$. 

Let us define 
$Q_{0}=\bigoplus_{k} W^{ \triangle }_{k}(X)\otimes \mathbb{C}$ and $Q_{2}=\bigoplus_{k} W^{ \triangle }_{k}(\mathrm{Sub}(X))\otimes \mathbb{C}$. Build $\hat Q_1$ and $Q_1$ for $\bigoplus_{k} C^{ \triangle }_{k}(\mathrm{Sub}(X))$ with this different geometrically controlled X-module structure such that $Q_0$ is submodule of $\hat Q_1$, $\hat Q_1$ and $Q_1$ are isomorphic.
Next construct a uniform family of geometrically controlled Poincare complexes $\left\{\left(Q_{t}, b_{t}, S_{t}\right)\right\}_{t \in[0,2]}$ to connect $Q_0$ and $Q_2$:
$$\left(Q_{t}, b_{t}, S_{t}\right)=\left\{\begin{array}{ll}\left(\widehat{Q}_{1}, b_{t}, S_{t}\right) & \text { if } 0 \leq t \leq 1 \\ \left(Q_{t}, b, S\right) & \text { if } 1 \leq t \leq 2\end{array}\right.$$

Similar, define the geometrically controlled $X$-module $Q_{2j}=\bigoplus_{k} {W}^{ \triangle}_{*}(\mathrm{Sub}^j(X))$. If there is a geometrically controlled Poincare complex of $X$, propagation of Poincare complex on $Q_{2j}$ approach 0 when j go to $\infty$. 

Moreover connect all $Q_{2j}$, $\left\{\left(Q_{t}, b_{t}, S_{t}\right)\right\}_{t \in[0,\infty)}$ is a uniform family of geometrically controlled Poincare complexes.
$$\varepsilon_1<\left\|B_{t} \pm S_{t}\right\|<C_1$$
$p(x)$ is a polynomial on $[\varepsilon, C] \cup[-C,-\varepsilon]$ such that
$\sup _{x \in[\varepsilon, C]}\left|p(x)-\frac{1}{x}\right|<\frac{1}{C}.$
Then $p(B_t-S_t)$ is invertible. Let $$U_{t}:=\left(B_{t}+S_{t}\right) \cdot p\left(B_{t}-S_{t}\right)$$
In fact, it is a norm-bounded and uniformly continuous path of
invertible elements and the propagation of $U_t$ approach 0. So we can use the definition of  K-homology class of the signature operator on Riemannian manifold $M$ to define the K-homology class on $X$.
\begin{definition}
The K-homology class of the signature operator of $X$ is defined to be the K-theory class of the path $U$ in $K_1(C_L^*(X))$. K-homology class is $[D_{\mathrm{sign}}]$ .
\end{definition}

\begin{definition}
The K-homology class of the signature operator on $X$ is defined to be the K-theory class in $K_0(C_L^*(X))$ determined by $Q$: a norm-bounded and uniformly continuous path of $\sigma$-quasi-projections $[0,\infty] \to C^*(X)$. 
\end{definition}

\section{Non Witt Space with conical singularities}

In this chapter we consider the odd non Witt space where the Witt condition only fails in the link of conical singularities. We will show the Lagrange structure make $X$ to be a geometrically controlled Poincare pseudomanifold. Then we give the higher signature index class for this spaces in the framework of Higson and Roe. 

Let $X$ be an odd dimensional pseudomanifold. Here the dimension we assume is $k=2s+1$. For the odd codimensional stratum $\chi_{k-2j+1}$, we define the intrinsic link about $x\in\chi_{k-2j-1}$ is $L(\chi_{k-2j-1},x)$. We know the Witt condition about the intersection homology of intrinsic link is  $L(\chi_{k-2j-1},x)$:
$$\mathrm{IH}^{\bar m}_{j}(L(\chi_{k-2j-1},x),Q)= 0.$$  

\begin{definition}\label{nonwi1}
The Non Witt space $X$ in this chapter is a $2s+1$ dimensional pseudomanifold satisfies:  
$$\mathrm{IH}^{\bar m}_{s}(L(\chi_0,x),Q)\neq 0.$$
and for any other odd codimension $2j+1$ stratum $\chi_{k-2j+1}$, we have:  
$$\mathrm{IH}^m_{j}(L(\chi_{2s+1-2j-1},x),Q)= 0.$$
\end{definition}

Because the Witt condition is necessary for $\mathrm{IH}^m_i(X)\cong \mathrm{IH}^n_i(X)$, we can get $\mathrm{IH}^m_i(X)\neq \mathrm{IH}^n_i(X)$ at least for some dimension $i$. Specially, the $W^{\bar{m}}_*(X)$ is chain equivalent to $W^{\bar{n}}_*(X)$.

In this section, we use the technique of Xie and Higson's article \cite{X} Appendix C to filter the $W^{\bar{m}}_*(X)$ and $W^{\bar{n}}_*(X)$. This method is a generalization of framework in \cite{S} Chapter III section 3.  First we define the perversity $\bar P_r$ :
\[
    \bar P_r(i)=\left\{
                \begin{array}{ll}
                  \bar{m}(i) \ \text{for} \  i<r\\
                \bar{n}(i) \ \text{for} \  i>r
                \end{array}
              \right.
  \]
Let us consider the filtration 
$$W_{*}^{\bar{m}}=W_{*}^{\bar{p}_{2 r+1}} \subset W_{*}^{\bar{p}_{2 r-1}} \subset \cdots \subset W_{*}^{\bar{p}_{3}} \subset W_{*}^{\bar{p}_{1}}=W_{*}^{\bar{n}}$$ 
By the definition, for all other odd codimension 
$ 2j+1\neq 2s+1$, we have Witt condition, 
$$\mathrm{IH}^m_{l}(L(\chi_{2s+1-2j-1},x),Q)= 0,$$
 then it can be proved $W^{\bar P_{2j-1}}_*(X)$ and $W^{\bar P_{2j+1}}_*(X)$  are chain equivalent by \cite{S} Chapter III, Theorem 3.2.In fact in \cite{X} Appendix C, the chain equivalence is in a geometrically controlled chain equivalent. Similar,  $W^{\bar{n}}_*(X)$ and $W^{\bar P_{2s-1}}_*(X)$ are chain equivalent.

So we only need to consider the chain complex $W^{\bar P_{2s+1}}_*(X)$ and $W^{\bar P_{2s-1}}_*(X)$, the allowable requirement of the two chains complex are the same except codimension $2s+1$ .  When $y\in W^{\bar P_{2s-1}}_i(X)$, let us consider the $\bar P_{2s-1}$ allowable inequality respect to $X_0=\chi_{2s+1-(2s+1)}$: 
$$\mathrm{dim}(y\cap \chi_{2s+1-(2s+1)})\leq i-(2s+1)+\bar P_{2s-1}(2s+1)=i-s-1.$$
The stronger restriction about $y\in W^{\bar P_{2s+1}}_i(X)$ is: $$\mathrm{dim}(y\cap \chi_{2s+1-(2s+1)})\leq i-(2s+1)+\bar P_{2s+1}(2s+1)=i-s-2.$$
Clearly we can find for $i<s+1$ and $i>s+1$, $W^{\bar P_{2s+1}}_i(X)=W^{\bar P_{2s-1}}_i(X)$. So we only need to consider the  $s+1$ dimensional chain . Let us fix the direct sum decompose 
\begin{equation}
    W^{\bar P_{2s-1}}_{s+1}(X)=W^{\bar P_{2s+1}}_{s+1}(X)+U.
\end{equation}
We need to change the $W^{\bar{P}_{2s+1}}_{s+1}(X)$ and $W^{\bar P_{2s-1}}_{s+1}(X)$ such that the new chains are chain equivalent respect to the inclusion map. Next Step shows that the Poincare duality map still holds for new chains.

Let $T$ be the triangulation for defining $W^{\bar m}_{*}(X)$. Define $T'$ as first barycentric subdivision. For the chain $W^{\bar{m}}_*(X)$, there is a basis which element is minimal and modeled by $\mathcal{M}_k$ in \cite{X}. There is a natural basis $\sum u_i$  in the space of U . So we require $u_i$ intersect with $X_0$.   $$\mathrm{dim}(u_i\cap X_{0})=0.$$ 

We assume the intersection points are $\tau_i\in X_0$. For each $\tau_i$, we can see the subchain of $u_i$ as the combination of simplices $\sigma_j$ which is the join complex with the intersection point and $s$-simplex $v_i\in C^{T'}_{s+1}(lk(\sigma_i,T'))$:
$$u_i=\tau_i*v_i.$$ Here $lk(\sigma_i,T')$ is the link respect to $\tau_i$ in $T'$, then it is the intrinsic link with singular stratum. 

In \cite{S} Chapter III, Lemma 3.3 , we know $\bar{v_i}$ is actually a cycle of  $C^{T'}_{s+1}(lk(\sigma_i,T'))$
$$\partial({v_i})=0. $$ 
So we can find  the subspace $U$ is linear combination of $\tau_i*v_i$. 
Moreover, there is a  canonical isomorphism between cycle $\bar{v_i}\in C^{T'}_{s}(lk(\sigma_i,T)')$ and $v_i$.
In addition, by Siegl \cite{S} Lemma 3.4, $v_i$ satisfies the allowable condition of perversity $\bar m$. Combine all information above, we get 
$$\bar v_i\in W^{\bar m}_s(lk(\sigma_i,T)').$$

We can decompose the cycle $\bar{v_i}$ into the boundary of chain $$\bar{w} \in W^{\bar P_{2s+1}}_{s+1}(lk(\sigma_i,T)')$$ and the representative element in homology group
$\bar{h}\in \mathrm{IH}^{\bar{m}}_s lk(\sigma_i,T)$, such that $$\bar{v}=\partial (\bar{w})+\bar{h}.$$  
Let the set of $s$-cycle in  $W^{\bar m}_s(lk(\sigma,T))$ be $V$. Next we will separate $V$ into 3 sets.
\subsection{Lagrange Structure}
Let us consider the lower middle intersection homology 
$\mathrm{IH}^{\bar{m}}_s(lk(\sigma_i,T))$. 
 Cheeger in \cite{cheeger1979spectral} prove that if the signature of the link is 0, there is a Lagrangian space in $\mathrm{H}$. This means that there is a subspace $\mathrm{H_l}$ which is isomorphic to its annihilator about the intersection form. In the geometrical viewpoint it means: 
\begin{enumerate}\label{l1}
    	\item $\mathrm{IH}^{\bar{m}}_s (lk(\sigma_i,T))=\mathrm{H}_{lag}\oplus \mathrm{H}_{lag}^{\perp}$,
        \item $ \mathrm{H}_{lag} \cong \mathrm{H}_{lag}^{\perp}\mathrm{,} $.
        \item For every $h_1\in \mathrm{H}_{lag}$ and $h_2\in \mathrm{H}_{lag}^{\perp}$, then there exist two representatives $v_1$ for $h_1$ and $v_2$ for $h_2$ , we have $v_1\pitchfork v_2 \neq 0$. Here the signal $\pitchfork$ is the transverse intersection.
    \end{enumerate}
By assumption, if $-[L']$ is the fundamental class of $lk(\sigma_i,T')$ , we know there is a mapping from $\mathrm{IH}_{\bar{m}}^s(lk(\sigma_i,T))$ to $\mathrm{IH}^{\bar{m}}_s(lk(\sigma_i,T))$ induced by Poincare dual map of $\mathbb{P}'=-[L']$. For any two cohomology class 
$[v^*_1] \in\mathrm{Hom}(\mathrm{H}_{lag},\mathbb{C})$ and  $[v^*_2]\in \mathrm{Hom}(\mathrm{H}_{lag}^{\perp}, \mathbb{C})$, then we can get $$\mathbb{P}'([v^*_1])\in \mathrm{H}_{lag}^{\perp}\ , \ \text{and} \ \mathbb{P}'([v^*_2])\in \mathrm{H}_{lag}.$$ 
Let the space $V_{lag}$ be all the possible nonzero representative of $\mathrm{H}_{lag}$, i.e 
\begin{equation} \label{eqv1}
 V_{lag}=\{h + \partial w \in V | h\neq 0 \in\mathrm{H}_{lag} \ , and \ w \in W^{p_{2s+1}}_{s+1}(lk(\tau_i,T)')\}. 
\end{equation} 
We can define $V_{lag}^{\perp}$ similarly .
For $V_o$ we define :
\begin{equation} \label{eqo}
V_o=\{ \partial w \in V|w \in W^{p_{2s+1}}_{s+1}(lk(w_i,T)')\}.
\end{equation} 
So we decompose $V$ into:
$$V=V_{lag}\oplus V_{lag}^{\perp}\oplus V_o$$
We know $U=\sum \tau_i*v_i$, so the separation of $V$ is the separation of U. Define $C(V_{lag})$ as: 
$$C(V_{lag})=\sum \tau_i*v_i \quad \text{when} \quad v_i\in V_{lag} $$ Similarly, we can define $C(V_{lag}^{\perp})$ and $C(V_o)$.  Hence we can decompose $U$ as :
\begin{equation}
U=C(V_{lag})\oplus C(V_{lag}^{\perp})\oplus C(V_o)
\end{equation}
Let us define:  
$$\widetilde{W}^{\bar P_{2s-1}}_{s+1}=W^{\bar P_{2s-1}}_{s+1}-C(V_{lag}^{\perp})$$
$$\widetilde{W}^{\bar P_{2s+1}}_{s+1}=W^{\bar P_{2s+1}}_{s+1}+C(V_{lag})+C(V_o)$$
For the definition, we can still get $\widetilde{W}^{\bar P_{2s+1}}_{s+1}\subset \widetilde{W}^{\bar P_{2s-1}}_{s+1}$.

Because $\partial\widetilde{W}^{P_{2s+1}}_{s+1}(X)\subset\widetilde{W}^{P_{2s+1}}_{s}(X)$, the chain complex are well defined in the $X$. For any subcomplex $Y$ of the $X$, we can let 
$$\widetilde{W}^{P_{2s-1}}_{j}(Y)=W^{P_{2s-1}}_{j}(Y)\cap\widetilde{W}^{P_{2s-1}}_{j}(X).$$ 
 
The refined intersection homology is not topological invariant or stratified homotopy invariant because it depends on the Lagrange structure of boundary, but it can play role in other fields.

Because we know the $W^{P_{2j-1}}_{*}(X)$ is equivalent to ${W}^{P_{2j+1}}_{*}(X)$ when $j<s$ or $j>s$, we need to add the same space for these chain complex. If $j< s$, we define :
$\widetilde{W}^{\bar P_{2j-1}}_{j}(X)=W^{P_{2s-1}}_{j}(X)$ If $j>s$
$\widetilde{W}^{\bar P_{2j+1}}_{j}(X)=W^{P_{2s+1}}_{j}(X).$ Specially, for chain complex of lower middle perversity $W_{*}^{\bar{m}}=W_{*}^{\bar{p}_{2 r+1}}$ and upper middle perversity $W_{*}^{\bar{p}_{1}}=W_{*}^{\bar{n}}$, we can make the same change.
\begin{definition}\label{mi1}
The modified intersection chain $\widetilde{W}^{\bar m}_{*}(X)$ and $\widetilde{W}^{\bar{n}}_{j}(X)$ are defined as:
\begin{equation}
   \widetilde{W}^{\bar{m}}_{j}(X)={W}^{\bar m}_{j}(X) +C(V_{lag})+C(V_o). 
\end{equation}
\begin{equation}
   \widetilde{W}^{\bar{n}}_{j}(X)={W}^{\bar{n}}_{j}(X)-C(V_{lag}^{\perp}).
\end{equation}
\end{definition}
Because $\partial \widetilde{W}^{\bar{m}}_{j}(X)\subset\widetilde{W}^{\bar{m}}_{j-1}(X)$, in fact we construct two chain complex $\widetilde{W}^{\bar{m}}_n(X)$ and $\widetilde{W}^{\bar{m}}_n(X)$:
$$ \widetilde{W}^{\bar{m}}_n(X)\mathop\to\limits^{\partial}  \widetilde{W}^{\bar{m}}_{n-1}(X)\mathop\to\limits^{\partial} \widetilde{W}^{\bar{m}}_{n-2}(X).... \widetilde{W}^{\bar{m}}_{2}(X)\mathop\to\limits^{\partial} \widetilde{W}^{\bar{m}}_1(X)\mathop\to\limits^{\partial} \widetilde{W}^{\bar{m}}_0(X).$$
These two chain complex interpolate the filter chain complex with lower middle perversity and upper middle perversity.
\begin{lemma} \label{l3}
The new chains $\widetilde{W}^{\bar m}_{*}(X)$ and $\widetilde{W}^{\bar n}_{*}(X)$ are chain equivalent respect to the inclusion map $\iota$.

\end{lemma} 
\begin{proof}
Witt condition holds for other odd codimensional stratums. $W^{P_{2s-1}}_{*}(X)$ is equivalent to ${W}^{n}_{*}(X)$ and 
$\widetilde{W}^{P_{2s-1}}_{*}(X)$ is chain equivalent to $\widetilde{W}^{n}_{*}(X)$ Because of \cite{S} Chapter III, Theorem 3.2. In fact, $\widetilde{W}^{P_{2s+1}}_{*}(X)$ and $\widetilde{W}^{ P_{2s-1}}_{*}(X)$ are chain equivalent due to the construction.  
\end{proof}
By the construction, $\widetilde{W}^{\bar{n}}_{j}(X)$ are chain equivalent to  $\widetilde{W}^{P_{2s-1}}_{j}(X)$.
Let $\widetilde{W}_{\bar{n}}^{*}(X)$ be $\mathrm{Hom}(\widetilde{W}^{\bar{n}}_{*},\mathbb{C}),$ and $\widetilde{W}_{\bar{m}}^{*} =\mathrm{Hom}(\widetilde{W}^{\bar{m}}_{*}(X),\mathbb{C})$. Let $d_i=\partial_i^*$ be the differential operator which is adjoint of the boundary map $\partial$.  Because  $\mathrm{Hom}(,\mathbb{C})$ functor are exact, then the $\widetilde{W}^{\bar{m}}_{*}(X)$ and $\widetilde{W}^{ \bar{n}}_{*}(X)$ are chain equivalent by the lemma \ref{l3}.  Here the chian map from $\widetilde{W}^{\bar{m}}_{*}(X)$ to $\widetilde{W}^{ \bar{n}}_{*}(X)$ is $\iota^*$ which is induced by $\iota$. \\
\begin{tikzcd}
   & & ... & \widetilde{W}_{\bar{n}}^{i+1}(X) \arrow{l}{d_{i+1}} \arrow[swap]{d}{\iota^{*}_{i+1}}
& \widetilde{W}_{\bar{n}}^i(X)\arrow{d}{\iota^{*}_i} \arrow{l}{d_{i}} & \widetilde{W}_{\bar{n}}^{i-1}(X) \arrow{d}{\iota^{*}_{i-1}} \arrow{l}{d_{i-1}} & ... \arrow{l}{d_{i-2}} \\
  & &...& \widetilde{W}_{\bar{m}}^{i+1}(X) \arrow{l}{d_{i+1}} 
& \widetilde{W}_{\bar{m}}^i(X) \arrow{l}{d_{i}} &  \widetilde{W}_{\bar{m}}^{i-1}(X) \arrow{l}{d_{i-1}} &...\arrow{l}{d_{i+1}}
\end{tikzcd}\\
\begin{prop}
There is a basis
of $\widetilde{W}^{\bar{m}}_{*}(X)$ such that each basis element is minimal and moreover modeled by $M_k$.
\end{prop}
Because of the proposition A.6 in \cite{X} (also see\ref{lmofw}),  we know we just add the Lagrange subspace $C(V_{lag})$ of $U$ to $W^{\bar{m}}_{*}(X)$ to construct $\widetilde{W}^{\bar{m}}_{*}(X)$. 
\subsection{Cone Formula}
The diagonal approximation and Poincare duality relay on the cone formula of intersection homology. Let $Z$ be a simplex of dimension $j$ in $X$, and the cone $C(Z)=\tau*Z$. The intersection homology $\mathrm{IH}^{\bar p}_{i}$ with perversity $\bar p$ we have the cone formula because of \cite{kirwan2006introduction}:
\begin{equation}
\mathrm{IH}^{\bar p}_{i}(C(Z))=
\begin{cases}
\mathrm{IH}^{\bar p}_{i}(Z) \quad  i<j- \bar p(j+1)  \\
0 \quad  otherwise.
\end{cases}
\end{equation}
Because we know the difference between  $\widetilde{W}^{\bar{m}}_{*}(X)$ and $W^{\bar{m}}_{*}(X)$ is in the dimension $s+1$. When $j<2s-1$, we get  
$$s+1 >j-\bar m(j+1).$$ So we get the cone formula related $\widetilde{W}^{ \bar{m}}_{i}(C(Z))$ is just same with $\mathrm{IH}^{\bar p}_{i}(C(Z))$. The intersection homology of cone formula respect to $\widetilde{W}^{ \bar{n}}_{i}(C(Z))$ is same too, because $\widetilde{W}^{ \bar{m}}_{*}$ is chain equivalent with $\widetilde{W}^{ \bar{m}}_{*}$. 
\begin{equation}
\widetilde{\mathrm{IH}}^{\bar m}_{i}(C(Z))=\begin{cases}
\mathrm{IH}^{\bar m}_{i}(Z) \quad  i<j- \bar m(j+1)  \\
0 \quad  otherwise.
\end{cases}
\label{c1}\end{equation}

When $j=2s$ and $\tau$ is the vertex in the singularity, $s+1=j-\bar m(j+1).$ For intersection homology with lower middle perversity and upper middle perversity is same when $i\neq s$ $$\mathrm{IH}^{\bar m}_{i}(C(Z))=\mathrm{IH}^{\bar n}_{i}(C(Z)). $$ 
For dimension $s$, $\mathrm{IH}^{\bar n}_{s}(C(Z))=0$ and $\mathrm{IH}^{\bar m}_{s}(C(Z))={\mathrm{IH}}^{\bar m}_{s}(Z)$. Here is because $W^{\bar n}_{s-1}(X)=W^{\bar m}_{s+1}(X)+U.$ and the element of $U$ is the cone of cycle in $W^{\bar m}_s(lk(\sigma,T))$.   

Consider the construction in previous section, we add $C(V_{lag})$ to ${W}^{ \bar{m}}_{s+1}(x)$, the non trivial homology element is from $C(V^{\perp}_{lag})$. That means $\widetilde{\mathrm{IH}}^{\bar m}_{s}(C(Z))$ should be $\mathrm{H}_{lag}$. We get the cone formula :
\begin{equation}
\widetilde{\mathrm{IH}}^{\bar m}_{i}(C(Z))=\widetilde{\mathrm{IH}}^{\bar n} _{i}(C(Z))=
\begin{cases}
\mathrm{IH}^{\bar m}_{i}(Z) \quad  i<s\\
\mathrm{H}_{lag}\in \mathrm{IH}^{m}_{s}(Z) \quad i=s \\
0 \quad  otherwise.
\end{cases}
\label{c2}    
\end{equation}
\begin{remark}
The property of the intersection homology is invariant after the normalization of pseudomanifold, so the self dual chain should satisfy this property. The new homology is not independent of stratification.
\end{remark}

\subsection{Diagonal Approximation for Modified Intersection Chain}

Diagonal approximation plays important role in Poincare dual map. However, the front $j$ face or back $n-j$
face of $(\bar p,i)$-allowable chain 
may not be $(\bar p,i)$-allowable or $(\bar p,n-j)$-allowable respectively, so the ordinary Whitney-Alexander diagonal approximation fails to exist in the intersection chain. In \cite{X} Apendix B
, there exists the diagonal approximation map for lower middle perversity  $W^{\bar{m}}_{*}(X)$ and upper middle perversity $W^{\bar{n}}_{*}(X)$ :
$$\bigtriangleup : W^{\bar{0}}_{*}(X)\to W^{\bar{n}}_{*}(X) \otimes W^{\bar{m}}_{*}(X).$$ 
with those properties: 
\begin{enumerate}
    \item diagonal approximation is natural in X, for the filtered simplicial complex and placid simplicial maps.
    \item   $\bigtriangleup (x) = x\otimes x $ for any $(\bar{0},0)-$ allowable simplex $x\in X $.
\end{enumerate}
Here $\bar{0}=0$ is the zero perversity, and the placid map is a stratum-preserving map for every pure stratum $T$, the codimension of every pure stratum in $f^{-1}(T)$ is no less than codimension of $T$.
Furthermore, this diagonal approximation is unique up to homotopy in the category of geometrically controlled. 

\begin{remark}
First we recall that for the zero perversity $\bar 0$, the minimal element of $W^{\bar{0}}_{k}(X)$ is modeled by $W^{\bar 0}_k(C(Z))$. Here Z is the intrinsic link of simplex $\sigma$. Hence it is natural to consider the link 
\end{remark}

The boundary map in tensor product of chains is : $$\partial(c\otimes c')=\partial c\otimes c'+(-1)^k c\otimes \partial c'.$$ So the tensor product of cycle is still a cycle. In this section, I will build the diagonal approximation $\bar\bigtriangleup$ in the new chain $\widetilde{W}^{\bar{m}}_{s}(X)$ and $\widetilde{W}^{\bar{n}}_{s+1}(X)$.

\begin{lemma}\label{diagonala}
There exists a diagonal approximation $\widetilde{\bigtriangleup}$ for the new chains $\widetilde{W}^{\bar{m}}_{j}(X)$ and $\widetilde{W}^{ \bar{n}}_{j}(X)$ in the meaning of \cite{X}. 
\begin{equation}\widetilde \bigtriangleup :W^{\bar{0}}_{*}(X)\to \widetilde{W}^{\bar{m}}_{*}(X)\otimes \widetilde{W}^{ \bar{n}}_{*}(X).\end{equation}\label{da1}
Moreover $\widetilde{\bigtriangleup}$ is unique up to chain homotopy.
\begin{proof}
The proof is standard method of acyclic models.
Because in article \cite{X} apendix B has defined the diagonal approximation map $$\bigtriangleup :W^{\bar{0}}_{*}(X)\to{W}^{\bar{m}}_{*}(X)\otimes{W}^{\bar{n}}_{*}(X).$$  
In order to define \ref{da1}, consider in \ref{mi1} that we just modified ${W}^{\bar{m}}_{s+1}(Y)$ and ${W}^{\bar{n}}_{s+1}(Y)$.  When the dimension $*$ is less than $s+1$, the diagonal approximation map $\widetilde{\bigtriangleup}$ is the original $\bigtriangleup$. Let us use the mathematical induction to deal with dimension from s+1 to 2s. The basis of $W^0_{s+1}(X)$ is modeled over the cone, it is enough to consider the case $W^0_{s+1}(C(Z))$ for some cone $C(Z)$.

For $s<j<2s+1$, because $\bigtriangleup_{j}$ should be a chain morphism, this means $$\partial(\bigtriangleup_{j}(\partial \omega))=0.$$ So $\bigtriangleup_{j}(\partial \omega)$ is a cycle. 
Because of the lemma \ref{l2}, we can observe that is a boundary of $$\exists \zeta\in (\widetilde{W}^{ \bar{n}}_{*}(C(Z)\otimes \widetilde{W}^{ \bar{n}}_{*}(C(Z))_{j+1},\  \text{s.t} \ \bigtriangleup_{i}(\partial \omega) =\partial \zeta.$$ So we can define $\bigtriangleup_{j+1}$ because it satisfy the chain map property $\partial \Delta_{i+1}(\omega)=\partial \zeta=\Delta_{i}(\partial \omega)$:
$$\bigtriangleup_{j+1}(\omega)=\zeta.$$ 

Next we prove the uniqueness up to homotopy. Let us assume there is another diagonal approximation $\widetilde\bigtriangleup' $. Then construct the homotopy $\{h_i\}$ between $\widetilde\bigtriangleup $ and $\widetilde\bigtriangleup' $ by using mathematical induction.
Because of $\widetilde\bigtriangleup_0 = \widetilde\bigtriangleup'_0.$
then we define $h_0=0$. If we have defined $h_i$ for all $i<s$, the chain homotopy $h_{i+1}:W_{i+1}^{\overline{0}}(X) \rightarrow(\widetilde{W}_{*}^{\bar{m}}(X) \otimes \widetilde{W}_{*}^{\bar{n}}(X))_{i+2} $ must meet : $$\partial h_{i+1}+h_{i} \partial =\Delta_{i+1}-\Delta_{i+1}^{\prime}.$$
When $\xi$ is a basis of $W_{i+1}^{\overline{0}}(X)$,  $\partial h_{i+1}\xi =(\Delta_{i+1}-\Delta_{i+1}^{\prime}-h_{i} \partial) \xi$ is a cycle. This is because 
$$\begin{aligned} & \partial\left(\Delta_{i+1}-\Delta_{i+1}^{\prime}-h_{i} \partial\right) \xi =\left(\Delta_{i} \partial-\Delta_{i}^{\prime} \partial-\partial h_{i} \partial\right) \xi \\=&\left(\Delta_{i} \partial-\Delta_{i}^{\prime} \partial-\left(\Delta_{i}-\Delta_{i}^{\prime}-h_{i-1} \partial\right) \partial\right) \xi = 0 \end{aligned}$$
Then due to lemma \ref{l2}, we know $\partial h_{i+1}\xi$ should be a boundary of $\zeta$ in $\widetilde{W}_{*}^{\bar{m}}(X) \otimes \widetilde{W}_{*}^{\bar{n}}(X))_{i+2}$. Then it is natural to define $h_{i+1}(\omega) =\zeta$. Then we define a chain homotopy ${h_j}$ between $\widetilde\bigtriangleup $ and $\widetilde\bigtriangleup' $.
\end{proof}
\end{lemma}
Application of acyclic method need to compute the homology. For any $i$-subcomplex of $Z\in X$,  $C(Z)\in X$ is a cone.  Assuming the diagonal approximation map exist for the link $Z$: $$\widetilde\bigtriangleup :W^{\bar{0}}_{*}(Z)\to \widetilde{W}^{\bar{m}}_{*}(Z)\otimes\widetilde{W}^{\bar{n}}_{*}(Z).$$ 
For the the chain complex  $\widetilde{W}^{\bar{m}}_{*}(Z)\otimes\widetilde{W}^{\bar{n}}_{*}(Z)$ in the image of $\widetilde\bigtriangleup(W^{\bar{0}}_{*}(Z))$, then  we have the lemma below:
\begin{lemma}\label{l2}
$$\mathrm{H_{k}}(\widetilde{W}^{\bar{m}}_{*}(C(Z))\otimes\widetilde{W}^{\bar{n}}_{*}(C(Z)))=0,$$ for $k\geq i$.

\begin{proof}
Because of algebraic Kunneth formula :
 $$\mathrm{H}_{k}(\widetilde{W}^{\bar{n}}_{*}(C(Z)\otimes \widetilde{W}^{ \bar{n}}_{*}(C(Z))=\bigoplus\limits_{i}\widetilde{\mathrm{IH}}^{\bar m}_{i}(C(Z)) \otimes\widetilde{\mathrm{IH}}^{\bar n}_{k-i}(C(Z)).$$ we can compute $\mathrm{H_{k}}(\widetilde{W}^{ \bar{n}}_{*}(C(Z))\otimes\widetilde{W}^{\bar{n}}_{*}(C(Z)))$. 
When dimension of $Z<2s$, we know cone formula \ref{c1} $\widetilde{\mathrm{IH}}^{\bar m}_{i}(C(Z))=0$ for $i\geq  j- \bar m(j+1)$ . 
So when $k\geq i$, we can find the $\mathrm{H}_k(\widetilde{W}^{ \bar{n}}_{*}(C(Z)\otimes \widetilde{W}^{ \bar{n}}_{*}(C(Z))=0$.

For the diemsion of $Z$ is $2s$, the cone formula \ref{c2} show that:  $$\widetilde{\mathrm{IH}}^{\bar m}_{i}(C(Z))=\widetilde{\mathrm{IH}}^{\bar n}_{i}(C(Z))=0 \quad \text{for} \ i> s.$$ For $k>2s$, we have $\mathrm{H_{k}}(\widetilde{W}^{ \bar{n}}_{*}(C(Z)\otimes \widetilde{W}^{ \bar{n}}_{*}(C(Z))=0.$ 

When $k=2s$, if $$\mathrm{H}_{2s}(\widetilde{W}^{ \bar{m}}_{*}(C(Z)\otimes \widetilde{W}^{ \bar{n}}_{*}(C(Z))\neq 0,$$ 
then the nonzero generator is from $\mathrm{\widetilde{IH}}^{ \bar{m}}_{s}(C(Z))\otimes \mathrm{\widetilde{IH}}^{ \bar{n}}_{s}(C(Z))$. Let us assume the generator is $h\otimes g$,  here $h,g\in \mathrm{H}_{lag}$. However, consider the 2s-dimensional diagonal approximation  $\widetilde\bigtriangleup_{2s}$ in $Z$ , because of the product in Lagrangian Structure of $\mathrm{IH}^{m}_{s}(Z)$, there does not exist $h\otimes g$ in the image of $\widetilde\bigtriangleup_{2s}$. That is impossible. So we get $\mathrm{H}_{2s}(\widetilde{W}^{ \bar{n}}_{*}(C(Z)\otimes \widetilde{W}^{ \bar{n}}_{*}(C(Z))=0$ . 
\end{proof}
\end{lemma}
\subsection{Poincare Duality}
The original Poincare duality proof is based on the dual cell decomposition. In order to prove the Poincare duality is a chain equivalence in geometrically controlled category, we need see $\widetilde{W}^{ \bar{n}}_{*}(X)$ as the chain complex of geometric $\mathbb{C}-$ module \ref{geomodule} in this section.  


In pseudomanifold the dual cell structure $D(\sigma')$ is unique piecewise linear homomorphism to $\hat\sigma *\partial D(\sigma')= \hat\sigma*lk(\sigma, T')$.  If we decompose it further, we get $lk(\sigma, T')=S^k*L(X)$. $L(X)$ is intrinsic link unique up to piecewise linear homomorphism. 

With help of diagonal approximation, we can define the cap product for the chain complex $\widetilde{W}_{ \bar{n}}^{j}(X)$ of $n-$ dimensional $X$ by that:
\begin{equation}
   \widetilde{\cap}:\widetilde{W}_{ \bar{n}}^{j}(X)\otimes {W}^{0}_{n}(X) \xrightarrow{i \otimes \widetilde\bigtriangleup} {W}_{ \bar{n}}^{j}(X)\otimes (\widetilde{W}^{\bar{n}}_{*}(X) \otimes \widetilde{W}^{\bar{m}}_{*}(X))\xrightarrow{\varepsilon\otimes1}\widetilde{W}^{\bar{m}}_{n-j}(X). 
\end{equation}
here the $\varepsilon : \widetilde{W}^{\bar{n}}_{*}(X) \otimes \widetilde{W}_{\bar{n}}^*(X)\to \mathbb{C}$ is the evaluation map. Similar with manifold, we define the Poincare dual map $\widetilde{\mathbb{P}}$ is cap product with fundamental class $[X]$:
\begin{equation}
  \widetilde{\mathbb{P}}:=-\widetilde{\cap}[X]: \widetilde{W}_{\bar{m}}^{j}(X) \rightarrow \widetilde{W}_{n-j}^{\bar{n}}(X)  
\end{equation}

Before proving the Poincare duality, we need to prove the Mayer-Vietories sequence by using $\widetilde{W}^{\bar{m}}_{s}(X', T')$ instead of $ 
\widetilde{W}^{\bar{m}}_{s}(X, T)$. 
This is because the double subdivision $T''$ can give the definition of the boundary of two link .   

We know vertex of $T'$ is the barycenter of $\sigma$ in $T$, the simplicial structure of $(\mathrm{St}\hat\sigma,T'')$ is indentical $C(lk(\hat\sigma, T'))$ .
It is natural to use the filtration on the cone and cone formula defined before to study $\widetilde{W}_{i}^{\bar m}(X)|_{\mathrm{St} (\hat{\sigma},T'')}$. 
In particular, $\widetilde{W}_{i}^{\bar m}(X)|_{\mathrm{St} (\hat{\sigma},T'')}=\widetilde{W}_{i}^{\bar{m}}(C(lk(\hat\sigma, T')))$
\begin{lemma}
(Mayer-Vietories sequence) Assume $X$ is a non Witt space in \ref{nonwi1} with the Lagrange structure in the link of $X_0$ . $Y_1$ and $Y_2$ are closed subpseudomanifold of $X$, and $X=Y_1 \cup Y_2$. Then we have the short exact sequence:
$$0\to\widetilde{W}^{\bar{m}}_{*}(Y_1\cap Y_2)\stackrel{\iota\oplus\iota}\longrightarrow  \widetilde{W}^{\bar{m}}_{*}(Y_1)\oplus \widetilde{W}^{\bar{m}}_{*}(Y_2)\stackrel{\iota-\iota}\longrightarrow \widetilde{W}^{\bar{m}}_{*}(X)\to 0$$
\end{lemma}

The question is the inclusion map from $ \iota :Y_1 \to X$ is not a placid map because the codimension of $\iota ^{-1}$ of pure stratum in boundary.
Instead, we need to cover X by images of cones.  
Because we only change the $s+1$ dimensional chain, it is enough to consider the inclusion map of subcomplex involve stratum with non Witt condition. 

\begin{lemma} \label{p1}
If $X$ is the $n$ dimensional oriented Non Witt space defined in \ref{nonwi1} with the Lagrange structure in the link of $X_0$,  the generalized Poincare duality map $\widetilde{\mathbb{P}}$ \ from  \ $\widetilde{W}_{\bar{m}}^{*}(X)$ to $\widetilde{W}^{ \bar{n}}_{n-*}(X)$ 
$$\widetilde{\mathbb{P}}:=-\widetilde{\cap} [X]:\widetilde{W}_{\bar{m}}^{j}(X)\to\widetilde{W}^{\bar{n}}_{n-j}(X),$$
is a geometrically controlled chain equivalence.
\begin{proof} 
 We use the mathematical induction of dimension $i$ to prove this lemma. Dimension 0 case is clear. Suppose the Poincare duality holds for the $i<k$ dimension. Then consider dimension $k+1$, It is enough to show Poincare duality hold for $X$ is $k+1$ dimension pseudomanifold. 
 
 The standard proof is based on Mayer-Vietories argument, and the local dual map is guaranteed by Lagrange structure of the cone. 
 It is enough to prove when $\widetilde{\mathbb{P}}|_{\partial Y} $ is a chain equivalence then the relative chain map $\widetilde{\mathbb{P}}= -\widetilde{\cap} [Y]$
 $$\widetilde{\mathbb{P}}:\widetilde{\mathrm{IH}}_{m}^i(Y,\partial Y)\to \widetilde{\mathrm{IH}}^{\bar n}_{k+1-i}(Y).$$ is chain equivalence. Here the relative cap product are still defined as the composition above.
 \begin{equation}
\widetilde{\cap}:\widetilde{W}_{ \bar{n}}^{j}(Y,\partial Y)\otimes {W}^{\bar 0}_{n}(Y) \xrightarrow{i \otimes \widetilde\bigtriangleup} {W}_{ \bar{n}}^{j}(Y,\partial Y)\otimes (\widetilde{W}^{\bar{n}}_{*}(Y) \otimes \widetilde{W}^{\bar{m}}_{*}(Y))\xrightarrow{\varepsilon\otimes1}\widetilde{W}^{\bar{m}}_{n-j}(Y). 
 \end{equation}
 Suppose $Y$ is a star of simplex $\sigma,$ $Y=\hat \sigma * \partial (Y)$. 
 By assumption: $$\widetilde{\mathbb{P}}=-\widetilde{\cap} [\partial Y] :\widetilde{W}_j^{\bar m}(\partial Y ) \to \widetilde{W}^j_{\bar n}(\partial Y) $$  is chain equivalence for $\partial Y$. 
 Assume that $\omega$ and $\theta$ are basis of $\widetilde{\mathrm{IH}}^{\bar m}_{*}(\partial Y)$ and $\widetilde{\mathrm{IH}}^{\bar n}_{*}(\partial Y)$ respectively. $x$ and $y$ are basis of $\widetilde{W}^{\bar m}_{*}(\partial Y)$ and $\widetilde{W}^{\bar n}_{*}(\partial Y)$. 
 The image of $\widetilde\bigtriangleup[\partial Y]$ can be written as : 
 $$\widetilde\bigtriangleup[\partial Y]=\sum_{|\omega|+|\theta|=k-1} \omega \otimes \theta+\partial\left(\sum_{|x|+|y|=k} a_{x y} x \otimes y\right)$$

The first target is to prove $\widetilde\bigtriangleup[\partial Y]$ is the boundary of chain complex in .  It depends on the cone formula. 

When the dimension $k<2s$,  we use the original cone formula \ref{c1}, that means \begin{equation}
\widetilde{\mathrm{IH}}^m_{i}(Y)=
\begin{cases}
\mathrm{IH}^{m}_{i}(\partial Y) \quad  i<k-1-\bar m(k)  \\
0 \quad  otherwise.
\end{cases}\end{equation}
When $k=2s$ , we can use the cone formula such that :
\begin{equation}
\widetilde{\mathrm{IH}}^m_{i}(Y)=
\begin{cases}
\mathrm{IH}^{m}_{i}(Y) \quad  i<s\\
\mathrm{H}_{lag}\in \mathrm{IH}^{m}_{s}(Y) \quad i=s \\
0 \quad  otherwise.
\end{cases}   
\end{equation}
Because we know $|\omega|+|\theta|=k-1$,  if the dimension  of $\omega$ is $i\neq s$,  it is clear that only one of 
$\omega$ or $\theta$ is a boundary. When the dimension of $\omega$ is $s$, we know that Lagrange structure of $\partial Y$ make only one element in $\mathrm{H}_{lag}$ or $\mathrm{H}_{lag}^{\perp}$ is a boundary. 

If $\omega$ is a boundary, then let $\bar\omega$ be a chain such that $\partial\bar\omega = \omega$; if $\omega$ is not a boundary, then define $\bar\omega = \omega$. Define $\bar\theta$ in a similar way. In the chain complex $\widetilde{\mathrm{IH}}^{\bar m}_{*}( Y)\otimes\widetilde{\mathrm{IH}}^{\bar n}_{*}(Y)$  ), we have
$$ \sum_{|\omega|+|\theta|=k-1} \omega \otimes \theta=\partial\left(\sum_{|\widetilde{\omega}|+|\widetilde{\theta}|=k} \widetilde{\omega} \otimes \widetilde{\theta}\right).$$

Moreover, we get $\widetilde\bigtriangleup[\partial Y]$ is the boundary of $\sum \bar{\omega} \otimes \bar{\theta}+\sum a_{x y} x \otimes y$.
Because of the proposition of chain map  $\partial \widetilde\bigtriangleup[Y]=\widetilde\bigtriangleup[\partial Y]$. It concludes that the $\partial \widetilde\bigtriangleup[Y]$ is a boundary of $\sum \bar{\omega} \otimes \bar{\theta}+\sum a_{x y} x \otimes y$.

Next we combine the homology of lemma \ref{l5} to get :
$$ \widetilde\bigtriangleup[Y]=\sum_{|\bar\omega|+|\bar\theta|=j+1} \bar\omega \otimes \bar\theta+\sum_{|x|+|y|=j+1} a_{x y} x \otimes y $$
Then we can define the duality operator $\widetilde{\mathbb{P}}: \widetilde W_{\bar{p}}^{i}(Y,\partial Y) \rightarrow \widetilde W_{n-i}^{\bar{q}}(Y)$. 
Because of lemma  \ref{l2}, this map is the homology isomorphism. Because the quasi-isomorphism of free chain complex is chain equivalence, the dual map $\widetilde{\mathbb{P}}$ is chain equivalence. 

If we see the Link without Witt condition, the the symmetry of Lagrange structure make dual map exist on higher dimension. That is when $v\in V_{lag}^{\perp}$, then $\widetilde{\mathbb{P}}(v)\in C(V_{lag})$.  And when $v\in V_{lag}$, then $\widetilde{\mathbb{P}}(v) \in C(V_{lag}^{\perp})$.   

Recall Higson and Roe's argument in \cite{HR2} section 4, Mayer-Vietories sequence argument plays an important role to prove the Poincare dual map is geometrically controlled chain equivalence in bounded geometry combinatorial manifold. Similar, we can focus the star of a simplex $\sigma$ named $Y$ in $X$ to prove Poincare duality in geometrical control category because the map in the star of simplex is geometrically controlled. 

\end{proof}
\end{lemma} 
When we combine the lemma \ref{p1} and \ref{l3} together. For consistent with other notation in \cite{HR1}, we use $b$ for the boundary map of $\widetilde{W}^{\bar{n}}_{i}(X)$, and $b^*$ is adjoint of $b$ which is the differential of $\widetilde{W}_{\bar{n}}^{i}(X)$. Consider the diagram below:

 \begin{tikzcd}
& & &... \arrow{d} & ...  \arrow{d}{} &... \arrow{d}{}\\
 & & &\widetilde{W}_{\bar{n}}^{i}(X) \arrow{r}{\iota^{*}_{i}} \arrow[swap]{d}{b^{*}_{i+1}}
& \widetilde{W}_{\bar{m}}^i(X)\arrow{d}{b^{*}_i} \arrow{r}{\cap [X]} & \widetilde{W}^{\bar{n}}_{k-i}(X) \arrow{d}{b_i} \\
& & &\widetilde{W}_{\bar{n}}^{i+1}(X) \arrow{r}{\iota^{*}_{i+1}} \arrow{d}{}
& \widetilde{W}_{\bar{m}}^{i+1}(X) \arrow{d}{} \arrow{r}{\cap [X]} &  \widetilde{W}^{\bar{n}}_{k-i-1}(X)  \arrow{d}{} \\
& & & ... & ... &...
\end{tikzcd}\\

We know the $\widetilde{W}_{\bar{n}}^*(X)$ and $\widetilde{W}_{\bar{m}}^*(X)$ are chain equivalent, and the cap product with $[X]$ introduce chain equivalence. 
So $\widetilde{W}_{\bar{n}}^*(X)$ and $\widetilde{W}^{\bar{n}}_{k-*}(X)$ are chain equivalent.

When $X$ is a Poincare pseudomanifold, it is natural to define a Hilbert Poincare complex on $X$. Because we know the basis of chain complex $\widetilde{W}^{\bar{n}}_*$ give a inner product. We identify $(\widetilde{W}^{\bar{n}}_i,b^*)$ and $(\widetilde{W}_{\bar{n}}^{i},b)$ under this inner product. Let $\widetilde{\mathbb{P}}^{*}$ be the adjoint of $\widetilde{\mathbb{P}}$. We have two lemma below.
\begin{lemma}\label{pdp}
 $\widetilde{\mathbb{P}}^*$ is chain homotopy to  $(-1)^{i(n-i)} \widetilde{\mathbb{P}}$ in the geometrically controlled category. The two maps
 $$\widetilde{\mathbb{P}}:\left(\widetilde{W}_{\bar{n}}^{i}(X), b^{*}\right) \stackrel{\widetilde{\cap}[X]}{\longrightarrow}\left(\widetilde{W}_{n-i}^{\bar{n}}(X), b\right)$$
 and 
 $$\widetilde{\mathbb{P}}^{\prime}:\left(\widetilde{W}_{\bar{n}}^{i}(X), b^{*}\right) \stackrel{\widetilde\cap[X]}{\longrightarrow}\left(\widetilde{W}_{n-i}^{\bar{m}}(X), b\right) \stackrel{\iota}{\rightarrow}\left(\widetilde{W}_{n-i}^{\bar{n}}(X), b\right)$$
 are chain homotopic in the geometrically controlled category.
 \begin{proof}
The proof is similar with the Appendix C of \cite{X}, and all homotopy is from the unique of diagonal approximation. If we define $T: \widetilde{W}_{*}^{\bar{n}}(X) \otimes \widetilde{W}_{*}^{\bar{n}}(X) \rightarrow \widetilde{W}_{*}^{\bar{n}}(X) \otimes \widetilde{W}_{*}^{\bar{n}}(X)$ is chain isomorphism by :
\[T(x \otimes y)=(-1)^{|x||y|} y \otimes x\]
For $\alpha \in \widetilde{W}_{\bar{n}}^{i}(X)$, we can find that 
$\widetilde{\mathbb{P}}(\alpha)=\varepsilon(\alpha \otimes \widetilde{\Delta}[X])$ and  $\widetilde{\mathbb{P}}^{*}(\alpha)=\varepsilon(\alpha \otimes(T \circ \widetilde{\Delta})[X])$. Clearly, $T \circ \widetilde{\Delta}$ is still a diagonal approximation and $T \circ \widetilde{\Delta}$ is unique up to chain homotopy because of lemma \ref{diagonala}. Then $\widetilde{\mathbb{P}}^*$ is chain homotopy to  $(-1)^{i(n-i)} \widetilde{\mathbb{P}}$.

For the second statement, we just construct two diagonal approximation $\Delta: W_{*}^{\overline{0}}(X) \rightarrow \widetilde{W}_{*}^{\bar{n}}(X) \otimes \widetilde{W}_{*}^{\bar{n}}(X)$ and $\Delta: W_{*}^{\overline{0}}(X) \rightarrow \widetilde{W}_{*}^{\bar{m}}(X) \otimes \widetilde{W}_{*}^{\bar{n}}(X)$. It is clear $\Delta$ and $(1 \otimes \iota) \circ \Delta^{\prime}$ are homotopy equivalent diagonal approximation maps from $W_{*}^{\overline{0}}(X)$ to $\widetilde{W}_{*}^{\bar{n}}(X) \otimes \widetilde{W}_{*}^{\bar{n}}(X)$. Then we know $\widetilde{\mathbb{P}}^{\prime}(\alpha)=\varepsilon\left(\alpha \otimes(1 \otimes \iota) \circ \Delta^{\prime}[X]\right)$ and $\widetilde{\mathbb{P}}(\alpha)=\varepsilon(\alpha \otimes \Delta[X])$ by 
construction of Poincare map, it finish the proof. 
 \end{proof}
\end{lemma}
For the complex of inner space $(\widetilde{W}^{ \bar{n}}_{*}(X),b)$,  
then the complex $(\widetilde{W}^{ \bar{n}}_{*}(X),b,\widetilde{\mathbb{P}})$ is a geometrically controlled Poincare complex in the meaning of Higson and Roe in \cite{HR2}. Finally, we get the theorem.
\begin{theorem} If $X$ is the 2s+1 dimensional oriented Non Witt space defined in \ref{nonwi1} with the Lagrange structure in the link of $X_0$, then the Poincare dual map is chain equivalence for the chain $\widetilde{W}^{ \bar{n}}_{*}(X)$. Moreover,$(\widetilde{W}^{ \bar{n}}_{*}(X),b,\widetilde{\mathbb{P}})$ is a geometrically controlled Poincare complex.
\end{theorem} 
Banagl in \cite{BB} define a topological stratified pseudomanifold exists nontrivial self dual sheaf in $SD(\hat X)$ which call $L$-space. Here we use the same name. 
\begin{definition}\label{lspace}
If the non Witt space is a geometrically controlled Poincare pseudomanifold, it called $L$ space.
\end{definition}
In particular, because the chain map introduce a map for the homology group ${\widetilde{\mathrm{IH}}}^{\bar{n}}_*(X)$ of $\widetilde{W}^{\bar{n}}_*(X)$, this chain equivalent introduce the isomorphism ${\widetilde{\mathrm{IH}}}^{\bar{m}}_j(X)\cong {\widetilde{\mathrm{IH}}}^{\bar{n}}_{n-j}(X).$ 

The signature should be stratified homotopy invariant and cobordism with closed self dual Non Witt space. 

\section{$C^*-$ Algebraic Signature of $X$}\label{sign}
Next we can define $C^*-$ Algebraic Signature of $X$ with the procedure introduced in \ref{hilbertpc}.
Given the chain complex $\widetilde{W}^{\bar{n}}_{k-*}(X)$ with the canonical inner product, we identify $\widetilde{W}^{\bar{n}}_{*}(X)$ with $\widetilde{W}^{\bar{n}}_{*}(X)$, and recall $\widetilde{\mathbb{P}}^{*}$ is the adjoint of $\widetilde{\mathbb{P}}$. Because of lemma \ref{pdp}, then $T$ is :
$$T=\frac{1}{2}(\widetilde{\mathbb{P}}+(-1)^{(n-i) i}\widetilde{\mathbb{P}}^{*})$$  
then we can define $\widetilde{E}^{\bar{n}}_{*}(X)$ as the $\ell^2$ completion of $\widetilde{W}^{ \bar{n}}_{*}(X)$ in section \ref{geocon}. Then chain complex of Hilbert spaces $(\widetilde{E}^{\bar{n}}_{*}(X), b,T)$ is an analytically controlled Hilbert Poincare complex because of proposition \ref{geohp}. If we define $B=b+b^*$, and $S$ is defined as a bounded adjointable operator for every $v\in E_p$
$$S(v)=i^{p(p-1)+l} T(v).$$
Then $B+S$ is an invertible. For every Hilbert-Poincare complex of dimension n determines an signature as index class in \ref{signature}.

Next we consider the Hilbert Poincare complex over $C^*-$ algebra $C_r^*(\Gamma)$ for equivariant case. 
The key point is we use the local coefficient system  $\mathbb{F}=C_r^*(\Gamma)$ instead of $\mathbb{C}$ for the chain $(\widetilde{W}^{\bar{n}}_{*}(X)\otimes \mathbb{C}$. See the definition of coefficient system in \ref{coeff}. Here we obtain a chain complex of Hilbert $C_r^*(\Gamma)$- module. We define the Hilbert Poincare $C_r^*(\Gamma)$ module to be $E_{i}^{\bar{n}}(X,C_r^*(\Gamma))$

 The so called $C^*-$ algebraic higher signature $\operatorname{sign}_{\Gamma}(X, f) \in K_{n}\left(C_{r}^{*}(\Gamma)\right)$ is the signature of the Hilbert-Poincare complex $(E_{i}^{\bar{n}}(X ; C_r^*(\Gamma)),b,T)$.

Next we define the self dual L cobordant and stratified homotopy. In this dissertation, all the continuous map between is required to keep the Lagrange structure. Precisely, for any continuous map $f :X_1 \to X_2$,   the Lagrange structure $H_{lag,2}$ of $X_2$ is defined as push forward of $f$ for $H_{lag,1}$ in $X_1$.

Let $X_1$ and $X_2$ are two closed oriented L spaces with continuous maps $f_1 :X_1 \to B\Gamma$ and $f_2 :X_2 \to B\Gamma$. We say $X_1$ and $X_2$ are $\Gamma -$ equivariantly L-cobordant if there exist a L space with boundary W and a continuous map $\partial W=X_{1} \sqcup\left(-X_{2}\right)$ and $\left.F\right|_{X_{1}}=f_{1} \text { and }\left.F\right|_{X_{2}}=f_{2}$. Here we require the map $f_{1}$ and $f_2$ keep the Lagrange structure. Clearly, this L cobordism give the Hilbert Poincare pair, we know signature of Hilbert-Poincare complex is bordism invariant due to \ref{bordism1}. 

A stratified homotopy equivalence between X and Y is a homotopy equivalence in the category of codimension preserving maps. Here the stratum preserving map between two stratified spaces $\phi: X \to Y$ is for each stratum of $S$ of $Y$ , we have
$\mathrm{codim}(\phi^{-1}(S)) = \mathrm{codim}(S)$.  The stratified homotopy equivalence of L space $\phi$ induce a chain equivalence of $W^{ \bar{n}}_{*}(X)$ and $W^{ \bar{n}}_{*}(Y).$ Because of the definition , it is a homotopy equivalence of Hilbert Poincare complex between $ (\widetilde{W}^{ \bar{n}}_{*}(X),b,\widetilde{\mathbb{P}})$ and $(\widetilde{W}^{ \bar{n}}_{*}(Y),b,\widetilde{\mathbb{P}})$.

The next theorem is direct result of higher signature on Hilbert Poincare complex. 
\begin{theorem}\begin{enumerate}
\item The $C^*-$ algebraic higher signatures $\operatorname{sign}_{\Gamma}(X, f) \in K_{n}\left(C_{r}^{*}(\Gamma)\right)$ of L spaces $X$ are invariant under L cobordism. 
If $X_1$ and $X_2$ are n dimensional two closed oriented L spaces with continuous maps $f_1 : X1 \to B\Gamma$ and $f_2 : X_2 \to B\Gamma.$ Suppose $X_1$ and $X_2$ are $\Gamma$-equivariantly L-cobordant, then
$$\operatorname{sign}_{\Gamma}\left(X_{1}, f_{1}\right)=\operatorname{sign}_{\Gamma}\left(X_{2}, f_{2}\right).$$ 

\item $C^*-$ algebraic higher signatures of L spaces are invariant under stratified homotopy equivalences which keeps the Lagrange structure. Suppose $X$ and $Y$ are two closed oriented L spaces, and $f:Y \to B\Gamma $ is a continuous map. If $\phi :X\to Y$ is a stratified homotopy equivalence and keep the Lagrange structure, then
$$\operatorname{sign}_{\Gamma}(X, f \circ \phi)=\operatorname{sign}_{\Gamma}(Y, f)$$

\end{enumerate}
\end{theorem}
\subsection{Homotopy Invariance under Subdivision}
In this subsection, I will show $\widetilde{W}_{\bar{n}}^*(X)$ and $\widetilde{W}_{\bar{m}}^*(X)$ are geometrically controlled homotopy invariant under a subdivision of M in \ref{sub1}. We follow the framework in Section 6 and 7 of \cite{X}. Then we can use the procedure of  $Sub^n(M)$ to construct $K-$ homology class in $X$.

Given a triangulation $T$ of $X$. We denote $\widetilde{W}_{\bar{m}}^*(X; T, T')$ as the geometrically controlled Poincare complex based on the  barycentric subdivision of $T$. 
Now assume $S$ is a subdivision of $T$. We construct the geometrically controlled Poincare complexes $\widetilde{W}_{\bar{n}}^*(X; T, S')$  and $\widetilde{W}_{\bar{n}}^*(X; T, S')$. Consider the inclusion chain maps below:
$$\tau_1: \widetilde{W}_{\bar{n}}^*(X; T, T') \to \widetilde{W}_{\bar{n}}^*(X; T, S') $$
and
$$\tau_2: \widetilde{W}_{\bar{n}}^*(X; S, S') \to \widetilde{W}_{\bar{n}}^*(X; T, S') $$
\begin{lemma}
  $\tau_1$ and $\tau_2$ are geometrically controlled chain equivalences.
\end{lemma}
The proof is still an argument of Mayer-Vietoris sequence. Let $Y$ be a star of $\sigma$ in $T$. Because the cone formula in \ref{c1} and \ref{c2}, an isomorphism of homology in cone induce chain equivalence between $\widetilde{W}_{\bar{n}}^*(Y; T, T')$ and $\widetilde{W}_{\bar{n}}^*(Y; S, S')$.

\begin{corollary}
If $X$ is a closed oriented L space, then the two geometrically controlled Poincare complexes  $\widetilde{W}_{\bar{n}}^*(X; T, T')$  and $\widetilde{W}_{\bar{n}}^*(X; S, S')$  are geometrically controlled chain homotopy equivalent.
\end{corollary}
Suppose $\tau:  \widetilde{W}_{\bar{n}}^*(X; T, T') \to \widetilde{W}_{\bar{n}}^*(X; S, S') $ is the chain map, and $\tau*$ is dual of $\tau.$ We need to prove $\tau$ preserve the Poincare dual. Given triangulation $T$, let $\mathbb{P}_{T}(\alpha)=\varepsilon\left(\alpha \otimes \Delta_{T}[X]\right)$ be the Poincare dual map from $\alpha \in W_{\bar{n}}^{i}\left(X ; T, T^{\prime}\right)$. Similarly, 
$\mathbb{P}_{S}(\beta)=\varepsilon\left(\beta \otimes \Delta_{S}[X]\right)$ is Poincare dual map for triangulation $S$. Then it is enough to observe that $\tau \circ \mathbb{P}_{T} \circ \tau^{*}(\beta)=\varepsilon\left(\beta \otimes\left(\tau \circ \Delta_{T}\right)[X]\right)=\varepsilon\left(\beta \otimes \Delta_{S}[X]\right)=\mathbb{P}_{S}(\beta)$.

Assume $\mu_i^k$ is the basis of $\widetilde{W}^{\bar{m}}_{k}(X)$ which is supported on the star of vertex $v_i^k$. Then it can define an $X-$ module on $\mathcal{H}_0$ by acting on the vertex $f\cdot\mu_{i}^{k}=f(v_{i}^{k})\mu_{i}^{k}$.  
\begin{equation}
    \mathcal{H}_{0}=\bigoplus_{k} \widetilde{W}_{k}^{\bar{m}}(X) \otimes \mathbb{C}
\end{equation}
For each non negative integer $i$, define $Sub^i(M)=Sub(Sub^{i-1}(M))$, we can define $\mathcal{H}_{k}$ based on the successive refinements:
\begin{equation}
    \mathcal{H}_{k}=\bigoplus_{k} \widetilde{W}_{k}^{\bar{m}}(Sub^k(X)) \otimes \mathbb{C}
\end{equation}
Let $\mathcal{H}$ be the $\ell_2-$ completion of $\oplus_{i=0}^{\infty} \mathcal{H}_{i}$. Then $\mathcal{H}$ is an ample nondegenerate $X$-module by inheriting $X$-module structure from $\mathcal{H}_{k}$. Because the geometrically controlled Poincare complex is homotopy equivariant with subdivision. Then it is possible to use the uniformly bounded subdivision to control the propagation.  Next is the standard construction of K-homology class in \cite{XX} Appendix B. There is an outline in section \ref{khl}.
\begin{definition} 
The K-homology class of the signature operator of L space $X$ is defined to be the K-theory class of the path $U$ in $K_1(C_L^*(X))$. K-homology class is $[D_{\mathrm{sign}}]$ .
\end{definition}

\begin{definition} 
The K-homology class of the signature operator on L space $X$ is defined to be the K-theory class in $K_0(C_L^*(X))$ determined by $Q$: a norm-bounded and uniformly continuous path of $\sigma$-quasi-projections $[0,\infty] \to C^*(X)$. 
\end{definition}
\clearpage
\section{Connection with others Research}
Albin, Leichtnam, Mazzeo and Piazza in \cite{AA} introduce the self dual mezzoperversity to construct the signature package in Cheeger space. The authors above and Banagl in \cite{B} redefine a self dual intersection homology in Non Witt space. Banagl in the book \cite{BB} construct a signature and L class in Non Witt space. In these construction, he talks about the obstruction of signature existence on the Non-Witt space.  Although the self-dual sheaf in \cite{BB} is not stratified homotopy, it is independent of choice for Lagrangian structure for $L$ class. 

\subsection{Hilbert Poincare complex on Cheeger space}
In \cite{HR2} , Higson and Roe prove the analytical Hilbert Poincare complex $\left(\Omega_{L^{2}}^{*}(X), d\right)$ based on Hodge de Rham complex and $\left(C_{*}^{\ell^{2}}(X), b^{*}\right)$ from $\ell^2$ simplicial chain complex are chain equivalent  for combinatorial manifold. The similar result are generalized to Witt space with conical singularity in \cite{X}.
In this section,  we will prove Hilbert Poincare complex where \cite{AA} defined is controlled chain homotopy equivalent to Hilbert Poincare complex $(E_i(X),b) $ for the pseudomanifold with conical singularity. Then the two signature is equivalent.

In \cite{cheeger1979spectral}, Cheeger consider the 
Cheeger boundary condition carries the Cauchy information of the link, he prove a condition of self dual $L^2$ cohomology on pseudomanifold with conical singularity. In this section, I will follow Albin, Leichtnam, Mazzeo, and Piazza's framework to explain analytical signature on smooth stratified pseudomanifold \cite{A}.
Let us consider the stratified space $X$ with only one singular stratum $Y$. 
The resolution $\bar X$ of $X$ is obtained by blowing up each of conical fibers at its vertex. 

A smooth manifold $\bar X$ with boundary $\partial \bar X$,  a fibration of $\partial \bar X$ over $Y$ with fiber $Z$.
$$Z-\partial \bar X \to Y$$
Let $$\mathrm{mid }= \frac{\mathrm{dim} Z }{2}$$
In the complete edge metric is defined (also called iterated incomplete edge, or iie, metrics), 
$${dx^2}+x^2*g_{Z}+\phi^*(g_Y)$$
Thus sections of $\prescript{iie}{}{T^{*}X}$ are locally spanned by
$dx$, $dy$, $xdz$. The wedge operator is incomplete edge operator and we need to consider the complete extension. The minimal extension is $\mathcal{D}_{min}(d)$ and the maximum extension is $\mathcal{D}_{max}(d)$. All closed extension is between these two extensions. 
\begin{equation}
\mathcal{D}_{\max }\left(d\right)=\left\{\omega \in L^{2}\left(X ; \Lambda^{*}(\prescript{iie}{}{T^{*}X}) X\right): d\omega \in L^{2}\left(X ; \Lambda^{*}(\prescript{iie}{}{T^{*}X}) X\right)\right\}
\end{equation}
\begin{equation}
\begin{split}
\mathcal{D}_{\min }\left(d\right) &=\{\omega \in L^{2}\left(X ; \Lambda^{*}(\prescript{iie}{}{T^{*}X}) X\right): \\
& \exists\left(\omega_{n}\right) \subseteq \mathcal{C}_{c}^{\infty}\left(X ; \Lambda^{*} (\prescript{iie}{}{T^{*}X})\right) \text { s.t. } \omega_{n} \stackrel{L^{2}}{\longrightarrow} \omega \text { and }d\omega_{n} \text { is } L^{2} \text { -Cauchy }\}
\end{split}
\end{equation}

The $L^2$ differential form with $\mathcal{D}_{\max}$ or $\mathcal{D}_{\min}$ is a Hilbert complex \cite{A}. We know the de Rham cohomology $H^i_{DR}$ is isomorphic to simlicial homology $H_i$ for smooth manifold. Similarly, $L^2$ cohomology with $\mathcal{D}_{\min}$ is dual of upper middle perversity intersection homology group, and $L^2$ cohomology with $\mathcal{D}_{\max}$ is dual of lower middle intersection homology
For Witt space, the two extension is same. 

For non Witt space, Here the formulation of boundary condition is Cauchy data map. A local ideal boundary condition is a bundle homomorphism, which is empty when X is Witt space. It is corresponding to Lagrange structure of middle homology. 

The Mezzoperversity (flat structure) require the global condition which means flat choice the subbundle repsect to the connection $\bigtriangledown^H$. If $\mathcal{W}$ is a flat subbundle:
$$\mathcal{W} \longrightarrow \mathrm{{H}}^{\operatorname{mid}}(\partial X / Y) \longrightarrow Y$$
then Cheeger ideal boundary conditions from $\mathcal{W}$ is:
$$\mathcal{D}_{\mathcal{W}}(d)=\left\{\omega \in \mathcal{D}_{\max }(d): \alpha\left(\omega_{\delta}\right) \text { is a distributional section of } W\right\}.$$ A realization of de Rham operator in the ideal Cheeger condition is closed and self-adjoint Fredholm operator. Let $mathscr{D} \mathcal{W}$ is the  orthogonal complement of $W$ . Then the Poincare dual map is $$Q: \mathrm{H}_{\mathcal{W}}^{*}(\widehat{X}) \times \mathrm{H}_{\mathscr{D} \mathcal{W}}^{*}(\widehat{X}) \longrightarrow \mathbb{R}$$

\begin{definition}\cite{AA}
 A mezzoperversity $W$ is self-dual if $\mathcal{W} = \mathscr{D} \mathcal{W}$. A pseudomanifold $X$ with a self-dual mezzoperversity is a Cheeger space. 
\end{definition}
\begin{remark}
Banagl in \cite{BB} define a topological stratified pseudomanifold exists nontrivial self dual sheaf in $SD(\hat X)$ which call $L$-space. In the proposition 4.3 of \cite{B}, A smoothly stratified pseudomanifold is a Cheeger space if and only if it is an $L$-space.  The requirement of pseudomanifold  with conical singularity to be the Cheeger space is the signature of the link is 0.

\end{remark}
Let $(\hat X,g,B)$ be a Cheeger space with a self-dual mezzoperversity $\mathcal{W}$, and $B$ are the associated Cheeger ideal boundary conditions with $\mathcal{W}$. The signature operator
$\eth_{\text {sign }}^{+}$ is closed and Fredholm in Theorem 6.6 \cite{A}, then the Fredholm index is the signature of the  quadratic form with generalized Poincare duality
$$\mathcal{D}_{\mathbf{B}}\left(\eth_{\mathrm{sign}}^{\pm}\right)=\mathcal{D}_{\mathbf{B}}\left(\eth_{\mathrm{dR}}\right) \cap L^{2}\left(X ; \Lambda_{\pm}^{*}\left({ }^{\mathrm{iie}} T^{*} X\right)\right)$$

 For the $u_q\cong\mathbb{B} \times C\left(Z_{q}\right)$, because of Proposition 7.1 in \cite{A},  we have
\begin{equation}
    \mathrm{H}^{k}\left(\left.d\right|_{\mathcal{U}_{q}}, \mathcal{D}_{\mathbf{B}}\left(\left.d\right|_{\mathcal{U}_{q}}\right)\right)=\left\{\begin{array}{ll}\mathrm{H}_{\mathbf{B}\left(Z_{q}\right)}^{k}\left(Z_{q}\right) & \text { if } k<\frac{1}{2} \operatorname{dim} Z \\ W\left(Z_{q}\right) & \text { if } k=\frac{1}{2} \operatorname{dim} Z \\ 0 & \text { if } k>\frac{1}{2} \operatorname{dim} Z\end{array}\right.
\end{equation}
Because the signature of $Z_q$ is 0, $W(Z_q)$ is the Lagrange space of $H_k(Z_q)$. This means the cone formula is same with \ref{c2}.
Let the Laplacian operator $\Delta=\eth_{\text {sign }}^2$, and heat operator is $e^{-\Delta}$.

Next argument is basically same with proof of Hilbert Poincare complex equivalent of section 5 in \cite{X}. If we just consider the punctured cone $C_{0,1}(N)=(0,1) \times N$ with conic Riemann metric $d r^{2}+r^{2} g_{N}$ . 

Suppose the i-form $\theta$ in $C_{0,1}(N)$ is $\theta=g(r) \phi+f(r) d r \wedge \beta.$ Because of \cite[Theorem 3.1]{cheeger1979spectral}, we know
if the $\phi_{i}$ and $\psi_{j}$ are the orthonormal basis of harmonic $k$-form $H^k(N)$ satisfies the boundary condition $\pi_{\mathcal{H}^{k}(N)}[\theta(r, x)]=\Sigma f_{i} \phi_{i}+\Sigma g_{j} \psi_{j}$ with $f_{i}^{\prime}(0)=g_{j}(0)=0$.
Let $\phi_{i}$ be the basis of Lagrange space $V_{a}$ in the sense of \cite[formula 3.16]{cheeger1979spectral}. We need to map $g(r) \phi_i$ to the basis of $V_{lag}$. For $f(r)dr\wedge \omega$, it is in the image of $\mathcal{H}_{\Delta}$ in \cite[lemma 6.1]{X}, we do not need to worry that. The map 
 is defined by $\xi \mapsto \int_{\xi} e^{-\Delta} \omega$ in \cite[lemma 6.3]{X}. Because of \cite[lemma 6.5]{X}, it can define the chain map
$\Psi: $ $\left(L_{\mathrm{iie}}^{2}\left(X ;{ }^{\mathrm{iie}} \Lambda^{*} X\right), \mathcal{D}_{\mathcal{W}}(d)\right) \rightarrow\left(\widetilde{E}_{\bar{n}}^{*}(X), b^{*}\right)$ is a controlled chain homotopy equivalence.

\begin{theorem}
$\Psi^*$ is a controlled homotopy equivalence from the Hilbert Poincare complex \\$\left(L_{\mathrm{iie}}^{2}\left(X ;{ }^{\mathrm{iie}} \Lambda^{*} X\right), \mathcal{D}_{\mathcal{W}}(d)\right)$ to Hilbert Poincare complex $\left(\widetilde{E}_{\bar{n}}^{*}(X), b^{*}\right).$
\end{theorem}
  
\subsection{Refined Intersection Homology }
In this section I will introduce basic idea of sheaf in intersection homology. Then I will prove refined intersection homology in \cite{BB} is equivalent to my construction in some case. 

This based on the \cite{G2}. 
The stratification of n-dimensional $X$ is:
$$X=X_n \supset  X_{n-2} \supset X_{n-3} ... \supset X_1\supset X_0,$$
define $U_k=X\backslash X_{n-k}$, $Y_k=X_k\backslash X_{k-1}$ then  we have 
$$U_k\mathop\rightarrow\limits^{i_k} U_{k+1}\mathop\leftarrow\limits^{j_k}Y_{n-k} $$
for the sheaf $S$ of $U_{k+1}$, we have the exact sequence:
$$0\to j_{!}j^*S\to S\to i_*i^*S\to 0 $$
then the distinguished triangle is :
$X\mathop\rightarrow\limits^{f} Y$ $A$ is the sheaf on $X$ and $B$ is the sheaf on $Y$ then we have $B\to f_*f^*B$ and $f^*f_*A\to A$.

The basic idea of development of derived category is that complexes is better to handle than cohomology. In the derived category, any quasi isomorphism which is isomorphism of homology.
Let $\mathbf{A}^{\bullet}\in D^b_c(X),$ $U_x$ is a small distinguished open neighborhood of $x$. We have the hypercohomolgy:
$$\begin{aligned} \mathbf{H}^{i}\left(\mathbf{A}^{\bullet}\right)_{x} & \cong \mathcal{H}^{i}\left(U_{x} ; \mathbf{A}^{\bullet}\right) \\ \mathbf{H}^{i}\left(j_{x} ! \mathbf{A}^{\bullet}\right) & \cong \mathcal{H}_{c}^{i}\left(U_{x} ; \mathbf{A}^{*}\right) \end{aligned}$$

By Goresky and MacPherson's definition 3.3 in \cite{G2}, the intersection homology sheaf $\mathrm{AX_p}[\mathbb{S}]$  is the $\mathbb{S}$-constructible sheaf and satisfies four axioms.
(a) Normalization (b) lower bound (c) Vanishing condition: $H^m(s^*_{k+1})=0$ for all $m> \bar p(k)- n$. (d) Attaching
For support dimension and cosupport dimension is related with local intersection homology. The stalk is the colimit of $\mathcal{F}(U)$ and limit of $\mathcal{F}(U)$ is costalk. 

For any $x$ in pseudomanifold, the neighbourhood $N_x$ is homemorphic to  $N_x\cong R^{m-k}\times C(L_S)$. 
Let us focus on the pseudomanifold with isolated singularity
 $$\mathrm{I^{\bar p}H}_{-i}^{cl}\left(N_{x}\right) \cong \mathrm{I^{p} H}_{-i}^{c l}\left(\mathbb{R}^{m-k} \times C\left(L_{S}\right)\right) \cong \mathrm{I^{p} H}_{k-m-i}^{c l}\left(C\left(L_{S}\right)\right)$$
So here the costalk vanish is connectting with cone formula : $H^{i}\left(\jmath_{x}^{ !} \mathcal{I}^{p} \mathcal{S}_{X}^{\bullet}\right) \cong \mathrm{IH}^{\bar p}_{-i}\left(C\left(L_{S}\right)\right)$

When $X$ is a pseudomanifold with conical singularity, we will show that the modified lower middle perversity chain $\widetilde{W}_{\bar{n}}^*(X)$ is satisfied the axioms RP below.
\begin{definition} {[definition 2.1 in \cite{BB}]}
Let $X$ be a n dimensional stratified oriented pseudomanifold and $S$ a constructible bounded complex of sheaves. $S$ is a refined middle perversity complex of sheaves if satisfies the axioms $RP$ 
\begin{itemize}
\item Normalization: there is an isomorphism of the restriction of $S$ to the regular part $U_2$ of X and the constant rank 1 sheaf over $U_2$
\item Lower bound: $\mathrm{H}^{\ell}\left(j_{x}^{*} \mathbf{S}^{\bullet}\right)=0$ for any  $x \in \widehat{X}$ and $\ell<0$
\item $\tilde n$-stalk vanishing: $\mathrm{H}^{\ell}\left(j_{x}^{*} \mathbf{S}^{\bullet}\right)=0$ for any $x \in U_{k+1}\setminus U_2$ and $l > \bar n(k)$.
\item $\tilde m$-costalk vanishing: $\mathrm{H}^{\ell}\left(j_{x}^{!} \mathbf{S}^{\bullet}\right)=0$ for any $x \in X_{n-k}\setminus X_{n-k-1}$ and $l \leq \bar m(k) + n -k+ 1.$
\end{itemize}

\end{definition}

Because of Goresky and MacPherson’s {[AX1]}, Normalization and lower bound is satisfied for intersection homology respect with any perversity. Here the modified intersection homology chain is the interpolation so it is compatible with normalization and lower bound. $\tilde n$-stalk vanishing and $\tilde m$-stalk vanishing are self-dual by axioms. 

In fact the cone formula of $\widetilde{\mathrm{IH}^m_{i}}(X)$ is related with stalk vanish 
$$\widetilde{\mathrm{IH}^m_{i}}(C(L_s))=
\begin{cases}
\mathrm{IH}^{m}_{i}(L_s) \quad  i<s \\
\mathrm{H}_{lag}\in \mathrm{IH}^{m}_{s}(L_s) \quad i=s \\
0 \quad  otherwise.
\end{cases}
$$
We know  $\widetilde{\mathrm{IH}^m_{i}}(C(L_s))=0$ 

Then interpolated chain $\widetilde{W}_{\bar{n}}^*(X)$  I defined in this chapter satisfies axioms SD1 to SD4 in \cite{BB}.  This is equivalent to self dual sheaf in $SD(\hat X)$ of \cite{BB}. Then because of the Theorem 4.1 in \cite{B}, it is the redefined middle perversity complex of sheaves.

\section{Geometrically controlled Poincare complex on general non Witt space \label{cha:Summary}}
In this chapter, we will show Non Witt space $X$ can be a geometrically controlled Poincare pseudomanifold if there exist $\widetilde{W}^{\bar m}_*(X)$. The same framework in section \ref{sign} can be applied to construct the $C^*-$ signature on $X$. In fact, the method of construct the self dual chain Non Witt space with only one special stratum can apply to general case.
\subsubsection {Non Witt Space with only One Special Stratum }\label{ssnw}
In this section, we consider the $n$-dimensional pseudomanifold $X$ with one special singular $n-2s-1$ dimensional stratum $\chi_{n-2s-1}$. For each point $x$ of $\chi_{n-2s-1}$, the intrinsic link 
 is written as  $L(\chi_{n-2s-1},x)$. 
 For the Non-Witt space with only one exception stratum, it means  there is a point $x$ such that the lower middle perversity intersection homology about  $L(\chi_{n-2s-1},x)$ is not trivial
$$\mathrm{IH}^m_{s}(L(\chi_{n-2s-1},x),Q)\neq 0.$$ 
For all other odd codimensional stratum $\chi_{n-2i-1}$, here $i\neq s$,  the pseudomanifold  has trivial lower middle intersection homology for the link :
 $$\mathrm{IH}^m_{j}(L(\chi_{n-2i-1},x),Q)= 0.$$
 Because $X$ is not Witt space, and we can get the inclusion map from $W^{\bar m}_*(X)$ to $W^{\bar n}_*(X)$ is not chain equivalent. We still use the triangulation which is finer than the stratification. The general method still decompose the difference space between $W^{\bar m}_*(X)$ and $W^{\bar n}_*(X)$. Then construct new chains:
 \begin{equation}
  W^{\bar{m}}_j(X)\subset \widetilde{W}^{\bar{m}}_j(X)\subset \widetilde{W}^{\bar{n}}_j(X) \subset W^{\bar{n}}_j(X)   
 \end{equation}

In this section, we still use the technique from Xie and Higson's article Appendix C  in \cite{X}.
First, we define the perversity $P_r$:
\begin{equation} \label{p}
       P_r(l)=\left\{
                \begin{array}{ll}
                  \bar{m}(l) ,  \quad  l<r+1\\
                \bar{n}(l) , \quad l>r
                \end{array}
              \right.
\end{equation} 
Next, let r be the maximum integer such that $2r+1\leq n$. We consider the filtration 
$$W^{\bar{m}}_*(X) =W^{P_{2r+1}}_*(X) \subset W^{P_{2r-1}}_*(X) ...\subset W^{P_{1}}_*(X)=W^{\bar{n}}_*(X).$$ 
Because for any other odd codimension $i\neq s$, we have
 $\mathrm{IH}^m_{l}(L(\chi_{n-2i-1},x),Q)= 0$.
 Then $W^{P_{2i-1}}_*(X)$ and $W^{P_{2i+1}}_*(X)$  
  are chain equivalent according to \cite{S} Chapter III, Theorem 3.2. 
 So we get 
  $$W^{\bar{n}}_*(X) \sim W^{P_{2s-1}}_*(X)  \quad \text{and} \quad W^{\bar{m}}_*(X) \sim W^{P_{2s+1}}_*(X).$$ 

The question is to decompose the $W^{P_{2s-1}}_*(X)$ and $W^{P_{2s+1}}_*(X)$. When $j\neq 2s+1$, then $P_{2s+1}(j)=P_{2s-1}(j)$, so the allowable requirement of the chain complex $W^{P_{2s+1}}_*(X)$ and $W^{P_{2s-1}}_*(X)$ are same except for codimension $2s+1$.
 When $y\in W^{P_{2s-1}}_j(X)$, consider the allowable inequality regarding $\chi_{n-2s-1}=Y$: 
$$\mathrm{dim}(y\cap Y)\leq j-(2s+1)+P_{2s-1}(2s+1)=j-s-1.$$
The stronger restriction about $y\in W^{P_{2s+1}}_j(X)$ is: $$\mathrm{dim}(y\cap Y)\leq j-(2s+1)+P_{2s+1}(2s+1)=j-s-2.$$
Clearly,  for $j<s+1$ and $j>n-s+1$, the inequalities are same and $W^{P_{2s+1}}_j(X)=W^{P_{2s-1}}_j(X)$. In this case, we only need to consider the dimension $s<j<n-s+1$.  We need to define $\widetilde W^{p_{2s+1}}_j(X)$ and $\widetilde W^{p_{2s-1}}_{j}(X)$ based on  $W^{p_{2s+1}}_j(X)$ and $W^{p_{2s-1}}_{j}(X)$ such that: 
$$ W^{p_{2s+1}}_*(X)\subset\widetilde W^{p_{2s+1}}_*(X)\subset\widetilde W^{p_{2s-1}}_{*}(X)\subset W^{p_{2s-1}}_*(X) $$
Moreover the inclusion map from $\widetilde W^{p_{2s+1}}_j(X)$ to $\widetilde W^{p_{2s-1}}_{j}(X)$ is chain equivalent.

Let us fix the direct sum decomposition when $s<j<n-s+1$ \begin{equation}\label{dec}
    W^{P_{2s-1}}_{j}(X)=W^{P_{2s+1}}_{j}(X)+U^j.
\end{equation} 
Next, we want to build the connection between the $s$-dimensional cycle $v$ in the intrinsic link $L(Y,y)$ of the stratum $Y$ and $U^j$. For the open stratum $Y=\chi_{n-2s-1}$, it is the disjoint union of the interior of $n-2s-1$ dimensional simplex $\sigma^{n-2s-1}$. Let $u^j_i$ be basis of $U^j.$ Every $u^j_i$ is minimal and supported on the cone. 
$$\sum u^j_i=U^{j}.$$ 
Similar to the case of conical singularity, let $\sigma^{n-2s-1}_i$ be the simplex that intersect $u^j_i$ with this condition:
$$\mathrm{dim}(|u^j_i|\cap \sigma^{n-2s-1}_i)=j-s-1.$$ 

 Let $T$ be the triangulation for defining $W^{\bar m}_{j}(X)$. Define $T'$ as first barycentric subdivision of $T$.  Let $\tau^{j-s-1}_i$  be $j-s-1$ simplex of $T'$ such that:
 $$\operatorname{Int}\left(\tau^{j-s-1}_{i}\right) \subset|u^j_i| \cap \operatorname{Int}(\sigma^{n-2s-1}_{i})$$
We define $lk(\sigma_i,T')$ as the link of $\sigma^{n-2s-1}_i$ in $T'$. We can see $u^j_i$ as the join complex with $\tau^{j-s-1}_i$ and $s$-simplex $v_i\in C^{T'}_{s+1}(lk(\sigma_i,T'))$:

\begin{equation} \label{eq1}
    u^j_i=\tau^{j-s-1} _i*v_i.
\end{equation}

In \cite{S} Chapter III Lemma 3.3, Siegel shows $\partial u_i\cap \mathrm{Int}(\sigma_i)$ does not contain $\mathrm{Int}(\tau_i)$. In other words : $$\partial({v_i})=0. $$ 
Because we can build a isomorphism between cycle $\bar{v}_i\in C^{T'}_{s}(lk(\sigma^{n-2s-1}_i,T)')$ and $v_i$. Here $\bar v_i$ is the chain complex in $W^{\bar m}_*(lk(\sigma, T))$. So $\bar{v_i}$ is cycle of the link $C^{T'}_{s}(lk(\sigma^{n-2s-1}_i,T'))$.  We can rewrite the cycle $\bar{v_i}$ as addition of  the boundary of chain $\bar{w}$: $$\bar{w} \in W^{\bar m}_{s+1}(lk(\sigma^{n-2s-1}_i,T)')$$ and the Representative of $\bar{h} \in \mathrm{IH}^{\bar m}_{i}(lk(\sigma_i,T)')$
 such that $$\bar{v_i}=\partial (\bar{w})+\bar{h}.$$ 
 
 Hence we just consider the corresponding $\bar v_i$ to take place of $v_i$.If $c_i$ is a barycenter of $\sigma_i$,  we define the space $V(c_i)$ to be a set of all $s$-cycle in $W^{\bar m}_*(lk(\sigma_i, T))$.
 
In the conical singular case, we decompose the space of $s$-cycle $V$ in the intrinsic link as: $$V=V_{lag}\oplus V^{\perp}_{lag}\oplus V_o.$$ 
Similar, our target is the decomposition of $V$ and  furthermore $U^j$.

\subsubsection{Compatible Lagrange Structure}\label{compat}
In this section, we need to construct a compatible subspace of $V$ such that the space $U^j$ exist a subbundle The Lagrange structure is a Lagrange subspace of the middle intersection homology $\mathrm{IH}^{\bar m}_s(L)$ of the intrinsic link $L(x)$ for every point.  

In general, the existence of Lagrange structure is not simple when the dimension of singular stratum is not 0 ( conical singularity case). We know for each point $y$ of $Y$, (here $L(y)$ is the intrinsic link ) then $X$ can be seen as the fiber bundle over $Y$. The middle dimensional homology of the link is the vector bundle in the base space of $Y$.

Let us go back to \ref{eq1}, for $\tau^{j-s-1} _i*v_i$, we know $\bar{v_i}$ is the s-cycle of the link $C^{T'}_{s}(lk(\sigma^{n-2s-1}_i,T'))$. We need to consider the connection between $lk(\sigma^{n-2s-1}_i,T)$ for different $\sigma^{n-2s-1}_i$. We know that stratum $\chi_{n-2s-1}$ is disjoint union of interior of $n-2s-1$ dimensional simplex $\sigma^{n-2s-1}$. Let  the barycentric of  $\sigma^{n-2s-1}_i$ be $c_i$. We define the intrinsic link of $L(c_i)$  as the link in simplicial meaning where  the dimension of the points is higher than $n-2s-1$. What's more, for any simplex $\tau$ of $\sigma^{n-2s-1}_i$
$$L(\tau)=\{\mathscr{V}|\mathscr{V}\in lk(\tau,T) \ ,\ \text{and} \ \mathrm{dim}(\mathscr{V})>n-2s-1 \}$$
Clearly we use the intrinsic link of the barycentric to substitute  $ lk(\sigma^{n-2s-1}_i,T)'$ because $L(c_i)\cong lk(\sigma^{n-2s-1}_i,T)'$. In addition, the intrinsic link of any point in $\sigma^{n-2s-1}_i$ is $L(c_i)$.
Moreover, if the $dim(\tau)=j<n-2s-1$, we can find $$L(\tau)\cong S^{n-2s-2-j}*L(\tau).$$ 

Let $\sigma^{n-2s-1}_i,\sigma^{n-2s-1}_{i+1}$ be two adjacent simplices, the intrinsic link $L(c_i)$ and $L(c_{i+1})$ is cobordant.  If we  define the intrinsic link of $c_ic_{i+1}$ is $L(c_ic_{i+1})$, 
$$\partial L(c_ic_{i+1}) =L(c_i)-L(c_{i+1}).$$

In fact,  $L(c_{i+1})$ and $L(c_i)$ are homotopy equivalence,  because the allowable condition is same, there is map $\phi_i : \ L(c_{i+1})\to L(c_i)$ such that  $$\phi^*(\mathrm{IH}^{\bar m}_s(L(c_{i})))=\mathrm{IH}^{\bar m}_s(L(c_{i+1})).$$

Similar with conical singular case, we can use the same condition of \ref{l1} to choose the Lagrange subspace $\mathrm{H}_{lag}(L(c_i))$  of  $\mathrm{IH}^{\bar m}_s(L(c_i))$.  
Here the Lagrange subspace $\mathrm{H}_{lag}(L(c_i))$ is isomorphic to its annihilator $\mathrm{H}_{lag}(L(c_i))^{\perp}$ about the intersection form. 
$$\mathrm{H}_{lag}(L(c_i))\cong \mathrm{H}_{lag}(L(c_i))^{\perp}$$
We can require that:   
$$\phi^*(\mathrm{H}_{lag}(L(c_{i})))=\mathrm{H}_{lag}(L(c_{i+1})).$$

 Recall that s-dimensional cycle of $L(c_i)$ is $V(c_i)$. Similar with the conical case of 
\ref{eqv1}, we define $V_{lag}(c_i)$: $$V_{lag}(c_i)= \{h + \partial w \in V(c_i) | h\neq 0\in \mathrm{H}_{lag}(c_i) \ , and \ w \in W^{p_{2s+1}}_{s+1}(L(c_i))\}. $$ 
$$V_o(c_i)= \{ \partial w \in V(c_i) |\ w \in W^{p_{2s+1}}_{s+1}(L(c_i))\}. $$
The decomposition of $V(c_i)$ is 
$$V(c_i)=V_{lag}(c_i)\oplus V^{\perp}_{lag}(c_i)\oplus V_o(c_i)$$
 In fact, $\phi_i^*$ is the isomorphism from  $V_{lag}(c_{i})$ to $V_{lag}(c_{i+1})$ . The map $\phi^*$ here keep the intersection product of $\mathrm{IH}^{\bar m}_s(L(c_{i}))$. So we do not worry the Lagrange structure.
 
Let us give an example to chose the Lagrange structure on a special nontrivial product case. When $X_1$ is a closed line, the link $L(x)$ of every point $x\in X_1$ is a 2-torus $T^2$. That is a fibre bundle $\pi :E\to B=X_1$ where the fiber $F$ is $C(T^2)$. Choose $a$ and $b$ are two generator of $\mathrm{IH}^{\bar m}_1(T^2).$ 

For two points $c_1$ and $c_2$ on the line, the map of homology of link is defined as $\phi: \mathrm{IH}^{\bar m}_1(L(c_{1}))\to \mathrm{IH}^{\bar m}_1(L(c_{2}))$ define as $\phi:a \to a+b$. Let the figure \ref{fig1} represent the $E.$
After triangulation $T$, we use the $DECG$ and $ABHE$ to represent the cycle of torus. $BE$ represent a cycle $a$ and $AB$ is $b$. So it maps to $DH=a+b$ here. We just connect all the points to make $BEDH$ to be a simplex. Remember that we in fact use a subdivison and triangulation $T$ is finer than the graph. 
\begin{figure}\label{fig1}
\centering
	\includegraphics[scale=0.5]{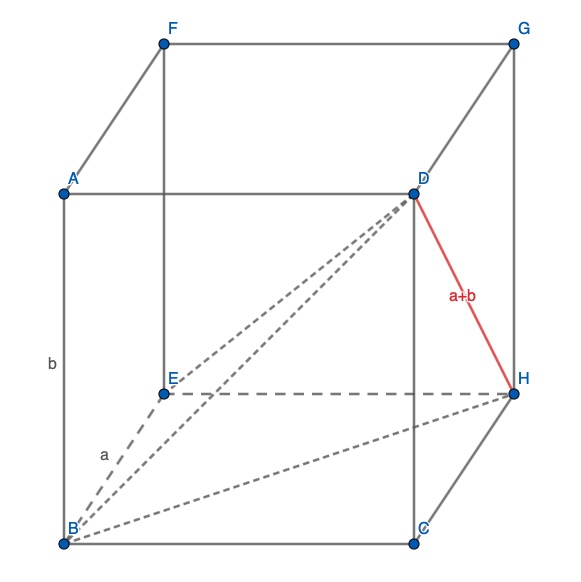}
	\caption{a non trivial product case}
\end{figure}

 We can define the compatible Lagrange space for all point in the segment $c_ic_{i+1}$ connected with $c_i$ and $c_{i+1}$.
\begin{definition}\label{cl}
In a cycle $S=\sum\limits_i c_ic_{i+1}$ , the compatible condition means there is a compatible Lagrange space $\mathrm{H}_{lag}(L(c_i))\in \mathrm{IH}^{\bar m}_s(L(c_i))$  for every $L(c_i)$ such that we can build the map connect for every segment $c_ic_j$:
$$\phi_i^*(V_{lag}(L(c_i))\cong V_{lag}(L(c_{i+1})).$$ 
In addition, for every $i$, we have $${\phi^*_{i-1}}^{-1}\circ {\phi^*_{i-2}}^{-1}.....\phi^{*}_{i+1}\circ \phi^*_{i}\ \text{is a isomorphism  of} \ V_{lag}(L(c_i) $$ 
\end{definition}
Suppose the compatible condition \ref{cl} holds for every circle in $\chi_{n-2s-1}$. For every two connected point $c_i$ and $c_j$, there is a map from $V_{lag}(c_i)$ to $V_{lag}(c_j)$. This means the compatible subspace $V_{lag}^s$ of intrinsic link exists for every point of $\chi_{n-2s-1}$. Let us consider the space $ U^{s+1}$,  we know $u^{s+1}_i$ is $\tau^{0}_i *v_i$, here $\tau^0_i$ is the intersection point in $\sigma^{n-2s-1}_i$. Because of the allowable condition of lower perversity, the intersection point should be barycenter of $\sigma^{n-2s-1}_i$. Correspond to conical case, we  define $U^{s+1}_{lag}$ to be spanned by $u^{s+1}_i$: 
\begin{equation} \label{u1}
   U^{s+1}_{lag}=\{ \sum\limits_i \lambda_ju^{s+1}_i | u^{s+1}_i=c_i*v_i, \ v\in V_{lag}(c_i), \  \lambda_j\in Q\}.
\end{equation}  
$${U^{s+1}_{lag}}^{\perp}=\{ \sum\limits_i \lambda_ju^{s+1}_i | u^{s+1}_i=c_i*v_i, \ v\in {V_{lag}(c_i)}^{\perp}, \  \lambda_j\in Q\}.$$    
 Clearly, we get the decomposition of $U^{s+1}$ via decomposition of intersection homology.  $U^{s+1}_{lag}$ is the space of $U^{s+1}$. 

Let us focus the segment $c_ic_j$ again, because $V(c_i)$ is the space of cycle,  we get  
$$\partial(V(c_i)*c_ic_j)=\{\alpha-\beta|\alpha\in V(c_i)\ ,\ and \ \beta\in V(c_j)\}$$
 Assume we define the subspace $V_{lag}(c_{i})$ and $V_{lag}(c_{i+1})$.
 Because we need to build the chain complex, this means the for every element $v$ of subspace of $V(L(c_ic_j))$,  $\partial v $ should be in $V(c_{i})$ and $V(c_{i+1})$. 

In this cycle $S$, because of compatible condition, we can define the $s$+$2$ dimensional simplex $\sum c_ic_{i+1}*v , \ v \in V_{lag}^{s}$ in the circle.  

So we define the subspace $V^{s+2}$ of $U^{s+2}$ by:
$$U_{lag}^{s+2}=\{\sum\limits_i \lambda_j u^{s+2}_i | u^{s+2}_i \in U^{s+2} ,\ \text{and} \ \ \partial (\sum\limits_i \lambda_j u^{s+2}_i)=\alpha-\beta\ , \ \alpha \ ,\ \beta \  \in U^{s+1}_{lag} \}$$

\begin{definition}
The decomposition of $U^j$ is defined via the mathematical induction, when $j=s+1$, $U^{s+1}_{lag}$ is defined by \ref{u1}.  For $j<n-s-3$, assume that $U_{lag}^{j-1}$ is defined. Let $U_{lag}^{j}$ be
$$U_{lag}^{j}=\{\sum\limits_i \lambda_ku^{j}_i | u^{j}_i\in U^{j} ,\ \text{and} \ \ \partial (\sum\limits_i \lambda_ku^{j}_i)=\alpha-\beta\ ,\ \alpha\ ,\ \beta \ \in U^{j-1}_{lag} ,\ \lambda_k \in Q\}$$
$${U_{lag}^{j}}^{\perp}=\{\sum\limits_i \lambda_ku^{j}_i | u^{j}_i\in U^{j} ,\ \text{and} \ \ \partial (\sum\limits_i \lambda_ku^{j}_i)=\alpha-\beta\ ,\ \alpha\ ,\ \beta \ \in {U^{j-1}_{lag}}^{\perp} ,\ \lambda_k \in Q\}$$
\end{definition}
Hence, we define every space $U_{lag}^j$ when $s<j<n-s+1$. 
Let us define 
$$U^j_o=\{\sum\limits_i \lambda_k \tau^{j-s-1} _i*v_i| v_i\in V_o(c_j) \ \lambda_k \in Q\}.$$
Then the new chain $W^{P_{2s+1}}_{j}[0](X)$ and $W^{P_{2s-1}}_{j}[0](X)$ for $s<j<n-s+1$ is: 
\begin{equation}
 W^{P_{2s+1}}_{j}[0](X)=W^{P_{2s+1}}_{j}(X)+U_{lag}^j +U^j_o. \label{cha1}
\end{equation}
\begin{equation}
W^{P_{2s-1}}_{j}[0](X)=W^{P_{2s-1}}_{j}(X)-{U_{lag}^j}^{\perp}.\label{cha2}
\end{equation}



If $j<s+1$ or $j>n-s$, we define :
$$W^{P_{2s-1}}_{j}[0](X)=W^{P_{2s-1}}_{j}(X), \ and \quad W^{P_{2s+1}}_{j}[0](X)=W^{P_{2s+1}}_{j}(X).$$  We can conclude next lemma. 

\begin{lemma} \label{l7}
$W^{P_{2s+1}}_{*}[0](X)$ and $W^{ P_{2s-1}}_{*}[0](X)$  are chain complex.  Moreover,   $W^{P_{2s+1}}_{*}[0](X)$ and $W^{ P_{2s-1}}_{*}[0](X)$ are chain equivalent respect to the inclusion map.
\begin{proof}
First, we check $W^{P_{2s+1}}_{j}[0](X)$ is a chain complex. This is equivalent to : 
$$\partial W^{P_{2s+1}}_{j}[0](X)\in W^{P_{2s+1}}_{j-1}[0](X).$$
For $j<s+1$ or $j>n-s$,  we do not change the chain. When  $s<j<n-s+1$,  
$$\partial W^{P_{2s+1}}_{j}[0](X) =\partial W^{P_{2s+1}}_{j}(X) + \partial U$$
That is the reason why we add the trivial homology class $U_o$ in the construction.   In the constructions, $W^{P_{2s+1}}_{*}[0](X)$ and $W^{ P_{2s-1}}_{*}[0](X)$  are same when $s<j<n-s+1$. When $j$ is other, they are still same due to the filtration. 
\end{proof}
\end{lemma} 
Because of $\partial W^{P_{2s+1}}_{s+1}[0](X)\in W^{P_{2s+1}}_{s}[0](X)$. For any subcomplex $Y$ of the $X$, we can let 
$$W^{P_{2s-1}}_{j}[0](Y)=W^{P_{2s-1}}_{j}(Y)\cap W^{P_{2s-1}}_{j}[0](X).$$ 
Then $\partial W^{P_{2s-1}}_{j}[0](Y)\in \partial W^{P_{2s-1}}_{j}(Y)\cap \partial W^{P_{2s-1}}_{j}[0](X).$
So the chain complex are well defined in the $X$. 

For the chain respect to other perversity $P_{2i+1}$ in \ref{p} if $i>s$ , let us define $W^{P_{2i+1}}_{j}[0](X)$ as the first modified chain :
 \begin{equation}
     W^{P_{2i+1}}_{j}[0](X)=W^{P_{2i+1}}_{j}[0](X)+U_{lag}^j +U^j_o.
 \end{equation}
When $i<s$
\begin{equation}
   W^{P_{2i-1}}_{j}[0](X)=W^{P_{2i-1}}_{j}[0](X)-{U^j_{lag}}^{\perp}. 
\end{equation}
We know $W^{P_{2i+1}}_{j}(X)=W^{\bar m}_j(X)$ and $W^{P_{1}}_{j}(X)=W^{\bar n}_j(X)$. 
Because of Witt condition for other codimension, $W^{P_{2i+1}}_{j}[0](X)$ is chain equivalent to $W^{P_{2s-1}}_{j}[0](X)$,  so in fact we build the interpolation chain $W^{\bar{m}}_{*}[0](X)$ between $W^{\bar{n}}_{*}(X)$ and $W^{\bar{m}}_{*}(X)$. However, that is not enough.

\begin{remark}
The intrinsic link of the vertex belong to the dimension $n+1$ is coboundary for intrinsic link of vertex belong to the dimension $n$. 
In the definition of Lagrange structure,  we need to keep the tranversality and 
local triviality. Locally, the neighbourhood of pseudomanifold is trivial product $R\times C(L)$. We expect $U^j$ keep tranversality. Here tranversality is that the preimage of fiber bundle is simplex.

\end{remark}
Next we need to consider the modified chain $W^{\bar m}_*[0]$ may break the Witt condition of other odd codimension. Let us compute the intersection homology of the link  about other odd stratum.

\begin{theorem}\label{abd}
If $k<s$,  on every point $x$ in the odd codimensional stratum $\chi_{n-2k-1}$,  the redefined homology with respect to $W^{m}_{*}[0](X)$ of intrinsic link  $L(\chi_{n-2k-1},x)$  is still trivial, 
 $$\mathrm{IH}^n_{k}[0](L(\chi_{n-2k-1},x),Q)= 0.$$
 \begin{proof}
  This is because when $n-k-1>n-s-1$, the difference in the redefined chains and original chains $W^m_*[0](X) - W^m_*(X)$ is not involved with the $k$-homology of the link $L(\chi_{n-2k-1},x)$.
 \end{proof}
\end{theorem}
Next we  use mathematics induction to construct the iterated structure. Let $i$ is integer from 0 to $[{n+1}/2]$ 
$$s=s_0 \quad \text{and} \quad s_i=s_0+i.$$ 
The $[i]$ changed chain and $[i]$ changed lower middle perversity intersection homology  is defined as  $W^{P_{2s_i+1}}_{*}[0](X)$ and $\mathrm{IH}^{P_{2s_i+1}}_{*}[i](X)$.   
For each $i$, if the $[i-1]$ redefined lower middle perversity intersection homology about the intrinsic link is trivial,   we do not changed the chain. $$\mathrm{IH}^n_{s_i}[i](L(\chi_{n-2s_i-1},x),Q)=0 .$$  we get : 
\begin{equation} 
   W^{P_{2s_i+1}}_{j}[i](X)=W^{P_{2s_i+1}}_{j}[i-1](X) 
 \end{equation}
and we know from \cite{X} $W^{P_{2s_i+1}}_{j}[i](X)$ is equivalent to $W^{P_{2s_i+1}}_{j}[i](X) 
$

If $[i-1]$ redefined lower middle perversity intersection homology about the intrinsic link is  not trivial, and we assume there is a compatible Lagrange structure. 
$$\mathrm{IH}^{P_{2si-1}}_{s_i}[i-1](L(\chi_{n-2s_i-1},x),Q)=\mathrm{H}_{lag_i} \oplus\mathrm{H}_{lag_i}^{\perp}$$ 
 This means we need to change $ W^{P_{2k+1}}_j[i](X)$ and $W^{P_{2k-1}}_j[i](X)$  because of the nontrivial homology. For $s_i+1<j<n-s_i$, we have 
\begin{equation}
   W^{P_{2s+1}}_{j}[i](X)=W^{P_{2s-1}}_{j}[i](X) +U^j[i]. 
   \end{equation}

Here, the existence of Lagrange structure for the homology of the link make it possible to decompose $U_j$ further. 
Next we use the formula \ref{cha1} and \ref{cha2} to define the changed chain about $ W^{P_{2s_i+1}}_{j}[i](X)$ and $W^{P_{2s_i-1}}_{j}[i](X)$ for $s_i+1<j<n-s_i$:
\begin{equation}
W^{P_{2s_i+1}}_{j}[i](X)=  W^{P_{2s_i+1}}_{j}[i-1](X)+U_{lag}^j[i] +U^j_o[i]. \label{cha1i}
\end{equation}
\begin{equation}
W^{P_{2s_i-1}}_{j}[i](X)=  W^{P_{2s_i-1}}_{j}[i-1](X)-{U_{lag}^j}^{\perp}[i].\label{cha2i}
\end{equation}
The previous theorem \ref{abd} make the $[i]$ changed chain does not affect the $[i-1]$ change.
Specially, let us define $W^{\bar m}_{j}[i](X)$ and $W^{\bar n}_{j}[i](X)$ as :
 $$W^{\bar m}_{j}[i](X)=  W^{\bar m}_{j}[i-1](X)+U_{lag}^j[i] +U^j_o[i].$$ 
$$W^{\bar n}_{j}[i](X)=W^{\bar n}_{j}[i-1](X)-{U^j_{lag}}^{\perp}[i].$$

Let us repeat the actions for all other odd codimensional stratum $n-2h-1$, then we get a sequence of $s_j$ represent for each actions. For every odd codimension $2s_j+1$, we modify based on the Lagrange structure in intersection homology of the link. Let the final interpolation chains be $\widetilde{W}^{\bar{n}}_{*}(X)$
and $\widetilde{W}^{\bar{m}}_{*}(X)$.

\begin{definition}\label{finalw}
Let $s_i$ be the largest possible number such that $2s_i+1\leq n$. 
Define the new chain as:
$$\widetilde{W}^{\bar{m}}_{j}(X) \coloneqq W^{\bar m}_{j}[i](X)$$  
$$\widetilde{W}^{\bar{n}}_{j}(X) \coloneqq W^{\bar n}_{j}[i](X)$$ 
\end{definition}

\begin{cond}\label{condition1}
The oriented non Witt space $X$ in this chapter assume to has compatible Lagrange structure for each odd codimensional stratum to construct $\widetilde{W}^{\bar{n}}_{*}(X)$ and
$\widetilde{W}^{\bar{m}}_{*}(X)$ on $X$.
\end{cond}
Similar, Let $\widetilde{W}_{\bar{n}}^{*}(X) =\mathrm{Hom}(\widetilde{W}^{\bar{n}}_{*},\mathbb{C}),$ and $\widetilde{W}_{\bar{m}}^{*} =\mathrm{Hom}(\widetilde{W}^{\bar{m}}_{*}(X),\mathbb{C})$. Because  the $\mathrm{Hom}(,\mathbb{C})$ functor are exact, then the $\widetilde{W}^{\bar{m}}_{*}(X)$ and $\widetilde{W}^{ \bar{n}}_{*}(X)$ are chain equivalent by the lemma \ref{l7}.  Here the chain map from $\widetilde{W}^{\bar{m}}_{*}(X)$ to $\widetilde{W}^{ \bar{n}}_{*}(X)$ is $g_*$. And because of the construction, the inclusion map are chain equivalence. 
\begin{lemma} \label{l34}
The new chains $\widetilde{W}^{\bar m}_{*}(X)$ and $\widetilde{W}^{\bar n}_{*}(X)$ are chain equivalent respect to the inclusion map $\iota$.

\end{lemma} 
\begin{remark}
The definition of the repeated interpolated chain is closely related with mezzoperversity for iterative Cheeger boundary conditions at each non-Witt stratum of even codimension of \cite{AAW}:
$$\widetilde{\mathcal{L}}=\left\{W\left(Y_{n-3}\right), W\left(Y_{n-5}\right), \ldots, W\left(Y_{\ell}\right)\right\}$$
\end{remark}

\subsubsection{Cone Formula}\label{cf}
In this subsection, we compute the cone $C(Z)=\tau*Z$ which is the join complex with $\tau$. Because of the requirement of diagonal approximation, we only need to consider $\tau\in\chi_{j}$, here $\mathrm{dim}(z)=j$ and $z$ is the subcomplex of $L(\tau,\chi_j)$. 
The regular cone formula for the intersection homology for perversity $\bar p$ is :
\begin{equation}
\mathrm{IH}^{\bar m}_{i}(C(Z))=
\begin{cases}
\mathrm{IH}^{\bar m}_{i}(Z) \quad  i<j-\bar m(j+1)  \\
0 \quad  otherwise.
\end{cases}\end{equation}
If  $j<2s_0$,  we know $\widetilde{W}^{\bar m}_j(X)=W^{\bar m}_{j}(X)$ when $j<s_0+1$, so the cone formula is unchanged related $\widetilde{W}^{ \bar n}_{i}(C(Z))$ and $\widetilde{W}^{ \bar{m}}_{i}(C(Z))$ : 
\begin{equation}
\widetilde{\mathrm{IH}^{\bar m}_{i}}(C(Z))=
\begin{cases}
\mathrm{IH}^{\bar m}_{i}(Z) \quad  i<j- \bar m(j+1)  \\
0 \quad  otherwise.
\end{cases}\label{c3}\end{equation}
The codim of $(\tau)$ with respect to $Z$ is $j+1$ , for the allowable condition of vertex $\tau$,  we can get 
$$i-(j+1)+ \bar{m} (j+1)<0 $$
If the allowable condition is satisfied, $\mathrm{IH}^{\bar m}_{i}(C(X))=W^{\bar m}_{i}(C(X)-\{v\}).$ Because the the chain $\widetilde{W^{\bar{m}}_{*}}(X)$ is an interpolation of $W^{\bar{m}}_{*}(X)$  and $W^{\bar{n}}_{*}(X)$, it is easy to prove that when $i>j-\bar{m} (j+1)$ we have 
\begin{equation} \label{kc1}
\widetilde{\mathrm{IH}}^{\bar m}_{i}(C(Z))=0.
\end{equation}
Because  $\bar m(j+1)= \bar n(j+1)$ for $j$ is odd.  So when $j=2s$, the only difference happen in dimension $s$:
$$\widetilde{\mathrm{IH}^{\bar m}_{s}}(C(Z))\neq {\mathrm{IH}^{\bar n}_{s}}(C(Z))$$
\begin{lemma}\label{cl1}
When $j=2s_k$ and $\mathrm{IH}^n_{s_k}[k](L(\chi_{n-2s_k-1},x),Q)=0$, we have: $\widetilde{\mathrm{IH}^{\bar m}_{s_k}}(C(Z))=0.$ When $j=2s_k$ and $\mathrm{IH}^n_{s_k}[k](L(\chi_{n-2s_k-1},x),Q)=\mathrm{H}_{lag_k}\oplus\mathrm{H}_{lag_k}^{\perp}$, then $\widetilde{\mathrm{IH}^{\bar m}_{s_k}}(C(Z))=\mathrm{H}_{lag_k}.$
\begin{proof}
When $i=s_k+1$, the construction of $\widetilde W^{\bar m}_{i}(C(Z))$ is chain equivalence to $W^{P_{2s_k+1}}_{i}[k](C(Z))$, we know $\mathrm{IH}^{P_{2s_k+1}}_{i}[k](C(Z))$ depends on the Lagrange structure. This construction means the cone formula changed on the dimension $s_k$.
\end{proof}
\end{lemma}

\begin{lemma}\label{cl2}
When  $j<2s_k$, the $\mathrm{IH}^n_{s_k}[k](C(Z))=\mathrm{IH}^n_{s_k}[k](Z)$. 
\begin{proof}
Let us see the allowable condition respect to the conical singular stratum. Because the allowable condition is unchanged we can conclude that the intersection homology is same as usual. 
\end{proof}
\end{lemma}
Similarity, Let Borel-Moore intersection homology (intersection homology with compact support )  be $BM[\widetilde{\mathrm{IH}^{\bar m}_{i}}]$. When $j<2s_0$, $BM[\widetilde{\mathrm{IH}^{\bar m}_{i}}]$ of the cone is that :
\begin{equation}
BM[\widetilde{\mathrm{IH}^{\bar m}_{i}}](C(Z))=
\begin{cases}
BM[\widetilde{\mathrm{IH}^{\bar m}_{i}}](Z) \quad  i\geq j-\bar m(j+1)\\
0 \quad  \text{otherwise}.
\end{cases}
\end{equation}
If $j=2s_0$, the construction of $BM[\widetilde W^{\bar m}_{j}](C(Z))$ is chain equivalence to $BM[W^{P_{2s_0+1}}_{j}[0]](C(Z))$ \\
This means the cone formula changed on the dimension $s_0$:
\begin{equation}
BM[\widetilde{\mathrm{IH}^{\bar m}_{i}}](C(Z))=
\begin{cases}
BM[\widetilde{\mathrm{IH}^{\bar m}_{i}}](Z) \quad  i>s_0\\
\mathrm{H}_{lag_0}(Z)\in BM[\mathrm{IH}^{\bar m}_{s0}](Z) \quad i=s_0 \\
0 \quad  \text{otherwise}.
\end{cases}
\label{c4}    
\end{equation}
When $j=2s_k$ and $\mathrm{IH}^n_{s_i}[i](L(\chi_{n-2s_i-1},x),Q)=0 $ , then we know it is similar with \ref{c3} although the chain $\widetilde{W^{\bar m}_j}(X)$ changed. 
\begin{equation}
BM[\widetilde{\mathrm{IH}^{\bar m}_{i}}](C(Z))=
\begin{cases}
BM[\widetilde{\mathrm{IH}^{\bar m}_{i}}](Z) \quad  i\geq j-m(j+1)  \\
0 \quad  otherwise.
\end{cases}\label{c5}\end{equation}
For $j=2s_k$ and $\mathrm{IH}^n_{s_i}[i](L(\chi_{n-2s_i-1},x),Q)\neq0 $,  the construction of $\widetilde{W^{\bar m}_{s_k}(C(Z))}$ means the cone formula changed on the dimension $s_k$:
\begin{equation}
BM[\widetilde{\mathrm{IH}^{\bar m}_{i}}](C(Z))=
\begin{cases}
BM[\widetilde{\mathrm{IH}^{\bar m}_{i}}](Z) \quad  i>s_k\\
\mathrm{H}_{lagk}\in BM[\mathrm{IH}^{\bar m}_{sk}](Z) \quad i=s_k \\
0 \quad  \text{otherwise}.
\end{cases}
\label{c6}    
\end{equation}
Hence, we in fact the cone formula build the locally duality with boundary in the cone . 

\subsubsection{Diagonal Approximation} In order to prove the Poincare duality with respect to $\widetilde W^{m}_{j}(X)$, we need to build the diagonal approximation for the new chain $\widetilde W^{m}_{j}(X)$. A diagonal approximation should be a natural chain homomorphism for the category of filtered simplicial complexes and placid simplicial maps
and $\bigtriangleup(X)=x\otimes x$ for any $(\bar 0,0)$-allowable simplex. $\bar{0}=0$ is the zero perversity.

Because of \cite{X}
, there exists an unique diagonal approximation map  up to homotopy
 for lower middle perversity and upper middle perversity :
$$\bigtriangleup : W^{\bar{0}}_{*}(X)\to W^{\bar{n}}_{*}(X) \otimes W^{\bar{m}}_{*}(X).$$ 

What we need is to construct the new diagonal approximation 
$\widetilde \bigtriangleup$ decompose $W^{\bar{0}}_{*}(X) $  to $\widetilde{W}^{\bar{m}}_{*}(X)\otimes \widetilde{W}^{ \bar{n}}_{*}(X).$ Like the conical singularity, the key of diagonal approximation is cone formula for $\widetilde W^{\bar m}_{j}(X)$ and $\widetilde{W}^{ \bar{n}}_{*}(X)$ we defined before. 


\begin{lemma}
There exists a diagonal approximation $\widetilde{\bigtriangleup}$ for the  chains complex $\widetilde{W}^{\bar{m}}_{j}(X)$ and $\widetilde{W}^{ \bar{n}}_{j}(X)$ in the sense of \cite{X}. 
$$\widetilde \bigtriangleup :W^{\bar{0}}_{*}(X)\to \widetilde{W}^{\bar{m}}_{*}(X)\otimes \widetilde{W}^{ \bar{n}}_{*}(X).$$
Moreover $\widetilde{\bigtriangleup}$ is unique up to chain homotopy.
\begin{proof}
It is the application of method of acyclic models.

Because the basis element of $W^0_{s+1}(X)$ is minimal and modeled over the cone, it is enough to consider the case $W^0_{s+1}(C(Z))$ for some cone $C(Z)$:
$$\bigtriangleup_j :W^{\bar{0}}_{j}(C(Z))\to\widetilde{W}^{\bar{m}}_{*}(C(Z)\otimes\widetilde{W}^{\bar{n}}_{*}(C(Z)).$$ 
Recall that we only modified  ${W}^{\bar{m}}_{j}[i](C(Z))$ and ${W}^{\bar{n}}_{j}[i](C(Z))$ when $s_i<j<n-s_i+1$. When the dimension $j$ is less than $s_0+1=s+1$, the diagonal approximation map $\widetilde{\bigtriangleup}_j$  is the original $\bigtriangleup_j$ which is defined in the article \cite{X}. 

Let us use the mathematical induction to deal with dimension from $s+1$ to $n$. 

For $s<j<n$, because $\bigtriangleup_{j}$ should be a chain morphism, this means $$\partial(\bigtriangleup_{j}(\partial \omega))=0.$$ So $\bigtriangleup_{j}(\partial \omega)$ is a cycle in $\widetilde{W}^{\bar{m}}_{*}(C(Z)\otimes\widetilde{W}^{\bar{n}}_{*}(C(Z))$.  Because of the lemma \ref{l5}, the trivial homology  means $\bigtriangleup_{j}(\partial \omega)$  is a boundary of $$\zeta\in (\widetilde{W}^{ \bar{m}}_{*}(C(Z)\otimes \widetilde{W}^{ \bar{n}}_{*}(C(Z))_{j+1}.$$ 
That is: $\bigtriangleup_{2s}(\partial \omega) =\partial \zeta.$ So we can define : $$\bigtriangleup_{j+1}(\omega)=\zeta.$$ 
This prove the existence of the modified diagonal approximation.

Next we prove the uniqueness up to homotopy. Let us assume there is another diagonal approximation $\widetilde\bigtriangleup' $. Then we can inductively construct the homotopy $h_i$ between $\widetilde\bigtriangleup $ and $\widetilde\bigtriangleup' $. Because of $\widetilde\bigtriangleup_0 = \widetilde\bigtriangleup'_0.$
then we define $h_0=0$. If we have defined $h_i$ for all $i<s$, we need to build $h_{i+1}$, then we get $\partial h_{i+1}=\Delta_{i+1}-\Delta_{i+1}^{\prime}-h_{i} \partial$. When $\xi$ is a basis of $W_{i+1}^{\overline{0}}(X)$,  $\partial h_{i+1}\xi =(\Delta_{i+1}-\Delta_{i+1}^{\prime}-h_{i} \partial) \xi$ is still a cycle. 
Then due to lemma \ref{l5}, we know $\partial h_{i+1}\xi$ should be a boundary of $\zeta$ in $\widetilde{W}_{*}^{\bar{m}}(X) \otimes \widetilde{W}_{*}^{\bar{n}}(X))_{i+2}$. Then it is natural to define $h_{i+1}(\omega) =\zeta$. We have define a chain homotopy ${h_j}$ between $\widetilde\bigtriangleup $ and $\widetilde\bigtriangleup' $.
\end{proof}
\end{lemma}

\begin{lemma}\label{l5}
 Let $Z\in X$ be a $i-$simplex. If the diagonal approximation map $\widetilde\bigtriangleup$ exist for $Z$: $$\widetilde\bigtriangleup :W^{\bar{0}}_{*}(Z)\to\widetilde{W}^{\bar{m}}_{*}(Z)\otimes\widetilde{W}^{\bar{n}}_{*}(Z).$$  
for the image of diagonal approximation,  we can get  $$\mathrm{H_{k}}(\widetilde{W}^{ \bar{m}}_{*}(C(Z)\otimes \widetilde{W}^{ \bar{n}}_{*}(C(Z))=0,$$ for $k\geq i$.

\begin{proof}
First we use the algebraic Kunneth formula
: $$\mathrm{H}_{k}(\widetilde{W}^{ \bar{m}}_{*}(C(Z)\otimes \widetilde{W}^{ \bar{n}}_{*}(C(Z))=\bigoplus\limits_{j}\widetilde{\mathrm{IH}}^{m}_{j}(C(Z)) \otimes\widetilde{\mathrm{IH}}^{n}_{k-j}(C(Z)),$$
Because of \ref{kc1} and cone formula \ref{c3}, when $j\!>\!i-\!\bar{m}(i+1)$ we get :
$$\widetilde{\mathrm{IH}}^{\bar m}_{j}(C(Z))=\widetilde{\mathrm{IH}}^{\bar n}_{j}(C(Z))=0$$
Hence $k>i$, we can find the 
$$\mathrm{H}_k(\widetilde{W}^{ \bar{n}}_{*}(C(Z)\otimes \widetilde{W}^{ \bar{n}}_{*}(C(Z))=0.$$

Next we need to prove $\mathrm{H}_i(\widetilde{W}^{ \bar{n}}_{*}(C(Z)\otimes \widetilde{W}^{ \bar{n}}_{*}(C(Z))=0.$ When dimension of $Z$ is $i<2s_0$, because of the cone formula \ref{c3}, it is clear the equation holds.  

When the dimension of $Z$ is $2s_0$ or more generally $2s_k$ , we need to use the new cone formula lemma:
\begin{equation}
\widetilde{\mathrm{IH}}^{\bar m}_{i}(C(Z))=
\begin{cases}
\widetilde{\mathrm{IH}}^{\bar m}_{i}(Z) \quad  i<s_k\\
\mathrm{H}_{lag_k}(Z)\in \mathrm{IH}^{\bar m}_{s_k}(Z) \quad i=s_k \\
0 \quad  \text{otherwise}.
\end{cases}
\end{equation}
If $H_{2s_k}(\widetilde{W}^{ \bar{n}}_{*}(C(Z)\otimes \widetilde{W}^{ \bar{n}}_{*}(C(Z))\neq 0$, from algebraic Kunneth formula we find the only possible nontrival generator is from $\widetilde{W}^{ \bar{m}}_{s_k}(C(Z)\otimes \widetilde{W}^{ \bar{n}}_{s_k}(C(Z))$.  Let us use $h_i\otimes h_j$ to represent the nonzero term. However, consider the 2$s_k$-dimensional diagonal approximation  $\widetilde\bigtriangleup_{2s_k}$ ,  there does not exist $h_i\otimes h_j$ in the image of $\bigtriangleup_{2s_k}$ because of the Lagrangian Structure of $\mathrm{IH}^{m}_{s_k}(Z)$.  So we get
$$\mathrm{H}_{2s_k}(\widetilde{W}^{ \bar{n}}_{*}(C(Z)\otimes \widetilde{W}^{ \bar{n}}_{*}(C(Z))=0.$$

For the dimension of $Z$ is odd or not equal to $2s_k$, because of the cone formula in section \ref{cf} it is similar to prove : $$\mathrm{H}_{i}(\widetilde{W}^{ \bar{n}}_{*}(C(Z)\otimes \widetilde{W}^{ \bar{n}}_{*}(C(Z))=0.$$ 
\end{proof}
\end{lemma}

\subsubsection{Poincare Duality}
Similar, we will prove the Poincare duality in the category of geometric module.
With help of diagonal approximation, we can define the cap product by that:
\begin{equation}
  \widetilde\cap:\widetilde{W}_{ \bar{n}}^{j}(X)\otimes {W}^{0}_{n}(X) \xrightarrow{i \otimes \widetilde\bigtriangleup} {W}_{ \bar{n}}^{j}(X)\otimes (\widetilde{W}^{\bar{n}}_{*}(X) \otimes \widetilde{W}^{\bar{m}}_{*}(X))\xrightarrow{\varepsilon\otimes1}{W}^{\bar{m}}_{n-j}(X). 
\end{equation}
here the $\varepsilon : \widetilde{W}^{\bar{n}}_{*}(X) \otimes \widetilde{W}_{\bar{n}}^*(X)\to \mathbb{C}$ is the evaluation map. Define the Poincare dual map $\mathbb{P}$ is cap product with fundamental class $[X]$:
$$\widetilde{\mathbb{P}}:=-\widetilde\cap[X]$$
We need the Mayer-Vietories sequence for new chain $\widetilde W_{*}^{\bar{m}}(X)$.
\begin{lemma}
For $X$ is a PL pseudomanifold. $Y_1$ and $Y_2$ are closed subpseudomanifold of $X$, and $X=Y_1 \cup Y_2$. Then the short exact sequence of modified chain holds:
$$0 \rightarrow \widetilde W_{*}^{\bar{m}}\left(Y_{1} \cap Y_{2}\right) \stackrel{i_{1} \oplus i_{2}}{\rightarrow} \widetilde W_{*}^{\bar{m}}\left(Y_{1}\right) \oplus \widetilde W_{*}^{\bar{m}}\left(Y_{2}\right) \stackrel{i_{1}-i_{2}}{\longrightarrow} \widetilde W_{*}^{\bar{m}}(X) \rightarrow 0$$
\end{lemma}

\begin{lemma} \label{pl}
Suppose X is an $n$-dimensional oriented non Witt space with the condition \ref{condition1}. The general Poincare duality map $\widetilde{\mathbb{P}}$\  from \  $\widetilde{W}_{\bar{m}}^{*}(X)$ to $\widetilde{W}^{ \bar{n}}_{n-*}(X)$ :
$$\widetilde{\mathbb{P}}: \widetilde{W}_{\bar{m}}^{i}(X) \to \widetilde{W}^{ \bar{n}}_{n-i}(X).$$
is geometrically controlled chain equivalence .
\begin{proof} 
The proof is similar with Lemma \ref{p1}. We still use the mathematical induction to prove it, first we assume the Poincare duality holds for the $i\leq k$ dimension.  It is enough to show Poincare duality hold for $k+1$ dimension pseudomanifold case . Because of Mayer-Vietories sequence, we can focus the star of a $j$ dimensional simplex $\sigma$ named $Y$ in $X$ which is enough to prove Poincare duality in geometrical control category. In fact, if $Z$ is the link of $\hat\sigma$ in the first barycentric subdivision, we have $\left.\widetilde W_{i}^{\bar{m}}(X)\right|_{\mathrm{St} \hat{\sigma}}=\widetilde W_{i}^{\bar{m}}(C(Z)).$ It is enough to prove the relative chain map $$\widetilde{\mathbb{P}}:\widetilde{W}_{\bar m}^i(Y,\partial Y)\to \widetilde{W}^{\bar n}_{k+1-i}(Y).$$ is chain equivalence. 

By the induction, the dual operator $\widetilde{\mathbb{P}}=-\cap[\partial Y]:$
$$ \widetilde W_{\bar{p}}^{i}(\partial Y) \rightarrow \widetilde W_{n-i}^{\bar{q}}(\partial Y)$$ is chain equivalence $\partial Y$ because of the dimension. Consider the diagonal approximation $\widetilde\bigtriangleup[\partial Y]$:
 
 $$\widetilde\bigtriangleup[\partial Y]=\sum_{|\omega|+|\theta|=j} \omega \otimes \theta+\partial(\sum_{|x|+|y|=j+1} a_{x y} x \otimes y)$$
 Here $\omega$ and $\theta$ are basis of $\widetilde{\mathrm{IH}}^{\bar m}_{*}(\partial Y)$ and $\widetilde{\mathrm{IH}}^{\bar n}_{*}(\partial Y)$ respectively. 
 Let $x$ and $y$ be basis of $\widetilde{W}^{\bar m}_{*}(\partial Y)$ and $\widetilde{W}^{\bar n}_{*}(\partial Y)$.
By assumption, $Y=\hat \sigma * \partial (Y)$ . In the case $j=2s_k$ and $\mathrm{IH}^n_{s_i}[i](L(\chi_{n-2s_i-1},x),Q)\neq0 $, we have the homology formula :
\begin{equation}
\widetilde{\mathrm{IH}}^{\bar m}_{i}(Y)=
\begin{cases}
\widetilde{\mathrm{IH}}^{\bar m}_{i}(\partial Y) \quad  i<s_k=j-\bar m(j+1)\\
\mathrm{H}_{lag_k}(Y)\in \mathrm{IH}^{\bar m}_{s_k}(\partial Y) \quad i=s_k=j-\bar m(j+1) \\
0 \quad  \text{otherwise}.
\end{cases}
\end{equation}
\begin{equation}
\widetilde{\mathrm{IH}}^{\bar n}_{i}(Y)=
\begin{cases}
\widetilde{\mathrm{IH}}^{\bar n}_{i}(\partial Y) \quad  i<s_k=j-\bar n(j+1)\\
\mathrm{H}_{lag_k}(Y)\in \mathrm{IH}^{\bar m}_{s_k}(\partial Y) \quad i=s_k=j-\bar n(j+1) \\
0 \quad  \text{otherwise}.
\end{cases}
\end{equation}
When $|\omega|+|\theta|=j=2s_k$, remember that there does not exist  $h_i\otimes h_j$ in the lemma \ref{l5}, just one of $\omega$ or $\theta$ is a boundary. 

In other cases, we have :
\begin{equation}
\widetilde{\mathrm{IH}}^{\bar m}_{i}(Y)=
\begin{cases}
\widetilde{\mathrm{IH}}^{\bar m}_{i}(\partial Y) \quad  i\leq j-\bar m(j+1)  \\
0 \quad  otherwise   .
\end{cases}\end{equation}
\begin{equation}
\widetilde{\mathrm{IH}}^{\bar n}_{i}(Y)=
\begin{cases}
\widetilde{\mathrm{IH}}^{\bar n}_{i}(\partial Y) \quad  i\leq j-\bar n(j+1)  \\
0 \quad  otherwise.
\end{cases}\end{equation}
When $|\omega|+|\theta|=j$,  only one of $\omega$ or $\theta$ is a boundary. So we conclude  $\omega \otimes \theta$ is a boundary. That means there exist  $\tilde\omega \otimes \tilde\theta $ such that
 $$\sum\limits_{|\omega|+|\theta|=j} \omega \otimes \theta=\partial (\sum\limits_{|\widetilde{\omega}|+|\widetilde{\theta}|=j+1} \tilde\omega \otimes \tilde\theta\ )$$

Moreover, we get $\widetilde\bigtriangleup[\partial Y]$ is the boundary of $\sum \widetilde{\omega} \otimes \widetilde{\theta}+\sum a_{x y} x \otimes y$.
Because we know $\partial \widetilde\bigtriangleup[Y]=\widetilde\bigtriangleup[\partial Y]$. Next we combine the lemma \ref{l5} to get :
$$ \widetilde\bigtriangleup[Y]=\sum_{|\tilde\omega|+|\tilde\theta|=j+1} \tilde\omega \otimes \tilde\theta+\sum_{|x|+|y|=j+1} a_{x y} x \otimes y $$
Then we can define the duality operator $\widetilde{\mathbb{P}}: \widetilde W_{\bar{p}}^{i}(Y,\partial Y) \rightarrow \widetilde W_{n-i}^{\bar{q}}(Y)$. 
Because of lemma  \ref{l5}, this map is the homology isomorphism. Because the quasi-isomorphism of free chain complex is chain equivalence, the dual map $\widetilde{\mathbb{P}}$ is chain equivalence. 

It is similar to prove Poincare duality in geometrical control category because the map in the star of simplex is geometrically controlled. 
\end{proof}
\end{lemma} 

Finally, we can get the result below. 
\begin{theorem} If $X$ is the n dimensional oriented pseudomanifold satisfy the condition \ref{condition1}, then Poincare dual map $\widetilde{\mathbb{P}}$ is geometrically controlled chain equivalence for the chain $\widetilde{W}^{ \bar{m}}_{*}$.  In other words, $X$ is a geometrically controlled Poincare pseudomanifold. In this case, we call it $L$ space.
\end{theorem}
The proof of next Lemma is totally same with the the proof of lemma \ref{pdp} because of natural property of diagonal approximation.
\begin{lemma}
 $\widetilde{\mathbb{P}}^*$ is chain homotopy to  $(-1)^{i(n-i)} \widetilde{\mathbb{P}}$ in the geometrically controlled category. The two maps
 $$\widetilde{\mathbb{P}}:\left(\widetilde{W}_{\bar{n}}^{i}(X), b^{*}\right) \stackrel{\widetilde{\cap}[X]}{\longrightarrow}\left(\widetilde{W}_{n-i}^{\bar{n}}(X), b\right)$$
 and 
 $$\widetilde{\mathbb{P}}^{\prime}:\left(\widetilde{W}_{\bar{n}}^{i}(X), b^{*}\right) \stackrel{\widetilde\cap[X]}{\longrightarrow}\left(\widetilde{W}_{n-i}^{\bar{m}}(X), b\right) \stackrel{\iota}{\rightarrow}\left(\widetilde{W}_{n-i}^{\bar{n}}(X), b\right)$$
are chain homotopic in the geometrically controlled category.
 \end{lemma}
Then $(\widetilde{W}^{ \bar{n}}_{*}(X),b,\widetilde{\mathbb{P}})$ is the geometrically controlled Poincare complex. Let $$T=\frac{1}{2}\left(\widetilde{\mathbb{P}}+(-1)^{(n-i) i} \widetilde{\mathbb{P}}^{*}\right).$$ And we define $\widetilde{E}^{\bar n}_*(X)$ to be Hilbert space of $\widetilde{W}^{ \bar{n}}_{*}(X)$ with canonical inner product determined by a natural basis of minimal elements. $(\widetilde{E}^{\bar n}_*(X),b,T)$ is a n-dimensional Hilbert- Poincare complex. 
\clearpage
\subsection{General Case}
In this section, we consider the $n$-dimensional oriented pseudomanifold $X$ with a collection of exceptional stratums $\{\chi_{n-2s_i-1}\}_i$, $i\in \mathbb{Z}$. We require $s_i>s_j$ when $i>j$. For each $i$, the Witt condition \ref{nwitt} fails for the the intrinsic link $L(\chi_{n-2s_i-1},x)$.
It means there are stratums that for the lower middle perversity intersection homology about the intrinsic link $\{L(\chi_{n-2s_i-1},x)\}$ is not trivial
$$\mathrm{IH}^m_{s_i}(L(\chi_{n-2s_i-1},x),Q)\neq 0.$$ 
Witt condition \ref{nwitt} holds for all other odd codimensional stratum $\{\chi_{n-2k_j-1}\}$  when $k_j\neq s_i $. Here the pseudomanifold  has trivial rational lower middle intersection homology of the intrinsic link :
 $$\mathrm{IH}^m_{k_j}(L(\chi_{n-2k_j-1},x),Q)= 0.$$

We still use the triangulation which is finer than the stratification. Because of Witt condition,  $W^{\bar m}_*(X)$ is not equivalent to $W^{\bar n}_*(X)$. We need to decompose the difference space between $W^{\bar m}_*(X)$ and $W^{\bar n}_*(X)$. 
$$W^{\bar{m}}_*(X) =W^{P_{2r+1}}_*(X) \subset W^{P_{2r-1}}_*(X) ...\subset W^{P_{1}}_*(X)=W^{\bar{n}}_*(X).$$ 
  Here we change from the highest codimension $s_1$. When we consider the $X-X_{2s_1+1}$, it is still a oriented pseudomanifold where the lower middle perversity intersection homology about the intrinsic link $\{L(\chi_{n-2s_i-1},x)\}$ is not trivial . It is just the situation in section \ref{ssnw}.  

If there is a compatible Lagrange structure on the link  $L(\chi_{n-2s_2-1})$, we can construct the chain $\widetilde W^{m}_j(X-X_{n-2s_1-1})$ based on \ref{finalw}. Let us redefine that 
$ W^{m}_*(X)[1] = \widetilde W^{m}_*(X-X_{n-2s_1-1})$ and $ W^{n}_*(X)[1] = \widetilde W^{n}_*(X-X_{n-2s_1-1})$ for the chain in the stratum of $x \in X_{n-2s_1-1}$. We know $W^{m}_*(X)[1]$ and $ W^{n}_*(X)[1]$ are two equivalent chain and it does not affect the Witt condition of higher codimensional stratum. Hence we need to consider the next odd codimensional stratum $s_2$ where the non Witt conditions fails. We can keep construct  $\widetilde W^{m}_*(X)[2]$ based on the Lagrange structure of $L(\chi_{n-2s_2-1})$, all of this is similar with section \ref{compat}. If all odd codimensional stratum where the Witt condition fails has compatible Lagrange structure, the final interpolation chain is defined to be $\widetilde W^{\bar m}_*(X)$. With similar step we can build the diagonal approximation based on new chains. There is no difference between previous section and general situation. Finally, we can get the result below.
\begin{theorem} If $X$ is the n dimensional oriented pseudomanifold has enough compatible Lagrange structure to construct $\widetilde W^{\bar m}_*(X)$, then Poincare dual map with fundamental class $[X]$ is geometrically controlled chain equivalence for the chain $\widetilde{W}^{ \bar{m}}_{*}(X)$. In other words, $X$ is a geometrically controlled Poincare pseudomanifold. 

\end{theorem}

Moreover,$(\widetilde{W}^{ \bar{n}}_{*}(X),b,\widetilde{\mathbb{P}})$ is a geometrically controlled Poincare complex. We can define the Hilbert Poincare complex after completion. K homology of signature is defined with the similar method. 

\begin{remark}
In \cite{A}, Albin, Leichtnam, Mazzeo and Piazza define the smoothly stratified space with depth $k$ in mathematical induction. Let the Smooth manifold be depth $0$ and assume space with depth $K$ has been defined. Then the space with depth $k+1$. The self dual mezzoperversity should be connected with compatible Lagrange structure for the general Cheeger space . 
\end{remark}

\bibliographystyle{amsplain}
\bibliography{main}
\end{document}